%% file: darmon26.tex
\documentclass[11pt]{article} 

\usepackage[T1]{fontenc}
\usepackage[utf8]{inputenc} 

\usepackage{geometry}
\usepackage[parfill]{parskip} 

\input{packages.tex}
\input{commands.tex}

\input{theoremlayout.tex}

\sectionfont{\normalfont}
\subsectionfont{\normalfont\itshape}
\subsubsectionfont{\normalfont\itshape}

\title{\large Diagonal restriction of Eisenstein series and Kudla-Millson theta lift}
\author{\normalsize Romain Branchereau}

\begin{document}
\maketitle
\abstract{ We consider the Kudla-Millson theta series associated to a quadratic space of signature $(N,N)$. By combining a `see-saw' argument with the Siegel-Weil formula, we show that its (regularized) integral along a torus attached to a totally real field of degree $N$ is the diagonal restriction of an Eisenstein series. It allows us to express the Fourier coefficients of the diagonal restriction as intersection numbers, which generalizes one of the results of Darmon-Pozzi-Vonk \cite{DPV} to totally real fields. 
}
{
  \tableofcontents
}

\section{Introduction}The starting point of this paper is a result of Darmon-Pozzi-Vonk, that relates the diagonal restriction of an Eisenstein series to intersection numbers of geodesics. Let $\psi$ be a finite order odd Hecke character on the narrow class group $\Cl(F)^+$ of a real quadratic field $F$ of discriminant $d_F$. To such a character one can associate an Eisenstein series $E(\tau_1,\tau_2,\psi)$ which is a Hilbert modular form of parallel weight one and level $\SL_2(\Ocal)$. Its diagonal restriction $E(\tau,\tau,\psi)$ is a modular form of weight two and level $\SL_2(\Z)$. As such, it is zero, since this is the only such form. Instead one can look at the $p$-stabilization $E^{(p)}(\tau_1,\tau_2,\psi)$ for some prime $p$. The diagonal restriction $E^{(p)}(\tau,\tau,\psi)$ is now a form of weight two and level $\Gamma_0(p)$. Moreover, it is non-zero when $p$ is split.

Suppose that $p$ is a split prime and let $Y_0(p)$ be the modular curve $\Gamma_0(p)\backslash \HH$. To a fractional ideal $\afrak$ in $F$ and a square root $r$ of $d_F$ modulo $p$ one can associate a closed geodesic $\Qcalb_{\afrak,r}$ in $Y_0(p)$. Let $\Qcalb(\psi)$ be the $1$-cycle defined by
\begin{align}
\Qcalb(\psi) \coloneqq \sum_{[\afrak] \in \Cl(F)^+} \psi(\afrak)(\Qcalb_{\afrak,r}+\Qcalb_{\afrak,-r}) \in \Zcal_1(Y_0(p)),
\end{align}
and let $\Qcalb(0,\8)$ be the image in $Y_0(p)$ of the geodesic joining $0$ to $\8$. The following identity is proved in \cite[Theorem.~A]{DPV}
\begin{align} \label{DPVresult}
    E^{(p)}\left (\tau,\tau,\psi \right ) = L^{(p)}(\psi,0)- 2\sum_{n=1}^\8 \bigl < \Qcalb(0,\8),T_n\Qcalb(\psi)\bigr >_{Y_0(p)} e^{2i\pi n \tau} 
\end{align}
where $L^{(p)}(\psi,0)=(1-\psi(\p))(1-\psi({\p^\sigma }))L(\psi,0)$ and $T_n$ is a Hecke operator defined by double cosets. In {\em loc. cit.} the equality \eqref{DPVresult} is proved by computing the intersection numbers and comparing them with the Fourier coefficients of $E^{(p)}(\tau,\tau,\psi)$.

The work of Kudla and Millson provides a very natural framework to prove \eqref{DPVresult}. The goal of this paper is to recover the result of Darmon-Pozzi-Vonk by Kudla-Millson theory and to generalize it to totally real fields. For this we define a torus $C \otimes \psi$ such that the integral of the Kudla-Millson theta series over this torus gives the diagonal restriction of the Eisenstein series.

\paragraph{Kudla-Millson theta lift.} The Kudla-Millson theta lift is a lift from the cohomology of locally symmetric spaces attached to orthogonal or unitary groups to modular forms. In our case we want to work with orthogonal groups of certain quadratic spaces of signature $(N,N)$.

Let $F$ be a totally real field of degree $N$ with ring of integers $\Ocal$. Let $X^0 = F^2$ be the $2$-dimensional quadratic $F$-space with the quadratic form $Q^0(\xbf,\ybf)=xy'+x'y$ where $\xbf=(x,x')$ and $\ybf=(y,y')$ are vectors in $F^2$. At a place $v$ of $\Q$ let $F_{\Q_v} \coloneqq F_\Q \otimes \Q_v$, where we write $F_\Q$ instead of $F$ to emphasize that we view $F$ as a $\Q$-algebra. We fix a $\Q$-basis of $F$ so that we identify $F_\Q$ with $\Q^N$. Let $X=F^2_\Q$ be the $2N$-dimensional quadratic $\Q$ space with the quadratic form $Q \coloneqq \Tr_{F \mid \Q} Q^0$. Let $H(\Q) \coloneqq \SO(F^2_\Q)$ be its orthogonal group. The real vector space $F^2_\R$ is of signature $(N,N)$. Let $H(\R)^+$ be the connected component of the identity of its real points $H(\R)=\SO(F^2_\R)$. Let $\varphi_\fin$ in $\Scal(F_{\A_\fin}^2)$ be a finite Schwartz function, which is $K_\fin$-invariant by the Weil representation for some open compact subgroup $K_\fin$ of $H(\A_\fin)$. We define the double coset space
\begin{align} \label{adelicspace}
M_K \coloneqq H(\Q) \backslash H(\A)/K,
\end{align}
where $K=K_\8K_\fin$, and $K_\8$ is a maximal connected compact subgroup of $ H(\R)^+$ that is isomorphic to $\SO(N)\times \SO(N)$. The space $M_K$ is a disjoint union of locally symmetric spaces of dimension $N^2$, in general non-compact. If $M_K$ is non compact let $\overline{M_K}$ be any compactification, for example the Borel-Serre compactification. On $M_K$ there is a natural family of {\em special cycles} of codimension $N$
\begin{align}
    C_n(\varphi_\fin) \in \Zcal_{N^2-N}(\overline{M_K},\partial \overline{M_K};\R)
\end{align} indexed by positive rational numbers $n$. In \cite{km2,km3,KMIHES}, Kudla and Millson used the work of Weil \cite{weil} on theta series to construct a closed differential form
\[\Theta_{KM}(\tau,\varphi_\fin) \in \Omega^N(M_K)\]
where $\tau$ is in $\HH$. Since it is closed it represents a cohomology class in $H^N(M_K;\R)$ and can be paired with a (compact) cycle $C$ in $\Zcal_N(M_K;\R)$. The function
\begin{align} \label{integral}
\tau & \longmapsto \int_C \Theta_{KM}(\tau,\varphi_\fin)
\end{align}
is a holomorphic modular form of weight $N$. Moreover, we have the Fourier decomposition
\begin{align}
 \int_C \Theta_{KM}(\tau,\varphi_\fin) = \kappa \sum_{n \in \Q_{\geq 0}} \left ( \int_C \Theta_n(\varphi_\fin) \right ) e^{2i n \pi \tau},
\end{align}
where $\Theta_n(\varphi_\fin)$ is a Poincaré dual in $ \Omega^N(M_K)$ to the cycle $C_n(\varphi_\fin)$, when $n$ is positive. Here $\kappa$ is a constant equal to $2$ if $K_\fin \cap H(\Q)^+$ contains $-1$, and $1$ otherwise.
Since $C$ is compact, the Fourier coefficients for positive $n$ are equal to the topological intersection numbers $\Bigl < C_n(\varphi_\fin),C \Bigr >_{M_K}$ on $M_K$, see Section \ref{poincaredual}.

\paragraph{Limitations of the work of Kudla-Millson.} The integral of $\Theta_{KM}(\tau,\varphi_\fin)$ along a compact cycle is a modular form of weight $N$ whose Fourier coefficients are intersection numbers. If we replace the compact cycle by a non-compact cycle $C$, then the results of Kudla and Millson do not apply and the following problem may arise. First the integral \eqref{integral} might diverge. Secondly, even if the integral does converge the resulting function in $\tau$ can be non-holomorphic. This is for example what happens in \cite{funke-compositio}: the Kudla-Millson theta series associated to a quadratic space of signature $(1,2)$ is integrated over the whole modular curve and the resulting integral converges to a non-holomorphic modular form. Thirdly, it is not immediate that the Fourier coefficients $\int_C \Theta_n(\varphi_\fin)$ can be interpreted as an intersection number between the two non-compact cycles $C_n(\varphi_\fin)$ and $C$, since a priori such a number is not well-defined; see Remark \ref{rmkproblem}.

\paragraph{A non-compact cycle $C\otimes \psi$.} Let $\psi \colon F^\times \backslash \A_F^\times \longrightarrow \C^\times$ be a totally odd unitary Hecke character of finite order. To simplify the notation let us suppose in the introduction that $\psi$ is unramified. Hence the finite part of this Hecke character can be viewed as a character on the narrow class group, as in the beginning of the introduction.

Let $\SO(F^2) \subset \GL_2(F)$ be the orthogonal group of the quadratic space $X_F^0=F^2$ with the quadratic form defined above. The group $F^\times$ can be identified with $\SO(F^2)$ by
\begin{align}
F^\times& \longrightarrow \SO(F^2) \nonumber \\
t & \longmapsto \begin{pmatrix}
t & 0 \\ 0 & t^{-1}
\end{pmatrix}.
\end{align}
On the other hand the group $\SO(F^2)$ can naturally be embedded in $\SO(F^2_\Q) \subset \GL_{2N}(\Q)$ by restriction of scalars. Composing the two and passing to the adèles gives an embedding 
\begin{align} h \colon \A_F^\times \hooklongrightarrow \SO(F^2_\A)\subset \GL_{2N}(\A), \end{align} where $F^2_\A=F^2_\Q \otimes_\Q \A$. For $K_\fin$ large enough we have that $h(\widehat{\Ocal}^\times)$ is contained $K_\fin$. The embedding $h$ then induces an immersion of
$M_\Ocal$ in $M_K$ where 
\begin{align}
    M_\Ocal \coloneqq F^\times \backslash \A_F^\times/\{\pm 1\}^{N-1} \times \widehat{\Ocal}^\times.
\end{align} 
The space $M_\Ocal$ is not connected. There is a bijection between classes in the narrow class group $\Cl(F)^+$ and connected components of $M_{\Ocal}$. More precisely it can be written as a disjoint union
\begin{align}
    M_{\Ocal} = \bigsqcup_{[\afrak] \in \Cl(F)^+} \Gamma \backslash \R_{>0}^N
\end{align}
where $\Gamma \coloneqq \Ocal^{\times,+}$ are the totally positive units in $\Ocal$. The connected components are $N$-dimensional tori. The image by the immersion in $M_K$ of the connected component of $M_\Ocal$ corresponding to the ideal class $[\afrak]$ in $\Cl(F)^+$ is a torus $C_\afrak$. We twist it by the Hecke character $\psi$ to obtain a relative cycle
\begin{align}
    C\otimes \psi \coloneqq \sum_{[\afrak] \in \Cl(F)^+} \psi(\afrak) C_\afrak \in \Zcal_N(\overline{M_K},\partial \overline{M_K};\R),
\end{align}
where we view $\psi$ as a character on the ray class group.

In our case, the integral of $\Theta_{KM}(\tau,\varphi_\fin)$ over $C\otimes \psi$ does not converge, but can be regularized by adding a parameter $t^s$ with $s$ a complex number and isolating some singular terms as it is done in \cite{kudlasingular}. Moreover, although the cycles $C\otimes \psi$ and $C_n(\varphi_\fin)$ are both non-compact we show that the intersection number $\bigl < C_n(\varphi_\fin), C \otimes \psi \bigr >$ is well-defined, see \eqref{noncompactinter}.

\paragraph{Eisenstein series and diagonal restriction.} For a Hecke character $\psi$ as above and a finite Schwartz function $\phi_\fin$ in $\Scal(F^2_{\A,\fin})$ we can define an Eisenstein series
\begin{align}
    E(\tau_1, \dots,\tau_N,\phi_\fin,\psi) = \restr{E(\tau_1, \dots,\tau_N,\phi_\fin,\psi,s)}{s=0}
\end{align}
by analytic continuation, where $(\tau_1, \dots,\tau_N)$ is a point in $\HH^N$; see \eqref{unfold22}. It is a Hilbert modular form of parallel weight one. If we take $\tau = \tau_1= \cdots = \tau_N$, then the diagonal restriction $E(\tau,\dots,\tau, \phi_\fin,\psi)$ is a modular form of weight $N$.

Let $l_1$ and $l_2$ be the isotropic lines spanned by the isotropic vectors $\ebf_1\coloneqq \transp{(1,0)}$ and $\ebf_2 \coloneqq \transp{(0,1)}$ in $F^2$. For a Schwartz function on $F^2_{\A_\fin}$ let $\varphi_1(x) \coloneqq \varphi_\fin \begin{pmatrix}
 x \\ 0
\end{pmatrix}$ and $\varphi_2(y) \coloneqq \varphi_\fin \begin{pmatrix}
 0 \\ y
\end{pmatrix}$ be the Schwartz functions in $\Scal(F_{\A,\fin})$ obtained by restricting $\varphi_\fin$ to $l_1$ and $l_2$.

\begin{thm*}(Theorem \ref{fouriercoeffs}) Let $\varphi_\fin$ be any Schwartz function in $\Scal(F^2_{\A_{\fin}})$ such that $\varphi_1$ or $\varphi_2$ vanishes. The diagonal restriction has the Fourier expansion
\[E(\tau,\dots,\tau,\phi_\fin,\psi)=\zeta_\fin(\varphi_1,\psi^{-1},0)+\zeta_\fin(\varphi_2,\psi,0) +(-1)^N 2^{N-1}\kappa\sum_{n \in \Q_{>0}} \bigl < C_n(\varphi_\fin), C \otimes \psi \bigr >_{M_K} e^{2i \pi n \tau},\]
where $\phi_\fin =\Fcal\varphi_\fin$ is a partial Fourier transform of $\varphi_\fin$ and $\zeta_\fin$ is a zeta integral (see \eqref{zetadef}).
\end{thm*}

The equality is proved by showing that both sides of the equality are equal to the regularized integral 
\[2^{N-1}\int_{C\otimes \psi} \Theta_{KM}(\tau,\varphi_\fin).\]

Note that the intersection takes place in $M_K$ and is between the $N$-dimensional cycle $C \otimes \psi$ of dimension $N$ and the cycle $C_n(\varphi_\fin)$ of codimension $N$ (and dimension $N^2-N$).
The constant term consists of the values at $s=0$ of the analytic continuation of two zeta functions converging on two disjoint half planes $\re(s)>1$ and $\re(s)<-1$. Hence the condition of vanishing of $\varphi_1$ or $\varphi_2$ is used to make sure that one of the two terms is zero and that an analytic continuation exists.

\paragraph{A seesaw.} The theta series $\Theta_{KM}(\tau,\varphi_\fin)$ is a theta kernel for the dual pair $\SL_2(\Q) \times \SO(F^2_\Q)$. The theorem can be summarized by the following seesaw
\[
\begin{tikzcd}
\SO(F^2_\Q) \arrow[dash, dr] & \SL_2(F) \\
F^\times \arrow[dash]{u} \arrow[dash, ru] & \SL_2(\Q) \arrow[dash]{u},
\end{tikzcd}
\]
that relates the theta kernel $\Theta_{KM}(\tau,\varphi_\fin)$ to a theta kernel for  $F^\times \times \SL_2(F)$. The vertical arrow on the left correspond to the cycle $C\otimes \psi$, that is obtained from the group embedding of $F^\times$ in $\SO(F^2_\Q)$. The vertical arrow on the right correspond to the diagonal restriction of the Hilbert-Eisenstein series of parallel weight one.

\subsection*{Acknowledgements} This project was done during my thesis. I thank my advisors Nicolas Bergeron and Luis Garcia for suggesting me this topic and for their support. I also thank Henri Darmon and Jan Vonk for helpful discussions about this paper. I was funded from the European Union’s Horizon 2020 research and innovation programme under the Marie Skłodowska-Curie grant agreement N$\textsuperscript{\underline{\scriptsize o}}$754362 \includegraphics[width=6mm]{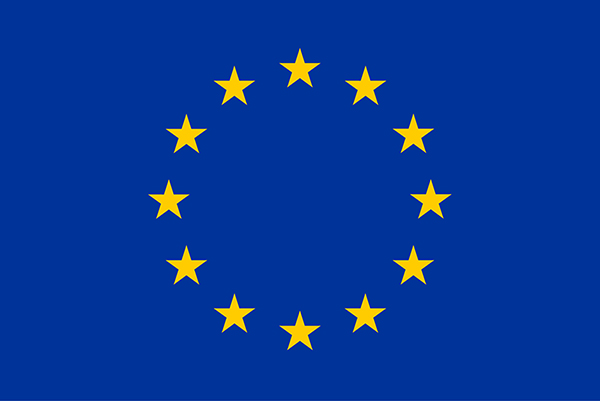}.

\section{Preliminaries} \label{notgen}

\paragraph{Intersection numbers.} \label{poincaredual}
Let $M$ be an arbitrary oriented Riemannian manifold of dimension $m$ with a Riemannian metric $q_z$. Let $o(T_zM)$ in $\bigwedge^m T_zM$ denote the orientation at a point $z$ in $M$.  Let $M_1$ and $M_2$ be two immersed oriented submanifolds of complementary dimensions $m_1$ and $m_2$, with orientations that we denote by $o(T_zM_1)$ and $o(T_zM_2)$. Let $z$ be an intersection point in $M_1 \cap M_2$. We say that the intersection is transversal if $T_zM=T_zM_1 \oplus T_zM_2$, equivalently if $T_zM_1 \cap T_zM_2$ is empty. In this case $ o(T_z M_1) \wedge o(T_zM_2)$ is a multiple of $o(T_zM)$ and we define
\begin{align}
    I_z(M_1,M_2) \coloneqq \sgn \left ( \frac{ o(T_z M_1) \wedge o(T_zM_2) }{o(T_zM)} \right ).
\end{align}
Furthermore, if $M_1 \cap M_2$ is finite we define the {\em topological intersection number} on $M$
\begin{align}
    \bigl < M_1,M_2 \bigr >_M \coloneqq \sum_{z \in M_1 \cap M_2 } I_z(M_1,M_2).
\end{align}
In particular, if one of the two submanifolds is compact, then the intersection is finite.

Viewing the Riemannian metric $q_z$ as a bilinear pairing on $T_zM$ we get a bilinear pairing
\begin{align}
\sideset{}{^k} \bigwedge T_zM \otimes \sideset{}{^k} \bigwedge T_zM & \longrightarrow \R  \nonumber \\
v_1 \wedge \cdots \wedge v_k \otimes w_1 \wedge \cdots \wedge w_k & \longmapsto \det q_z(v_i,w_j)_{ij}
\end{align}
for every positive integer $k$. Let $N_zM_1=(T_zM_1)^\perp$ be the normal space, which is the orthogonal complement of $T_zM_1$ with respect to $q_z$. We fix an orientation $o(N_zM_1)$ of $N_zM_1$ via the rule $    o(T_zM_1) \wedge o(N_zM_1) = o(T_zM).$ One can check that the intersection number is also given by 
\begin{align}
  I_z(M_1,M_2) = \sgn q_z\bigl (o(T_zM_2),o(N_zM_1) \bigr ).
\end{align}

\paragraph{Poincaré duals.}

Let $M$ be an $m$-dimensional real oriented manifold without boundary. The integration map yields a non-degenerate pairing \cite[Theorem.~5.11]{botu}
\begin{align} \cohom^{q}(M) \otimes_\R  \cohom_c^{m-q}(M)& \longrightarrow \R \nonumber \\
[\omega]\otimes [\eta] & \longmapsto \int_M \omega \wedge \eta,    
\end{align} where $\cohom_c(M)$ denotes the cohomology of compactly supported forms on $M$. This yields an isomorphism between $\cohom_c^{m-q}(M)^\vee$ and $\cohom^{q}(M)$, where $\cohom_c^{m-q}(M)^\vee$ is the dual space of $\cohom_c^{m-q}(M)$. An immersed submanifold $C$ of codimension $q$ in $M$ defines a linear functional on $\cohom_c^{m-q}(M)$ by
\begin{align}
 \omega \longmapsto \int_C \omega.   
\end{align} By the isomorphism between $\cohom_c^{m-q}(M)^\vee$ and $\cohom^{q}(M)$ there is a unique cohomology class $\PD(C)$ in $H^q(M)$ representing this functional {\em i.e. } 
\begin{align}
\int_M   \omega \wedge \PD(C) = \int_C \omega,
\end{align}
 for every closed $(m-q)$-form $\omega$ on $M$. We call $\PD(C)$ {\em the Poincaré dual class to $C$}, and any closed differential form representing the cohomology class $\PD(C)$ {\em a Poincaré dual form to $C$}. If $C$ is compact, then $\PD(C)$ is compactly supported.
Let $C$ and $C'$ be immersed submanifolds and let one of them be compact. Then
\begin{align}
    \int_{C}\PD(C')=\int_{M_K} \PD(C') \wedge \PD(C) = \bigl < C',C \bigr >_M,
\end{align}
where the right hand side is the topological intersection numbers.

\paragraph{Number fields.} \label{numberfields}  Let $F=\Q(\lambda)$ be a totally real field of degree $N$, let $f_\lambda$ be the minimal polynomial of $\lambda$, and $\Ocal$ be ring of integers. We will write $F_\Q$ when we want to emphasize that we view $F$ as a $\Q$-vector space. Let $\epsilon_1, \dots, \epsilon_N$ be an oriented $\Z$-basis of $\Ocal$. With this basis we identify $F_\Q$ as $\Q^N$ as column vectors. For $\mu$ in $ F^\times$ let $\gamma(\mu)$ in $\GL_N(\Q)$ be the matrix representing the left multiplication on $F_\Q$ by $\mu$. The image of this representation is the centralizer of $\gamma(\lambda)$ in $\GL_N(\Q)$. It is a maximal $\R$-split torus in $\GL_N(\Q)$. The map sending $(x,y)$ in $F \times F$ to $\Tr_{F/\Q}(xy)$ defines a non-degenerate pairing on $F$. The dual of $\Ocal$ is the inverse different ideal $\dfrak^{-1}$. Let $\sigma_1,\dots, \sigma_N$ be the $N$ real embeddings of $F$. We order them such that the matrix
\begin{align} \label{ginf}
 g_\8 \coloneqq \begin{pmatrix} \sigma_1(\epsilon_1) & \hdots & \sigma_1(\epsilon_N) \\[2.5ex] \vdots & & \vdots \\[2.5ex] \sigma_N(\epsilon_1) & \hdots & \sigma_N(\epsilon_N) \end{pmatrix} \in \GL_N(\R)
\end{align}
has positive determinant. We set 
\begin{align} \label{Adef}
A\coloneqq \transp{g}_\8 g_\8  =(\Tr_{F/\Q}(\epsilon_i\epsilon_j))_{ij} \in \GL_N(\Q), 
\end{align} whose determinant $d_F \coloneqq \det(A)$ is the discriminant of $F$.

\paragraph{Local fields.} \label{totreal}
Let $w$ denote a place of $F$ and $v$ a place of $\Q$. Over $\Q_v$ the (separable) minimal polynomial $f_\lambda$ of $\lambda$ splits as product $f_\lambda = \prod_{w \mid v} f_w
$ of irreducible factors $f_w$ in $\Q_v[x]$, where each factor correspond to a place $w$ dividing $v$; see \cite[Proposition.~8.2, p.~163]{neukirch}. For every $w$ let $\lambda_w$ in $\overline{\Q}_v$ be a root of $f_w$ and let $F_w \coloneqq \Q_v(\lambda_w)$. Let $\vert x \vert_w \coloneqq \sqrt[n_w]{\vert \N(x) \vert_v}$ be the valuation on $F_w$ that extends $\vert \cdot \vert_v$, where $n_w \coloneqq \deg(f_w)$ is the degree of $F_w$ over $\Q_v$. If $w$ is finite we denote by $\Ocal_w$ its ring of integers, $\m_w$ its unique maximal ideal, $\pi_w$ a uniformizer and $q_w$ the cardinality of the residue field $\Ocal_w/\m_w$. If $\p$ is the prime ideal corresponding to a finite place $w$ we will occasionally replace the index $w$ by $\p$ and write $F_\p,\Ocal_\p, \vert \cdot \vert_\p,\dots$.

If $v=p$ is finite, let  $f_\lambda \equiv \prod_{w \mid v} P_w^{e_w} \pmod{p}$ be the factorization modulo $p$ of $f_\lambda$ into irreducibles $P_w$. If $\Ocal=\Z[\lambda]$, then $p$ splits as
\begin{align} \label{factorisation}
p\Ocal=\prod_{w \mid v} \p_w^{e_w},    
\end{align}
where $\p_w=P_w(\lambda)\Z+p \Z$; see \cite[Proposition.~8.3, p.~48]{neukirch}. 

\paragraph{Restriction of scalars.}
For any $v$ we define the $N$-dimensional $\Q_v$-algebra $F_v \coloneqq \prod_{w \mid v} F_w$, and for a finite place $v=p$ we set $\Ocal_p \coloneqq \prod_{w \mid p} \Ocal_w$. At every place we have an embedding of  $F$ in $F_w $ given by sending $\lambda$ to $\lambda_w$. In particular when $v$ is totally split we have $N$ such embeddings. Let us denote by $F_{\Q_v}$ the $\Q_v$-algebra $F_{\Q} \otimes \Q_v$. We have an isomorphism of $\Q_v$-algebras:
\begin{align} \label{varsigmaiso}
    \varsigma_v \colon F_{\Q_v} & \longrightarrow F_v \nonumber \\
    \alpha \otimes t & \longmapsto (\alpha_w t)_{w\mid v},
\end{align}
where $\alpha_w$ is the image of $\alpha$ in $F_w$. It induces an isomorphism between $\Ocal \otimes_\Z \Z_p$ and $\Ocal_p$ at the finite places; see \cite{wielonsky} for a proof.
 We fix a $\Z_p$-basis of $\Ocal_w$ for every $w$ diving $p$. This gives a $\Z_p$-basis of $\Ocal_p$, which in turn identifies $F_p$ with $\Q_p^N$. On the other hand, at the archimedean place we have $F_\8=\R^N$ since $F$ is totally real. For $t$ in $F_v^\times$ we let $g(t)$ in $\GL(F_{\Q_v}) \simeq \GL_N(\Q_v)$ be the matrix representing the left multiplication on $F_v$ by $t$. Then there is a matrix $g_v$ in $\Hom(F_{\Q_v},F_v) \simeq \GL_N(\Q_v)$ such that for every $t$ in $F$ we have $\gamma(\varsigma^{-1}t)=g_v^{-1} g(t) g_v$ in $\GL_N(\Q_v)$. However we will usually identify $F_{ \Q_v}$ with $F_v$ and simply write $\gamma(t)$ instead of $\gamma(\varsigma^{-1}t)$. In the case where $v=\8$, then $F_\8=\R^N$. If $t=(t_1, \dots,t_N)$ in $F_\8$, then $g(t)$ is the matrix with $(t_1, \dots,t_N)$ on the diagonal and we can take $g_\8$ as in \eqref{ginf}. Let $F_\A \coloneqq F_\Q \otimes \A$. The map $\varsigma \coloneqq \otimes_{v } \varsigma_v$ induces an isomorphism $\varsigma$ between $F_\A$ and $\A_F,$ which also identifies $\widehat{\Z} \otimes \Ocal$ with $\widehat{\Ocal}$.

\paragraph{Haar measures.} \label{haar}
Let $F$ be a totally real number field and $F_w$ the completion at a place $w$. We identify $F_w$ with its dual $F^\vee_w$ by the pairing $\Tr_{F_w / \Q_v}(xy)$. We fix the additive character $\chi_{\Q_v}$ on $\Q_v$ defined by
\begin{align}
\chi_{\Q_v}(x) \coloneqq  \begin{cases} e^{2i\pi x} & \textrm{if} \; v=\8 \\
e^{2i\pi \{x\}_p} & \textrm{if} \; v=p,
\end{cases}
\end{align}
where $\{x\}_p$ is the fractional part of $x$ in $\Q_p$. Let $k$ be an étale algebra extension of $\Q_v$. We can extend this character to a character $\chi_k$ on $k$ by setting $\chi_k \coloneqq \chi_{\Q_v} \circ \Tr_{k \mid \Q_v}$.

There is a unique choice of Haar measure $dx_w$ on $F_w$ which is self dual with respect to $\chi_{F_w}$. This is the Haar measure normalized such that $\vol(\dfrak_w^{-1})=1$, where 
\[\dfrak_w^{-1} = \left \{ x \in F_w \vert \, \Tr_{F_w / \Q_v}(xy) \in \Z_v \, \textrm{for all} \, y \in \Ocal_{w} \right \}\]
is the inverse different ideal. This is equivalent to $\vol(\Ocal_{w})=\sqrt{\N(\dfrak_w)}$. We obtain a measure $dx\coloneqq \prod_w dx_w$ on $\A_F$ which satisfies $\vol(\widehat{\Ocal})=d_F^{-\frac{1}{2}}$.
For a Schwartz function $ \Phi_w$ in $\Scal(F_w)$ let $\Phi^\vee_w $ in $\Scal(F_w)$ be the Fourier transform defined by
\begin{align} \label{Fourier1}
    \Phi_w^\vee (t)\coloneqq \int_{\A_F} \Phi_w(u)\chi_w(-tu)du,
\end{align}
where $du$ is the self-dual measure. On $\A_F^\times$ we define the Haar measure
\begin{align}
    dt_w^\times \coloneqq m_w \frac{dt_w}{\vert t_w \vert_w},
\end{align}
where $m_w=1$ if $w$ is archimedean. After fixing a character $\psi$ of conductor $\fffrak$ the constant $m_w$ for $w$ non-archimedean will be choosen such that $\vol^\times(\widehat{U}(\fffrak))=1$.

\paragraph{Class groups.}

For an ideal $\fffrak$ of $\Ocal$ let $\Ical_\fffrak(F)$ be the set of fractional ideals in $F$ coprime to $\fffrak$ and $P_\fffrak(F)$ the set of principal ideals in $\Ical_\fffrak(F)$. Let $P_\fffrak(F)^+$ be the set of principal ideals whose generator is positive at all places.  We define the class group $\Cl(F) \coloneqq \Ical_\Ocal(F)/P_\Ocal(F)$ and the narrow class group $\Cl(F)^+\coloneqq \Ical_\Ocal(F)/P_\Ocal(F)^+$. We also define the ray class groups $\Cl_\fffrak(F) \coloneqq \Ical_\fffrak(F)/P_\fffrak(F)$ and $\Cl_\fffrak(F)^+ \coloneqq \Ical_\fffrak(F)/P_\fffrak(F)^+$. The map sending an ideal $\afrak$ to the finite adèle $(\pi_w^{w(\afrak)})_w$ induces an isomorphism
\begin{align}
    \Cl_\fffrak(F)^+ \simeq F^{\times,+} \backslash \A^\times_{F,\fin}/\widehat{U}(\fffrak).
\end{align}

\paragraph{Hecke characters and $L$-functions.} \label{Heckelfunc}
 Let $\psi \colon F^\times \backslash \A_F^\times \longrightarrow \C^\times$ be a finite order Hecke character. We can write $\psi=\psi_\8 \otimes \psi_\fin$ where the finite part $\psi_\fin$ is a character on $F^\times \backslash \A_{F,\fin}^\times$. The conductor of $\psi$ is the largest ideal $\fffrak$ of $\Ocal$ such that $\psi_{\fin}$ vanishes on
 \begin{align}
     \widehat{U}(\fffrak) = \prod_{w}U_w(\fffrak)
 \end{align}
where
 \begin{align}
U_w(\fffrak) \coloneqq  \begin{cases} \Ocal_w^\times & \textrm{if} \; w(\fffrak)=0 \\
1+\p^{w(\fffrak)} & \textrm{if} \; w(\fffrak)>0.
\end{cases}
\end{align}
We say that $\psi$ is \textit{unramified} if the conductor is $\Ocal$ {\em i.e.} $\psi_\fin$ is trivial on $\widehat{\Ocal}^\times$. At an archimedean places $\psi_\sigma$ is given by
 \begin{align}
    \psi_\sigma(t) = \vert t_\sigma \vert ^{a_\sigma} \left ( \frac{t_\sigma}{\vert t_\sigma \vert } \right )^{b_\sigma} 
 \end{align}
 where $a_\sigma$ can be any complex number and $b_\sigma$ is $0$ or $1$.
 We say that $\psi$ is \textit{totally odd} if $b_\sigma=1$ at all places. 
\noindent The finite part of a Hecke character of conductor $\fffrak$ can also be seen as a character
\begin{align}
    \psi \colon \Ical(\fffrak) & \longrightarrow \C^\times \nonumber \\
    \afrak & \longmapsto \psi(1,\afrak)
\end{align}
such that at principal ideals in $\Pcal_\fffrak^+(F)$ we have 
\begin{align}
    \psi((\alpha))=\prod_{\sigma} \vert \sigma(\alpha) \vert ^{a_\sigma} \left ( \frac{\sigma(\alpha)}{\vert \sigma(\alpha) \vert } \right )^{b_\sigma},
\end{align}
where the product is taken over all real places and pairs of complex embeddings. To such a character we can attach an $L$-function 
\begin{align}
    L^\fffrak(\psi,s) \coloneqq \sum_{{\substack{0 \neq \afrak \subseteq \Ocal \\ \gcd(\fffrak,\afrak)=1}}} \frac{\psi(\afrak)}{\N(\afrak)^s}.
\end{align}
It converges for $\re(s)>1$, admits a meromorphic continuation and a functional equation. Moreover, it has the Euler product 
\begin{align}
    L^\fffrak(\psi,s)=\prod_{\substack{0 \neq \p \subset \Ocal \; \textrm{prime} \\ \gcd(\p,\fffrak)=1}} L_{\p}(\psi,s)
\end{align}
where the local factors are $L_{\p}(\psi,s) \coloneqq (1-\psi(\p)\N(\p)^{-s})^{-1}$. For more on $L$-functions we refer to \cite{iwasawaL} and \cite{bump}.

\paragraph{Zeta functions.} \label{zetafunctions} For a Schwartz function $\Phi_w$ in $\Scal(F_w)$ we define the {\it local Zeta integrals}
\begin{align}
\zeta_w(\Phi_w,\psi,s) \coloneqq \int_{F_w^\times} \Phi_w(t_w)\psi_w(t_w)\lvert t_w \rvert^{s} dt_w^\times,
\end{align}
which converge for $\re(s)>0$. For a Schwartz function $\Phi$ in $\Scal(\A_F)$ we define the {\it Zeta integral}
\begin{align} \label{zetadef}
    \zeta(\Phi,\psi,s) \coloneqq \int_{\A_F^\times} \Phi(t)\psi(t)\lvert x \rvert^{s} dt^\times.
\end{align}
This function converges absolutely for $\re(s)>1$ by \cite[Proposition.~3.1.4 (iii)]{bump}. For a decomposable Schwartz function $\Phi=\otimes \Phi_w$ in $\Scal(\A_F)$ and for $\re(s)>1$ the decomposition
\begin{align}
\zeta(\Phi,\psi,s)=\prod_{w} \zeta_w(\Phi_w,\psi,s)
\end{align}
is valid, where the product is taken over the places of $F$. We will also denote by $\zeta_\fin$ the finite part of the zeta integral:
 \begin{align} \label{finitezeta}
\zeta_\fin(\Phi_\fin,\psi,s) \coloneqq \int_{\A^\times_{F,\fin}} \Phi_{\fin}(t_\fin)\psi_{\fin}(t_\fin)\lvert t_\fin \rvert^{s} dt_\fin^\times.
 \end{align}
The zeta function $\zeta(\Phi,\psi,s) $ admits a meromorphic continuation to the complex plane and a functional equation $\zeta(\Phi,\psi,s)=\zeta(\Phi^{\vee},\psi^{-1},1-s),$
where $\Phi^\vee$ in $\Scal(\A_F)$ is the Fourier transform of $\Phi$ defined in \eqref{Fourier1}. In our case, at an archimedean place $\sigma$, the Schwartz function will be $\Phi_\sigma(t)=e^{-\pi t^2}t$ and the character $\psi_\sigma(t)=\sgn(t)$. With this choice we then have 
\begin{align} \label{zeta1}
 \zeta_\8(\Phi,\psi,s)& = \Lambda(s),
\end{align}
where $\Lambda(s) \coloneqq \Gamma\left (\frac{1+s}{2} \right)^N\pi^{-\frac{N(1+s)}{2}}$.
At the places where $\Phi_w=\id_{\Ocal_w}$ and the character $\psi$ is unramified we have 
\begin{align} \label{zeta2}
 \zeta_w(\Phi,\psi,s)& =  \vol^\times (\Ocal^\times_w) L_{w}(\psi,s)
\end{align}
where $L_{w}(\psi,s)$ are the local factors as in Subsection \ref{Heckelfunc} for $\p$ corresponding to $w$.


\paragraph{Weil representation.} \label{locweil} Let $L$ be a number field (we will consider $L=\Q$ or $L=F$ totally real) and $(X_L,Q)$ be a $2m$-dimensional quadratic space over $L$. Let $(W_L,B)$ be the symplectic space $W_L=X_L\oplus X_L$ with the symplectic form 
\begin{align}
B\left ( \begin{pmatrix}
\xbf \\ \xbf'
\end{pmatrix}, \begin{pmatrix}
\ybf \\ \ybf'
\end{pmatrix} \right )=Q(\xbf,\ybf')-Q(\xbf',\ybf).
\end{align}
Let $k$ be the completion of $L$ at some place, let $X_k\coloneqq X \otimes_{L} k$ and $W_k \coloneqq W_L \otimes_L k$. Let $\chi_{k}$ be the additive character of $k$ defined in Subsection \ref{haar}.  Let
\begin{align}
    g = \begin{pmatrix}
     g_1 & g_2 \\ g_3 & g_4
    \end{pmatrix} \in \Sp(W_k) \subset \GL(W_k)
\end{align}
with $g_1$ and $g_4$ in $\GL(X_k)$,  and $g_2$ and $g_3$ in $\End(X_k)$. For $\varphi$ a Schwartz function in $\Scal(X_k)$ we define the {\em local (projective) Weil representation} \cite[p.~40]{moeglin}
\begin{align}
    \omega \colon \Sp(W_k) \longrightarrow \Ucal(\Scal(X_k))
\end{align}
by the operator
\begin{align} \label{weil rep}
    \omega(g)\varphi(\xbf) \coloneqq \int_{X_k/\ker(g_3)} \varphi( \adjast{g_1} \xbf+ \adjast{g_3}\ybf) \chi_k \left ( \frac{1}{2} Q (  \adjast{g_1}\xbf,\adjast{g_2} \xbf )+ Q( \adjast{g_2} \xbf, \adjast{g_3} \ybf ) + \frac{1}{2}Q(\adjast{g_3} \ybf , \adjast{g_4}\ybf  ) \right )d\mu(\ybf),
\end{align}
where $\adjast{g}$ is defined by $Q(g\xbf,\ybf)=Q(\xbf,\adjast{g}\ybf)$. The Haar measure $d\mu(\ybf)$ on $ X_k/ \ker(g_3)$ is the unique one that makes this operator unitary with respect to the $L^2$-norm on $\Scal(X_k)$. We can then extend the local Weil representation to a global Weil representation $\Scal(X_{\A_L})$ of $\Sp(W_{\A_L})$. 

The operator \eqref{weil rep} only defines a projective representation of $\Sp(W_k)$ in the sense that $\omega(g_1 g_2)=c(g_1,g_2) \omega(g_1) \omega(g_2)$ for some complex cocycle $c(g_1,g_2)$ of norm $1$; see \cite{rao}. However, on some subgroups of $\Sp(W_L)$ the cocycle $c(g_1,g_2)$ becomes trivial and we have a true representation of those subgroups.

\paragraph{Dual pairs and theta kernels.} \label{dualpairtheta}

A {\em dual pair} $(G,H)$ is a pair of subgroups $G$ and $H$ of $\Sp(W_L)$ that centralize each other in $\Sp(W_L)$. This allows us to view $G \times H$ as a subgroup of $\Sp(W_L)$ and to restrict the Weil representation to this product. If $\varphi$ is a Schwartz function in $\Scal(X_{\A_L})$ we define the theta kernel
\[\Theta(g,h,\varphi) \coloneqq \sum_{\xbf \in X(L)} \omega_{os}(g,h)\varphi(\xbf)\]
for $(g,h)$ in $G(\A_L) \times H(\A_L)$. A {\em seesaw pair} is a pair of dual pairs $(G^0,H^0)$ and $(G,H)$ such that $H^0$ is contained in $H$ and $G$ is contained in $G^0$. Such a pair is represented by a diagram
\[
\begin{tikzcd}
H \arrow[dash, dr] & G^0 \\
H^0 \arrow[dash]{u} \arrow[dash, ru] & G. \arrow[dash]{u}
\end{tikzcd}
\]
The seesaw principle relies on the observation that the theta kernel associated to these two dual pairs pairs agree on the common subgroup $H^0 \times G$. We will be interested in a seesaw that is obtained by restriction of scalars. It is the example $(2.19)$ given by Kudla in \cite{sspair} and will be explained in Section \ref{sectionclass}.

There are two types of dual pairs that will be involved. Let $k$ be one of the completions of $L$ and $f_1, \dots, f_{2m}$ a $k$-basis of $X_k$. Then we have an isomorphism between $W_k$ and $k^{4m}$ with the symplectic form
\begin{align}
    \begin{pmatrix}
     0 & A(Q) \\
     -A(Q) & 0
    \end{pmatrix}
\end{align}
where $A(Q)$ in $M_{2m}(k)$ is the symmetric matrix $(A(Q))_{ij}=Q(f_i,f_j)$. In this basis the symplectic group is given by
\begin{align}
\Sp(W_k) = \left \{ g \in \GL_{4m}(k) \left \vert \transp{g} \begin{pmatrix}  & A(Q)\\ -A(Q) & \end{pmatrix}g=\begin{pmatrix}  & A(Q) \\ -A(Q) & \end{pmatrix} \right . \right \}.
\end{align}

\begin{enumerate}
\item We view $X=k^{2m}$ with the quadratic form $A(Q)$. We can embedd $\GL_{2m}(k)$ in $\Sp(W_k)$ by
\begin{align} \label{embgl2gl1}
    \GL_{2m}(k) & \hooklongrightarrow \Sp(W_k) \nonumber \\
    g & \longmapsto \begin{pmatrix} g & \\ & \adjast{g}^{-1} \end{pmatrix},
\end{align}
where $\adjast{g} \coloneqq A(Q)^{-1}\transp{g}A(Q)$ is as above. The centralizer of $\GL_{2m}(k)$ is the center $\GL_1(k)$ of $\GL_{2m}(k)$, and we obtain a dual pair $\GL_{2m} \times \GL_1$ which we call the {\em linear pair}. The restriction of the operator \eqref{weil rep} to $\GL_{2m} \times \GL_1$ defines the Weil representation on $\Scal(X_k)$, given by
\begin{align}
\omega_l(g,t)\varphi(\xbf)=\lvert \det(gt) \rvert^\frac{1}{2}\varphi \left (\adjast{g}t\xbf  \right )
\end{align}
for $(g,t)$ in $\GL_{2m}(k) \times \GL_1(k)$.
\item  We can restrict the embedding \eqref{embgl2gl1} to the orthogonal group of $X_k$
\begin{align}
    \SO(X_k) & \hooklongrightarrow \Sp(W_k) \nonumber \\
    h & \longmapsto \begin{pmatrix}
     h & 0 \\ 0 & h
    \end{pmatrix},
\end{align}
since $\adjast{h}^{-1}=h$. Its centralizer is isomorphic to $\SL_2(k)$, embedded as follows
\begin{align}
    \SL_2(k) & \hooklongrightarrow \Sp(W_k) \nonumber \\
    \begin{pmatrix}
    a & b \\ c & d
    \end{pmatrix} & \longmapsto \begin{pmatrix}
    a &&b& \\
    &a&&b \\
    c&&d& \\
    &c&&d
    \end{pmatrix}.
\end{align}
Hence the pair $\SO(X_k) \times \SL_2(k)$ is the second example of dual pair, that we call \label{dualpairpage} the {\em orthosymplectic pair}. The restriction of the operators defined in \eqref{weil rep} to $\SL_2(k) \times \SO(X_k)$ yields the Weil representation on $\Scal(X_k)$: 
\begin{align} 
\omega_{os}(g,h) \varphi(\xbf) & = \omega_{os}(g,1) \varphi(h^{-1}\xbf), \\
\omega_{os} \left ( \begin{pmatrix}
1 & b \\ & 1
\end{pmatrix},1 \right ) \varphi (\xbf) & \coloneqq \chi_{k} \left (\frac{bQ(\xbf,\xbf) }{2}\right ) \varphi (\xbf),\\
\omega_{os} \left ( \begin{pmatrix}
a &  \\ & a^{-1}
\end{pmatrix},1 \right ) \varphi(\xbf) & \coloneqq \vert a \vert^m \varphi(a\xbf),\\
\omega_{os}\left ( S,1 \right )\varphi (\xbf) & \coloneqq \int_{X_k}\varphi(\ybf) \chi_{k} \left ( -Q(\xbf,\ybf) \right )d\mu(\ybf) \label{weilrepslgspin},
\end{align}
where $S=\begin{pmatrix}
0 & -1 \\ 1 & 0
\end{pmatrix}$ and $(g,h)$ in $\SL_2(k) \times \SO(X_k)$.

\end{enumerate}

Note that the restriction of the projective representation \eqref{weil rep} defines a true representation on $\SL_2(k) \times \SO(X_k)$ since we assumed that $X_L$ is even dimensional. For more on the Weil representation and the theta correspondence we refer to \cite{weil}, \cite{rao}, \cite{lionvergne} and \cite{moeglin}.

\paragraph{Hilbert modular forms and Eisenstein series.}

Let $\Gamma$ be a finite index subgroup of $\SL_2(\Ocal)$. Let $(\tau_1, \dots, \tau_N)$ be a point in $\HH^N\coloneqq \HH \times \cdots \times \HH$. The group $\Gamma$ acts on $\HH^N $ by $\gamma(\tau_1, \dots, \tau_N) =(\gamma_1\tau_N, \dots, \gamma_N\tau_N),
$ where $\gamma_i$ is the image in $\SL_2(\R)$ of $\gamma$ by the real embedding $\sigma_i$. A holomorphic function $f$ on $\HH^N$ is a Hilbert modular form of weight $(k_1, \dots,k_N)$ and of level $\Gamma$ if 
\begin{enumerate}
    \item $f(\gamma_1\tau_N, \dots, \gamma_N\tau_N)=(c_1 \tau_1+d_1)^{k_1} \cdots (c_N \tau_N+d_N)^{k_N} f(\tau_1, \dots, \tau_N)$, where $\gamma_i=\begin{psmallmatrix} a_i & b_i \\ c_i & d_i \end{psmallmatrix}$,
    \item if $F=\Q$ then $f$ is holomorphic at the cusps.
\end{enumerate}
We say that $f$ has parallel weight $k$ if all the $k_i$'s are equal to $k$. If $f$ is a Hilbert modular form then its diagonal restriction $f(\tau, \dots, \tau)$ is a modular form of weight $k_1+ \cdots + k_N$ and of level $\Gamma \cap \SL_2(\Z)$.

Let  $\phi_\sigma$ be the Schwartz function in $\Scal(F_\sigma^2)$ defined by
\[\phi_\sigma (\xbf_\sigma)= (-i)e^{-\pi \lvert z_\sigma \rvert^2} z_\sigma\]
where  $\xbf_\sigma=(x_\sigma,x'_\sigma)$ and $z_\sigma \coloneqq x_\sigma+ix'_\sigma$. Let $\phi=\phi_\8 \otimes \phi_\fin$ in $\Scal(F^2_{\A})$ where $\phi_\8  = \prod_{\sigma} \phi_\sigma$ is in $\Scal(F_\8^2)$ and $\phi_\fin$ is an arbitrary finite Schwartz function in $\Scal(F^2_{\A_\fin})$. For a totally odd finite order Hecke character $\psi$ we define the function 
\begin{align}
Z\left (g, \xbf , \phi,\psi,s \right ) \coloneqq\int_{\A_F^\times} \omega_l(g,t)\phi(\xbf) \psi(t) \lvert t \rvert^{s} dt^\times,
\end{align}
which converges absolutely for $\re(s)>0$. For $\tauud =(\tau_1, \dots, \tau_N)$ in $\HH^N$ let 
\begin{align}
    g_{\tauud}=\begin{pmatrix}
 \sqrt{\underline{v}} & \nicefrac{\underline{u}}{\sqrt{\underline{v}}} \\ 0 & \nicefrac{1}{\sqrt{\underline{v}}}
\end{pmatrix}\in \SL_2(F_\8) \simeq \SL_2(\R)^N
\end{align} where $\tauud =\underline{u}+i\underline{v}$. It is the matrix in $\SL_2(\R)^N$ that sends $(i, \dots,i)$ to $(\tau_1, \dots, \tau_N)$ by Möbius transformation. We define the Eisenstein series
\begin{align} \label{unfold22}
 E(\tau_1, \dots, \tau_N,\phi_\fin,\psi,s) & \coloneqq \sum_{\gamma \in P(F) \backslash \GL_2(F) }  (v_1 \cdots v_N)^{-\frac{1}{2}} Z\left (g_{\tauud},\gamma_0^{-1}\xbf_0, \phi,\psi,s \right ),
\end{align}
where $P(F)$ is the stabilizer of $\xbf_0 \coloneqq \transp{(1,0)}$ and $\gamma_0$ in $\GL_2(F)$ is one of the following representatives for $P(F) \backslash \GL_2(F)$
\begin{align}
    \begin{pmatrix} 1 & 0 \\ 0 & 1 \end{pmatrix} \; \textrm{or} \;\begin{pmatrix}
     \lambda & 1 \\ 1 & 0
    \end{pmatrix} \; \textrm{with} \; \lambda \in F^\times.
\end{align}
The Eisenstein series converges termwise absolutely for $\re(s)>N-1$, see \cite[Lemme p.~106]{wielonsky}. It admits an analytic continuation to the whole plane by  \cite[Proposition.~9]{wielonsky} and we set
\begin{align}
E(\tau_1, \dots, \tau_N,\phi_\fin,\psi) \coloneqq \restr{E(\tau_1, \dots, \tau_N,\phi_\fin,\psi,s)}{s=0}. 
\end{align}Since the Schwartz function decomposes as $\phi= \phi_\8 \otimes \phi_\fin$ we can decompose the integral 
\begin{align}
   Z\left (g_{\tauud},\begin{pmatrix} m \\ n \end{pmatrix}, \phi,\psi,s \right )=Z_\8\left (g_{\tauud},\begin{pmatrix} m \\ n \end{pmatrix}, \phi,\psi,s \right ) Z_\fin\left (1,\begin{pmatrix} m \\ n \end{pmatrix}, \phi,\psi,s \right ).
\end{align} 

\begin{lem} \label{archimedeanZ} We have
\[Z_\8\left (g_{\tauud},\begin{pmatrix} m \\ n \end{pmatrix}, \phi,\psi,s \right )=\frac{\Lambda(1+s)}{\left ( i\pi \right )^N} \frac{(v_1 \cdots v_N)^{\frac{1}{2}+s}}{\N(m-n\tauud)\vert \N(m- n \tauud) \vert^{s}}.\]
\end{lem}
\begin{proof}
Let $\tau_\sigma=u_\sigma+iv_\sigma$ and $g_{\tau_\sigma}=\begin{pmatrix}
 \sqrt{v_\sigma} & \nicefrac{u_\sigma}{\sqrt{v_\sigma}} \\ 0 & \nicefrac{1}{\sqrt{v_\sigma}}
\end{pmatrix}$, then $g_{\tau_\sigma}^{-1} \begin{pmatrix} m \\ n \end{pmatrix} =\begin{pmatrix} \alpha_\sigma \\ \beta_\sigma \end{pmatrix}$ with $\alpha_\sigma= \frac{m-n u_\sigma}{\sqrt{v_\sigma}}$ and $\beta_\sigma= n \sqrt{v_\sigma} $. Thus
\begin{align}
    \int_{F_\8^\times} \omega_{l}(g_{\tauud})\phi_\8 \begin{pmatrix}mt_\8 \\ nt_\8  \end{pmatrix} \psi_\8(t_\8) \lvert t_\8\rvert^{1+s} dt^\times_\8 = 2^N \prod_{\sigma} \int_{0}^\8 \phi_\sigma\begin{pmatrix} \alpha_\sigma t_\sigma \\ \beta_\sigma t_\sigma \end{pmatrix}  t_\sigma^{1+s} \frac{dt_\sigma}{t_\sigma},
\end{align}
since $\phi_\sigma$ and $\psi_\sigma$ are both odd functions. At the place $\sigma$ we have
\begin{align}
\int_{0}^\8 \phi_\sigma\begin{pmatrix} \alpha_\sigma t_\sigma \\ \beta_\sigma t_\sigma \end{pmatrix}  t_\sigma^{1+s} \frac{dt_\sigma}{t_\sigma}
& = -i z_\sigma\int_{0}^\8 e^{-\pi t_\sigma^2\vert z_\sigma\vert^2} t_\sigma^{2+s} \frac{dt_\sigma}{t_\sigma},
\end{align}
where $z_\sigma=\alpha_\sigma+i\beta_\sigma=\frac{\overline{m-n\tau_\sigma}}{\sqrt{v_\sigma}}$. The right hand side converges for $\re(s)>-2$ to
\begin{align}
& - \frac{iz_\sigma}{2\vert z_\sigma \vert ^{2+s}} \Gamma \left (1+\frac{s}{2} \right )\pi^{-\left (1+\frac{s}{2}\right )}= \frac{1}{2i\pi}\Gamma \left (1+\frac{s}{2} \right )\pi^{-\frac{s}{2}} \frac{v_\sigma^{\frac{1}{2}+s}}{(m-n\tau_\sigma)\vert m-n\tau_\sigma \vert^{s}}.
\end{align}
Hence 
\begin{align}
 Z_\8\left (g_{\tauud},\begin{pmatrix} m \\ n \end{pmatrix}, \phi,\psi,s \right )=\frac{\Lambda(1+s)}{\left ( i\pi \right )^N} \frac{(v_1 \cdots v_N)^{\frac{1}{2}+s}}{\N(m-n\tauud)\vert \N(m- n \tauud) \vert^{s}}.
\end{align}
\end{proof}

\begin{prop} \label{proptransfo} For $\gamma =\begin{pmatrix} a & b \\ c & d \end{pmatrix}$ in $\GL_2(F)^+$ and $\phi_\fin$ arbitrary we have
\[E(\gamma \tauud,\phi_\fin,\psi,s)= \vert\det(\gamma)\vert^{\frac{1}{2}} \N(c\tauud+d)\vert \N(c\tauud+d) \vert^{s} E(\tauud,\omega_l(\gamma^{-1})\phi_\fin,\psi,s).\]
In particular, if $\Gamma$ preserves $\phi_\fin$ then $E( \tauud,\phi_\fin,\psi)$ is a Hilbert modular form of parallel weight one and level $\Gamma$.
\end{prop}

\section{Kudla-Millson form and special cycles} \label{sectionkm}

\paragraph{The quadratic space.} Let us recall from the introduction the appropriate setting that we want to consider.
Let $F$ be a totally real field of degree $N$ with ring of integers $\Ocal$. Let $X^0_F \coloneqq F^2$ be the $2$-dimensional quadratic $F$-space with the quadratic form $Q^0(\xbf,\ybf)=xy'+x'y$ where $\xbf=(x,x')$ and $\ybf=(y,y')$ are two vectors in $F^2$. It is represented by the symmetric matrix $A(Q^0)=\begin{pmatrix}
0 & 1 \\ 1 & 0
\end{pmatrix}$. Let $\SO(F^2)$ be its orthogonal group over $F$. We have an isomorphism
\begin{align} \label{isoso11}
F^\times& \longrightarrow \SO(F^2) \nonumber \\
t & \longmapsto \begin{pmatrix}
t & 0 \\ 0 & t^{-1}
\end{pmatrix}.
\end{align} Let $X_\Q \coloneqq \Res_{F/\Q} F^2$ be the quadratic space obtained by restriction of scalars. As a vector space it is the $2N$-dimensional $F^2_\Q$ over $\Q$, and the quadratic form is given by $Q \coloneqq \Tr_{F/\Q} \circ \; Q^0$.  At every place $v$ of $\Q$ let $X_{\Q_v} \coloneqq X_\Q \otimes \Q_v = F^2_{\Q_v}$. Using the $\Z$-basis of $\Ocal$ that we fixed we obtain an isomorphism between $F_{\Q_v}^2$ and $\Q_v^{2N}$, equipped with the quadratic form
\begin{align}
A(Q)=\begin{pmatrix} 0 & A \\ A & 0 \end{pmatrix}
\end{align}
where $A= \transp{g_\8}g_\8$ is as in \eqref{Adef} and is positive definite. Let $H(\Q)=\SO(F_\Q^2)$ be the orthogonal group of $X_\Q$, given by
\begin{align}
H(\Q) = \left \{ h \in \SL_{2N}(\Q) \left \vert \transp{h} \begin{pmatrix}  & A \\ A & \end{pmatrix}h=\begin{pmatrix}  & A \\ A & \end{pmatrix} \right . \right \}.
\end{align}

\paragraph{A locally symmetric space.} \label{sectionlss}
Let $X_\R = F^2_\R$ be the real points of $X_\Q$. It is a space of signature $(N,N)$. We fix a $\Z$-basis of $\Ocal$ and use it to identify $F_\R^2$ with $\R^{2N}$. The Lie group $H(\R) \coloneqq \SO(F^2_\R)$ has two connected components and we denote by $H(\R)^+$ the connected component of the identity. Let $\D$ be the space of {\it oriented} negative $N$-planes
\begin{align}
\D \coloneqq & \left \{ z \subset F_\R^2 \, \vert \; z \, \textrm{oriented}, \; \dim(z)=N, \quad \restr{Q}{z} < 0 \right \};
\end{align}
it is an $N^2$-dimensional Riemannian manifold. This space has two connected components that we denote by $\D^+$ and $\D^-$. For an oriented subspace $z$ let $o(z)$ in $\wedge^N z$ be its orientation.

We also consider the basis $\ebf_{1}, \cdots, \ebf_N, \fbf_1, \cdots, \fbf_N$ of $F^2_\R$ where $\ebf_{k} \coloneqq (e_k, e_k)$ and $\fbf_{k} \coloneqq (e_k, -e_k)$, and $e_k$ is the standard unit vector in $\R^N$. Note that $Q(\ebf_k,\ebf_k)$ is positive and $Q(\fbf_k,\fbf_k)$ is negative. We orient  $F^2_\R$ by $\ebf_{1} \wedge \cdots \wedge \ebf_{N} \wedge \fbf_{1} \wedge \cdots \wedge \fbf_{N}$. Let $z_0 \coloneqq \bigl < \fbf_{1}, \dots, \fbf_{N} \bigr >$ be the negative plane oriented by $o(z_0) \coloneqq \fbf_{1} \wedge \cdots \wedge \fbf_{N}$.

The group $H(\R)^+$ acts transitively on $\D^+$ by sending $z_0$ to $hz_0$. Hence we can identify $\D^+$ with $H(\R)^+/K_\8(z_0)$ where $K_\8(z_0)$ is the stabilizer of $z_0$ in $H(\R)^+$ and is isomorphic to $\SO(N) \times \SO(N)$. We can also identify $\D$ with $H(\R)/K_\8(z_0)$. 
 
For a positive vector $\xbf$ in $F_\R^2$ we define the totally geodesic submanifold
\begin{align}
 \D_\xbf \coloneqq \left \{ z \in \D^+ \; \vert \; z \subset \xbf^\perp \right \}.   
\end{align}
and $\D^+_\xbf \coloneqq \D^+ \cap \D_\xbf$. Let $H_\xbf(\R)$ be the stabilizer of $\xbf$ in $H(\R)$ and let $K_\xbf(z_0) \coloneqq H_\xbf^+(\R) \cap K_\8(z_0)$. We then have a diffeomorphism
\begin{align}
    H_\xbf^+(\R)/K_\xbf(z) & \longrightarrow \D^+_\xbf \subset \D^+ \nonumber \\
    hK_U(z) & \longmapsto hz.
\end{align}

\paragraph{Orientations.}
We need to orient the spaces $\D^+$ and $\D_\xbf^+$. We fix orientations 
\begin{align} o(X_{\R}) \coloneqq \ebf_{1} \wedge \cdots \wedge \ebf_{q} \wedge \ebf_{p+1} \wedge \cdots \wedge \ebf_{p+q}, \qquad o(z_0) \coloneqq \ebf_{p+1} \wedge \cdots \wedge \ebf_{p+q} \end{align}
of $X_{\R}$ and of $z_0$, let us explain how this yields an orientation on $\D^+$. If $z$ is a negative plane, we write $z=h_zz_0$ for some $h_z$ in $H(\R)^+$. Let $\ebf_k(z)=h_z\ebf_k$ so that $z=\Span \bigl < \ebf_{p+1}(z), \cdots , \ebf_{p+q}(z)  \bigr >$ and the orientation of $z$ is $\ebf_{p+1}(z) \wedge \cdots \wedge \ebf_{p+q}(z)$. We can identify $T_z\D^+$ with $z^\vee \otimes z^\perp$, so that we have to orient $z^\vee \otimes z^\perp$. First we orient $z^\perp$ by the rule that $o(z^\perp) \wedge o(z)=o(X_{\R})$. Then we orient $z^\vee$ by $\ebf^\vee_{p+1}(z) \wedge \cdots \wedge \ebf^\vee_{p+q}(z)$. Finally, we need to orient the tensor product. If $V$ and $W$ are two vectors spaces with orientations $v_1 \wedge  \cdots \wedge v_N$ and $w_1 \wedge  \cdots \wedge w_M$ then the we orient the basis $v_i \otimes w_j$ with the lexicographic order from right to left. 
\begin{ex}
If $o(V)=v_1 \wedge v_2$ and $W=w_1 \wedge w_2$ then $o(V \otimes W) = (v_1 \otimes w_1) \wedge (v_2 \otimes w_1) \wedge (v_1 \otimes w_2) \wedge (v_2 \otimes w_2)$.
\end{ex}
Let us now orient $\D_\xbf^+$, given a positive vector $\xbf$. We have an isomorphism between $T_z\D_\xbf^+$ and $z^{\vee} \otimes (z^\perp \cap \xbf^\perp)$. With the Riemannian metric on $\D^+$ the normal bundle is $N_z\D_\xbf^+ = z^{\vee} \otimes \xbf$, which can be identified with $z^{\vee}$.
The space $N_z\D_\xbf^+$ is oriented by $z^{\vee}$ and we orient $T_z\D_\xbf^+$ by the rule $o(T_z\D_\xbf^+) \wedge o(N_z\D_\xbf^+)=o(T_z\D^+)$.

\paragraph{Adelic spaces.}
Let $K \coloneqq K_\8 K_\fin$ where $K_\fin$ is an open compact subgroup preserving a Schwartz function $\varphi_\fin$ in $\Scal(F^2_{\A_\fin})$ and $K_\8=K_\8(z_0)$ the maximal compact subgroup of $H(\R)^{+}$ stabilizing $z_0$. We define the double coset space
\begin{align}M_{K} \coloneqq H(\Q) \backslash H(\A)/K \simeq H(\Q) \backslash \D \times H(\A_\fin)/K_\fin,
\end{align}
where the second isomorphism sends $H(\Q)(h_\8,h_\fin)K$ to $H(\Q)( h_\8 z_0,h_\fin)K $. There exists elements $h_1, \dots, h_r$ in $H(\A_\fin)$ such that
\begin{align} \label{doublecosetH}
 H(\A_\fin)= \bigsqcup_{i=1}^r H(\Q)^+ h_i K_\fin.
\end{align} Let $\Gamma'_{h_i} \coloneqq H(\Q)^+ \cap h_i K_\fin h_i^{-1}$ and $\Gamma_{h_i}$ its image in $H^{\ad}(\Q)^+$, where the latter is the intersection of $H(\Q)$ with $H^{\ad}(\R)^+ \coloneqq H(\R)^+/Z_H(\R)$. We have a homeomorphism
\begin{align} \label{adeless}
M_{K} \simeq \bigsqcup_{i=1}^r M_{h_i},
\end{align}
where $M_h$ is the locally symmetric space $M_h \coloneqq \Gamma_{h} \backslash \D^+
$.
The map goes as follows. Let $H(\Q)(z,h_{\fin})K$ be a double coset. Let us define $\delta_z$ to be $1$ if $z$ is in $\D^+$ and $\delta_z$ is any element in $H(\Q)-H(\Q)^+$, if $z$ is in $\D^-$. The element $\delta_z$ permutes $\D^+$ and $\D^-$, hence $H(\Q)(z,h_{\fin})K=H(\Q)(\delta_z z,\delta_z h_{\fin})K$ 
with $\delta_z z$ in $\D^+$. Write $\delta_z h_{\fin}= h^{-1} h_i k_{\fin}$ for some $k_{\fin}$ in $K_\fin$ and $h$ in $H(\Q)^+$ and $i$ in $\{1, \dots, r\}$. Then this coset is mapped to the point $\Gamma_{h_i} h \delta_z z$.


\paragraph{Special cycles.} For $h_i$ in $H(\A_\fin)$ and $\xbf$ a vector in $F_\Q^2$ we define the {\it connected cycles} $C_{\xbf}(h_i)$ in $M_{h_i}$ to be the image of the composition
\begin{align}
    \Gamma_{h_i,\xbf} \backslash \D^+_{\xbf} \hooklongrightarrow \Gamma_{h_i,\xbf} \backslash \D^+  \longrightarrow  M_{h_i}
\end{align}
where $\Gamma_{h_i,\xbf} \coloneqq H_{\xbf}(\Q)^+ \cap h_i K h_i^{-1}$ and the second map is the natural projection. Note that $\Gamma_{h_i,\xbf}$ does not contain $-1$ so its image in $H^{\ad}(\Q)^+$ is $\Gamma_{h_i,\xbf}$. For a positive rational number $n$ we define the {\it weighted cycles} 
\begin{align}
C_n(\varphi_\fin,h_i) \coloneqq & \sum_{\substack{\xbf \in  \Gamma'_{h_i} \backslash F_\Q^2 \\ Q(\xbf,\xbf)=2n}} \varphi_{\fin}(h_i^{-1}\xbf) C_{\xbf}(h_i) \in \Zcal_{N^2-N}(\overline{M_{h_i}}, \partial \overline{M_{h_i}}; \R) \nonumber \\  
C_{n}(\varphi_{\fin}) \coloneqq & \sum_{i=1}^r  C_{n}(\varphi_\fin,h_i) \in \Zcal_{N^2-N}(\overline{M_K}, \partial \overline{M_K}; \R),
\end{align}
where both sums are finite.

\paragraph{Cohomology of $M_K$.} We can identify $\Omega^\bullet(M_h)$ with $\Omega^\bullet(\D^+)^{\Gamma_h}$, where the latter is the space of $\Gamma_h$-invariant forms on $\D^+$. We use the isomorphism \eqref{adeless} to define the space $\Omega^\bullet(M_{K}) \coloneqq \bigoplus_{i=1}^r \Omega^\bullet(M_{h_i})$ of differential forms on $M_K$ and the cohomology $\cohom^\bullet(M_{K};\R) \coloneqq \bigoplus_{i=1}^r \cohom^\bullet(M_{h_i};\R)$ of $M_K$. Similarly we define the homology of $M_K$ by $\cohom_\bullet(M_{K};\R) \coloneqq \bigoplus_{i=1}^r \cohom_\bullet(M_{h_i};\R)$. Let $C^\8(H(\A_\fin))$ be the space of smooth ({\em i.e. }locally constant) functions on $H(\A_\fin)$. Let $\g$ and $\kfrak$ be the Lie algebras of $H(\R)$ and $K_\8$. Let $\g=\p+\kfrak$ be a decomposition that is orthogonal with respect to the Killing form. We can also see $\Omega^\bullet (M_K)$ as
\begin{align} \label{XKform2}
    \Omega^\bullet(M_K) & \simeq \left [ \Omega^\bullet (\D^+) \otimes_{\Q} C^\8(H(\A_\fin))\right ]^{H(\Q)\times K_\fin} \nonumber \\
    & \simeq \left [  C^\8(H(\Q) \backslash H(\A)) \otimes_{\R} \sideset{}{^\bullet}\bigwedge \p^\ast \right ]^{K},
\end{align}
where $\p^\ast=\Hom(\p,\R)$. The first isomorphism is
\begin{align}
\omega \otimes f \longmapsto \sum_{i=1}^r  f(h_i)\omega \in \bigoplus_{i=1}^r \Omega^\bullet(\D^+)^{\Gamma_{h_i}}
\end{align}
and the second ones comes from the isomorphism
\begin{align}
    \Omega^\bullet(\D^+) \simeq \left [C^\8(H(\R)^+) \otimes_{\R} \sideset{}{^\bullet} \bigwedge \p^\ast \right ]^{K_\8}.
\end{align}

\paragraph{The Kudla-Millson form.}
The Kudla-Millson form
\[\varphi_{KM} \in \Omega^N(\D^+,\Scal(F_\R^2))^{H(\R)^+} \simeq \left [ \Omega^N(\D^+) \otimes \Scal(F_\R^2)\right]^{H(\R)^+} \] is an $N$-form valued in a Schwartz space that satisfies the following properties:
\begin{enumerate} \label{properties}
    \item $\varphi_{KM}(\xbf)$ is closed for any $\xbf$ in $F_\R^2$,
    \item $\varphi_{KM}$ is $H(\R)^+$-invariant in the sense that 
    \begin{align}
        h^\ast \varphi_{KM}(\xbf)=\varphi_{KM}( h^{-1}\xbf )=\omega_{os}(h)\varphi_{KM}(\xbf).
    \end{align}
    In particular it is $\Gamma_{\xbf}$-invariant, where $\Gamma_{\xbf}$ is the stabilizer of $\xbf$ in $\Gamma$. Hence we can view $\varphi_{KM}(\xbf)$ as a form on $\Gamma_\xbf \backslash \D^+$.
    \item Moreover, for any positive vector $\xbf$ in $F^2_\R$ and $\omega$ in $\Omega^{N^2-N}_c(\Gamma_\xbf \backslash \D^+)$ a compactly supported form, we have
    \[\int_{\Gamma_\xbf \backslash \D^+}\varphi_{KM}(\xbf) \wedge \omega = e^{-\pi Q(\xbf,\xbf)}\int_{\Gamma_\xbf \backslash \D_{\xbf}^+} \omega.\]
\end{enumerate}
The last property means that the form
\begin{align}
\varphi^0(\xbf) \coloneqq e^{\pi Q(\xbf,\xbf)}\varphi_{KM}(\xbf) \in \Omega^N(\Gamma_\xbf \backslash \D^+)
\end{align}
represents the Poincaré dual of $\Gamma_\xbf \backslash \D^+_\xbf$ in $\Gamma_\xbf \backslash \D^+$.
\begin{rmk}
The Kudla-Millson defined here is the original Kudla-Millson multiplied by $2^\frac{N}{2}$.
\end{rmk}
\paragraph{The Mathai-Quillen formalism.} We briefly recall how to recover the Kudla-Millson form from the Mathai-Quillen formalism, see \cite{rbr} for more details. This will simplify the explicit computations involving the Kudla-Millson form.

We have an orthogonal decomposition of $F^2_\R$ as $z_0^\perp \oplus z_0$ where $z_0$ is a before and $z_0^\perp$ is a positive $N$-plane. We identify $z_0$ with $F_\R = \R^N$ by choosing a basis and consider the metric vector bundle 
\begin{align}
    E \coloneqq H(\R)^+ \times_{K_\8(z_0)} \R^N,
\end{align} which is a rank $N$ bundle over $\D^+$. It consists of equivalence classes $[h,w]$ where $h$ in $H(\R)^+$ and $w$ is a vector in $\R^N$, and the equivalence relation is $[h,w]=[hk,k^{-1}w]$ for all $k$ in $K_\8(z_0)$. The metric is given by $-\restr{Q}{{z_0}}$, which is positive definite. The group $\Gamma_\xbf$ also acts on $E$ by $\gamma[h,w]=[\gamma h,w]$ and we have a vector bundle $\Gamma_\xbf \backslash E$ over $\Gamma_\xbf \backslash \D^+$. Let $E_0$ denote the image of the zero section. By integration along the fibers we obtain the Thom isomorphism
\begin{align}
    \cohom^{k+N}(\Gamma_\xbf \backslash E,\Gamma_\xbf \backslash (E - E_0)) & \longrightarrow \cohom^k(\Gamma_\xbf \backslash \D^+),
 \end{align}
where $\cohom^{k+N}(\Gamma_\xbf \backslash E,\Gamma_\xbf \backslash (E- E_0))$ is the cohomology with compact vertical support. For $k=0$ we call the preimage of $1$ in $\cohom^{N}(\Gamma_\xbf \backslash E,\Gamma_\xbf \backslash (E - E_0))$ the {\it Thom class} of the bundle $\Gamma_\xbf \backslash E$. A {\it Thom form} is any representative of the Thom class. In \cite{mq} Mathai and Quillen construct a canonical Thom form $U_{MQ}(E)$ in $\Omega^N(E)$, that is $H(\R)^+$-invariant (in particular $\Gamma_\xbf$-invariant) and represent the Thom class. 
Let $\pr$ be the orthogonal projection from $F^2_\R$ onto $z_0 \simeq \R^N$. Consider the section
\begin{align}
    s_\xbf \colon \D^+ & \longrightarrow E \nonumber \\
    z & \longmapsto [h_z, \pr(h_z^{-1}\xbf )]
\end{align}
where $h_z$ is any element in $H(\R)^+$ sending $z_0$ to $z$. Its zero locus is precisely $\D_\xbf^+$ by \cite[Proposition.~4.1]{rbr}  and 
\begin{align}
    \varphi_{KM}(\xbf) =e^{-\pi Q(\xbf,\xbf) } s_{\xbf}^\ast U_{MQ}(E).
\end{align}

\paragraph{The Kudla-Millson theta series.}
We define the theta series
\begin{align}\Theta_{os}(g,h_\fin,\varphi_{KM}\otimes\varphi_\fin) & \coloneqq \sum_{\xbf \in X_{\Q}} \omega_{os}(g,h_\fin) \varphi_{KM}(\xbf)\varphi_\fin(\xbf) \nonumber \\
& = \sum_{\xbf \in X_{\Q}} \omega_{os}(g) \varphi_{KM}(\xbf)\varphi_\fin(h_\fin^{-1}\xbf ) \in \Omega^N(\D^+)
\end{align}
for $g$ in $\SL_2(\A)$ and $h_\fin$ in $H(\A_\fin)$. Note that it converges since the Kudla-Millson form is rapidly decreasing. The form $\Theta_{os}(g,h_i,\varphi_\fin)$ is $\Gamma_{h_i}$-invariant, hence defines a form on $M_{h_i}$ and we obtain a form
\begin{align}
    \Theta_{os}(g,\varphi_{KM}\otimes\varphi_\fin) & \coloneqq \sum_{i=1}^r \Theta_{os}(g,h_i,\varphi_{KM}\otimes\varphi_\fin) \in \Omega^N(M_K).
\end{align}
For $\tau=u+iv$ in $\HH$ we define the classical Kudla-Millson theta series
\begin{align}\label{firstdef} \Theta_{KM}(\tau,h_\fin,\varphi_\fin) & \coloneqq v^{-\frac{N}{2}} \Theta_{os}(g_\tau,h_\fin,\varphi_{KM} \otimes \varphi_\fin)
\end{align}
where $g_\tau=\begin{pmatrix} \sqrt{v} & \nicefrac{u}{\sqrt{v}} \\ 0 & \nicefrac{1}{\sqrt{v}} \end{pmatrix}$ is a matrix in $\SL_2(\R)$ that maps $i$ to $\tau$ by Möbius transformations. 
\begin{rmk} We have $\omega_{os}(k_\theta) \varphi_{KM}=e^{i\theta N}$ for 
    \begin{align}
     k_\theta = \begin{pmatrix}
    \cos(\theta) & \sin(\theta) \\
    -\sin(\theta) & \cos(\theta)
    \end{pmatrix} \in \SO(2).   
    \end{align}
Since $\omega_{os}$ is a representation, one can check that \eqref{firstdef} does not depend on the choice of the matrix $g_\tau$ sending $i$ to $\tau$.

\end{rmk}

By summing over the different connected components we then also get a closed form
\begin{align}
 \Theta_{KM}(\tau,\varphi_\fin) \coloneqq \sum_{i=1}^r \Theta_{KM}(\tau,h_i,\varphi_\fin) \in \Omega^N(M_{K}).
\end{align}

\begin{lem} \label{fourierinter} We can rewrite the theta series as
\begin{align} 
     \Theta_{KM}(\tau,\varphi_\fin) & = \Theta_0(v,\varphi_\fin) +  \sum_{n \in \Q^\times} \Theta_n(v,\varphi_\fin) e^{2 i \pi n \tau} 
\end{align}
where
\begin{align}
    \Theta_n(v,\varphi_\fin) = \sum_{i=1}^r\sum_{\substack{\xbf \in X_{\Q} \\ Q(\xbf,\xbf)=2n}} \varphi_\fin(h_i^{-1}\xbf) \varphi^0(\sqrt{v}\xbf).
\end{align}
\end{lem}
\begin{proof}

It follows from the formulas \eqref{weilrepslgspin} for the Weil representation that
\begin{align}
    v^{-\frac{N}{2}}\omega_{os}(g_\tau) \varphi_{KM}(\xbf)=\varphi^0(\sqrt{v}\xbf)e^{i\pi \tau Q(\xbf,\xbf)}
\end{align}
where $\varphi^0(\xbf)=e^{\pi Q(\xbf,\xbf)}\varphi_{KM}(\xbf)$. Thus 
\begin{align} \label{suminproof}
    \Theta_{KM}(\tau,\varphi_\fin) & = \sum_{i=1}^r\sum_{ \xbf \in X_{\Q}} \varphi_\fin(h_i^{-1}\xbf) \varphi^0(\sqrt{v}\xbf)e^{i\pi \tau Q(\xbf,\xbf)}.
\end{align}
By summing over the vectors of same length we can rewrite \eqref{suminproof} as
\begin{align} \label{suminproof2}
    \Theta_{KM}(\tau,\varphi_\fin) & = \sum_{n \in \Q} \left (\sum_{i=1}^r\sum_{ \substack{\xbf \in X_{\Q} \\ Q(\xbf,\xbf)=2n}} \varphi_\fin(h_i^{-1}\xbf) \varphi^0(\sqrt{v}\xbf) \right )e^{2 i \pi n \tau} \nonumber \\
    & = \sum_{n \in \Q} \Theta_n(v,\varphi_\fin) e^{2 i \pi n \tau}.
\end{align}
\end{proof}

\begin{thm}[Kudla-Millson] \label{kmthm}
Let $C$ in $\Zcal_N(M_K;\R)$ be an $N$-cycle in $M_K$ and write the cycle as $\sum C_i$ where every $C_i$ is an $N$-cycle in $M_{h_i}$. Then
\begin{align} \label{kmlift} \int_{C} \Theta_{KM}(\tau, \varphi_\fin)=\sum_{i=1}^r \int_{C_i} \Theta_{KM}(\tau, h_i, \varphi_\fin)
\end{align}
is a {\em holomorphic} modular form of weight $N$. Moreover, it admits the Fourier expansion
\[\int_{C} \Theta_{KM}(\tau, \varphi_\fin)=\int_{C} \Theta_0(v,\varphi_\fin) + \sum_{n \in \Q_{ > 0}} \int_{C} \Theta_n(v,\varphi_\fin) e^{2 i \pi n \tau}\]
and for $n$ positive we have
\begin{align}
    \int_C \Theta_n(v,\varphi_\fin)= \kappa \bigl <C_n(\varphi_\fin),C  \bigr >_{M_K}.
\end{align}
\end{thm}
\noindent In particular, they show the holomorphicity (in $\tau$) of the resulting function and the vanishing of the Fourier coefficients vanishes if $n$ is negative. In the next section we will consider the integral of $\Theta_{KM}(\tau,\varphi_\fin)$ over a relative cycle $C\otimes \psi$ in $\Zcal_N(\overline{M_K},\partial \overline{M_K};\R)$ which is twisted by a totally odd unitary Hecke character $\psi$. In particular since $C\otimes \psi$  will be non-compact the result of Kudla-Millson does not apply immediately. Using a seesaw argument, we will show that 
\begin{align}
    \int_{C \otimes \psi} \Theta_{KM}(\tau,\varphi_\fin)
\end{align}
equals the diagonal restriction of an Eisenstein series, which will prove in particular that it is holomorphic. Although the vanishing of the negative Fourier coefficients also follows from this equality, we show it directly in Proposition \ref{Ifourier}. Moreover, we will show in Proposition \ref{intertop} that the Fourier coefficients can be interpreted as intersection numbers. The proof of the proposition, and the Remark \ref{rmkkappa} before it also explains the appearance of the factor $\kappa$ in Theorem \ref{kmthm}.

\begin{rmk}\label{rmkproblem}  If $C$ and $C'$ are immersed submanifolds, then
\begin{align} \label{interinteg2}
    \int_{C}\PD(C')=\int_{M_K} \PD(C') \wedge \PD(C) = \bigl < C',C \bigr >_{M_K},
\end{align}
where the right hand side is the topological intersection numbers. This works when at least one of the two cycles is compact. If the two cycles $C$ and $C'$ are non-compact and intersect infinitely many times, then the integrals \eqref{interinteg2} do not converge and the intersection number is not well defined. On the other hand, if $C$ and $C'$ intersect finitely many times in $M_K$ then $\bigl < C,C' \bigr >_{M_K}$ is well-defined and the integrals converge. However it does not mean that the equality \eqref{interinteg2} holds. Indeed, to make sense of \eqref{interinteg2} for non-compact cycles one would need to study the extension of the forms to the boundary of a compactification as well as the intersections of the cycles in that boundary, as it is done in \cite{funkmil} for example.

The special cycle $C_n(\varphi_\fin)$ is an immersed submanifold of codimension $N$, and for positive $n$ the form $\kappa^{-1}\Theta_n(v,\varphi_\fin)$ represents the Poincaré dual of $C_n(\varphi_\fin)$ in $\Omega^N(M_{h_i})$. If $C$ is compact we get
\begin{align} \label{someeq}
    \int_C \Theta_n(v,\varphi_\fin)=  \kappa \bigl < C_n(\varphi_\fin),C \bigr >_{M_K}.
\end{align}
 In our case we will replace $C$ by a non-compact cycle $C \otimes \psi$ that intersects $C_n(\varphi_\fin)$ in a compact set. We will show in Proposition \ref{intertop} that \eqref{someeq} also holds for $C \otimes \psi$.

\end{rmk}

\section{Integral of $\Theta_{KM}$ over a relative class $C\otimes \psi$ } \label{sectionclass}

We will now define the relative cycle $C\otimes \psi$ and compute $\int_{C \otimes \psi} \Theta_{KM}(\tau,\varphi_\fin)$.

\paragraph{The seesaw.} Let $W^0_F=X^0_F \oplus X^0_F$ be the $4$-dimensional symplectic $F$-space as on Page \pageref{locweil}, and let $W_\Q= \Res_{F/\Q} W^0_F$ be the restriction of scalars. Then $W_\Q=X_\Q \oplus X_\Q$ is a $4N$-dimensional symplectic $\Q$-space. Let $F_w^2$ be the completion of the quadratic space at a place $w$ of $F$. The isomorphism $\varsigma_v$ between $F_{\Q_v}$ and $F_v$ from \eqref{varsigmaiso} induces an isomorphism of quadratic $\Q_v$-spaces
\begin{align}
   F_v^2 \coloneqq \bigoplus_{w \mid v} F_w^2 \simeq F_{\Q_v}^2.
\end{align}
Hence we obtain a natural embedding of $\SO(F^2_v)$ in $H(\Q_v).$ We compose this embedding with the isomorphism \eqref{isoso11} to get 
\begin{align}
h \colon F^\times_v \hooklongrightarrow H(\Q_v),
\end{align}
that we will describe more concretely. First note that we can embed $\GL_N(\Q_v)$ in $H(\Q_v)$ by
\begin{align}
   \GL_N(\Q_v) & \hooklongrightarrow H(\Q_v) \nonumber \\
    M & \longmapsto \begin{pmatrix}
     \adjhash{M} & 0 \\ 0 & M
    \end{pmatrix}
\end{align}
where $\adjhash{M} \coloneqq A^{-1}\transp{M}^{-1}A$. The embedding $h$ is obtained by composing it with the embedding $\gamma$ of $F_v^\times$ in $\GL_N(\Q_v)$:
\begin{align} \label{embH0}
    h \colon F_v^\times & \hooklongrightarrow H(\Q_v) \nonumber \\
    t_v & \longmapsto \begin{pmatrix}
     \adjhash{\gamma(t_v)} & 0 \\[2.5ex] 0 & \gamma(t_v)
    \end{pmatrix}.
\end{align}
Note that at infinity we have $\gamma(t_\8)=g_\8^{-1}g(t_\8)g_\8$, hence $ \adjhash{\gamma(t_\8)}=\gamma(t_\8^{-1})$ and
    \begin{align}
   h(t_\8)= h_\8^{-1}\begin{pmatrix}
     g(t_\8)^{-1} & 0 \\[2.5ex] 0 & g(t_\8)
    \end{pmatrix} h_\8
    \end{align}
 where $h_\8=\begin{pmatrix}
     g_\8 & 0 \\ 0 & g_\8
    \end{pmatrix}$. The centralizer of $F_v^\times$ in $\Sp(W_{\Q_v})$ is $\SL_2(F_v)=\prod \SL_2(F_w)$, embedded in $\Sp(W_{\Q_v})$ by
\begin{align} \label{embsl2}
\SL_2(F_v) & \hooklongrightarrow \Sp(W_{\Q_v}) \nonumber \\
\begin{pmatrix} a & b \\ c & d \end{pmatrix} & \longmapsto \begin{pmatrix}
    \gamma(a) &&\gamma(b)& \\
    &\gamma(a)&&\gamma(b) \\
    \gamma(c)&&\gamma(d)& \\
    &\gamma(c)&&\gamma(d)
    \end{pmatrix}.
\end{align} We obtain a dual pair $F_v^\times \times \SL_2(F_v)$ in $\Sp(W_{\Q_v}) \nonumber$
and a seesaw
\begin{equation}
\begin{tikzcd}
H(\Q_v) \arrow[dash, dr] & \SL_2(F_v) \\
F_v^\times \arrow[dash]{u} \arrow[dash, ru] & \SL_2(\Q_v) \arrow[dash]{u}
\end{tikzcd}
\end{equation}
 where the righthand side is the diagonal embedding
 \begin{align}
     \iota_\Delta \colon \SL_2(\Q_v) \longrightarrow \SL_2(F_v).
 \end{align}

\paragraph{The twisted class $C\otimes \psi$.} The space $F_\sigma^2$ is of signature $(1,1)$. Let $\D_\sigma$ be the corresponding symmetric space, that we identify with $\R^\times$. We have an isomorphism of quadratic spaces between $F_\R^2$ and $F^2_\8 = \oplus_\sigma F^2_\sigma$. We can view the product of symmetric spaces $\D_0 \coloneqq\prod \D_\sigma$ as the subspace
\begin{align}
    \D_0 = \D_{\sigma_1} \times \cdots \times \D_{\sigma_N} \simeq \left \{ z \in \D \; \vert \; z = \oplus_\sigma (z \cap F_\sigma^2) \right \},
\end{align}
of $\D$. The connected component $\D_0^+$ can be identified with $\R_{>0}^N$. The image of $F_\8^\times=\prod F_\sigma^\times$ by $h$ is precisely the stabilizer of this subspace in $H(\R)$.

Let $\psi \colon F^\times \backslash \A_F^\times \longrightarrow \C^\times$ be a unitary totally odd Hecke character of finite order and conductor $\fffrak$. Let $K^0_\8$ be the compact
\begin{align}
K^0_\8 \coloneqq \{ t_\8 \in \{ \pm1\}^N \left \vert \det(t_\8)=1 \right .\} \subset F_\8^\times.   
\end{align} We set
\begin{align}
    M_{\fffrak} \coloneqq F^\times \backslash \A_F^\times/K^0(\fffrak) \simeq F^\times \backslash \D_0^+ \times \A_{F,\fin}^\times/\widehat{U}(\fffrak)
\end{align}
where $K^0(\fffrak)=K^0_\8 \times \widehat{U}(\fffrak)$. Suppose that the ideal $\fffrak$ is small enough that $h(\widehat{U}(\fffrak))$ is contained in $K_\fin$. Then the embedding $h$ induces an immersion
\begin{align} \label{emb1}
    h \colon M_{\fffrak} \longrightarrow M_K.
\end{align}
As in \eqref{doublecosetH} we have a decomposition
\begin{align} \label{decomp2}
    \A_{F,\fin}^\times= \bigsqcup_{[\afrak] \in \Cl_{\fffrak}(F)^+} F^{\times,+} t_\afrak \widehat{U}(\fffrak)
\end{align}
for some $t_\afrak$'s in $\A_{F,\fin}$ indexed by the ray class group.
The double coset space $M_\fffrak$ is the disjoint union of symmetric space
\begin{align} \label{union1}
    M_\fffrak=\bigsqcup_{[\afrak] \in \Cl_\fffrak(F)^+} \Gamma_\afrak \backslash \D_0^+,
\end{align}
where $\D_0^+=F_\8^{\times,+}$ and $\Gamma_\afrak \coloneqq F^{\times,+} \cap t_\afrak \widehat{U}(\fffrak)t_\afrak^{-1} = F^{\times,+} \cap \widehat{U}(\fffrak)$, since $\A_F^\times$ is commutative. For every class $[\afrak]$ in the ray class group let $C_\afrak$ be the image of the connected component $\Gamma_\afrak \backslash F_\8^{\times,+}$ by the map \eqref{emb1}. It defines a relative cycle in $\Zcal_N(\overline{M_K},\partial \overline{M_K};\R)$, that is non compact since 
\begin{align}
\Gamma_\afrak \backslash F_\8^{\times,+} \simeq \R_{>0} \times \Gamma_\afrak \backslash F_\8^{1,+},
\end{align}
where $F^1$ are the elements of norm $1$.  More precisely the map \eqref{emb1} goes as follows: a point $\Gamma_\afrak t_\8$ is mapped to $F^\times(t_\8,t_\afrak)K^0$, which is mapped to 
\begin{align}H(\Q)(h(t_\8),h(t_\afrak))K_\fin = H(\Q)(z,h(t_\afrak))K_\fin, \end{align}
where $z=h(t_\8)z_0$ in $\D^+$. The element $h(t_\afrak)$ lies in one some double coset $H(\Q)^+h_iK_\fin$. Hence we can write $h(t_\afrak)=h_\afrak^{-1} h_i k_\fin$ for some $h_\afrak$ in $H(\Q)^+$ and $k_\fin$ in $K_\fin$. Thus 
\begin{align} H(\Q)(z,h(t_\afrak))K_\fin= H(\Q)(h_\afrak z,h_i)K_\fin, \end{align}
which is sent to $\Gamma_{h_i} h_\afrak z$. The cycle $C_\afrak$ is the image of $\D_\afrak^+ \coloneqq h_\afrak \D_0^+$ in $M_{h_i}$ by the natural projection map from $\D^+$ onto $M_{h_i}$. We define the relative cycle
\begin{align}
    C\otimes \psi \coloneqq \sum_{\afrak \in \Cl_\fffrak(F)^+} \psi(\afrak) C_\afrak \in \Zcal_N(\overline{M_K},\partial \overline{M_K};\R),
\end{align}
where we view $\psi$ as character on the ray class group.

\paragraph{Restriction to $\D_0^+$.} Let $dt^\times$ be the Haar measure on $\A_F^\times$ normalized such that $\vol^\times(\widehat{U}(\fffrak))=1$.  We identify $F_\8^{\times,+}$ with $\R_{>0}^N$ and orient it by the volume form $dt^\times_\8=dt_1^\times \cdots dt^\times_N$. Let $\g_0$ be the Lie algebra of $F_\8^\times$, that we identify with $\R^N$. The choice of the volume form induces an isomorphism between $\sideset{}{^N}\bigwedge \g_0^\ast$ and $\R$. It also yields the isomorphism
\begin{align}
    \Omega^N(\D_0^+) \simeq \left [C^\8(\R_{>0}^N) \otimes \sideset{}{^N}\bigwedge \g_0^\ast \right ]^{K^0_\8} \simeq C^\8(\R_{>0}^N)^{K^0_\8}.
\end{align}
Moreover, combining with \eqref{XKform2} we get the isomorphism
\begin{align} \label{isom1}
    \left [  C^\8(F^\times \backslash \A_F^\times) \right ]^{K^0(\fffrak)} & \longrightarrow \Omega^{N}(M_\fffrak)  \nonumber \\
    \eta_\8 \otimes \eta_\fin & \longmapsto \sum_{\afrak \in \Cl_\fffrak(F)^+} \eta_\fin(t_\afrak)\eta_\8dt^\times_\8.
\end{align}
If $\tilde{\eta}$ in $C^\8(F^\times \backslash A_F^\times)^{K^0(\fffrak)}$ corresponds to $\eta$ in $ \Omega^{N}(M_\fffrak)$ then
\begin{align}
    \int_{M_\fffrak}\eta = \frac{1}{\vol^\times(K^0_\8)}\int_{F^\times \backslash \A_F^\times}\tilde{\eta}(t)dt^\times.
\end{align}
Since $K^0_\8$ is $\{ \pm 1 \}^{N-1}$ we have $\vol^\times(K^0_\8)=2^{N-1}$. 

Recall that we have isomorphism between $F^2_\R$ and $F^2_\8=\prod_\sigma F^2_\sigma.$ Let $(\xbf_{\sigma_1}, \dots, \xbf_{\sigma_N})$ in $F^2_\8$ be the image of $\xbf=(x,x')$ in $F^2_\R$, where $\xbf_\sigma=(x_\sigma,x'_\sigma)$. We identify $F_\R$ with $z_0$ by sending $v$ to $(v,-v)$. Consider the tautological bundle $E=H(\R)^+ \times_{K_\8} F_\R$ over $\D^+$. By \cite{rbr} the Kudla-Millson form\footnote{Note that the Kudla-Millson form used here differs by a factor $2^\frac{N}{2}$ with the classical Kudla-Millson that appears in the original paper \cite{km} as well as in \cite{rbr}.} is given by $\varphi_{KM}(\xbf)= s_{\xbf}^\ast U_{MQ}$ where $U_{MQ}$ is the Mathai-Quillen form on $E$.  The bundle $E$ splits over $\D_0^+$ {\em i.e. }we have the diagram
\begin{equation} \label{diagram}
\begin{tikzcd}
\restr{E}{\D_0^+} \simeq \R_{>0}^N \times F_\8 \arrow[hook, r] & E \\
\D_0^+ \simeq \R^N_{>0}  \arrow[u,"\oplus s_{\xbf_\sigma}"] \arrow[hook, r] & \D^+ \arrow[u,"s_\xbf"]
\end{tikzcd},
\end{equation}
where the top map sends a pair $(t_\8,v)$ in $ \R_{>0}^N \times F_\8$ to the class $[h(t_\8),g_\8^{-1}v]$ in $E$. Moreover the restriction of the section to $\restr{E}{\D_0^+}$ is given by $\oplus_\sigma s_{\xbf_\sigma}$ where
\begin{align}
s_{\xbf_\sigma}(t_\sigma)=\left (t_\sigma, \frac{t_\sigma^{-1}x_\sigma -t_\sigma x'_\sigma}{\sqrt{2}} \right ) \in \R_{>0} \times F_\sigma.
\end{align}
\begin{rmk}
Since $\D_\xbf^+$ is the zero locus $s_\xbf$ in $\D^+$, the intersection $\D_0^+ \cap \D_\xbf^+$ is the zero locus of $\oplus s_{\xbf_\sigma}$ in $\D_0^+$.
\end{rmk}

\begin{prop} \label{phikmrest}
For a vector $\xbf$ in $F^2_\R$ the restriction $\restr{\varphi_{KM}(\xbf)}{\D_0^+}$ corresponds to the smooth function
\[ \omega_{os}(h(t_\8))\varphi_\8(\xbf) \in C^\8(\R_{>0}^N)^{K^0_\8} \]
in the variable $t_\8$, where $\varphi_\8(\xbf)=\prod_{\sigma} \varphi_\sigma(\xbf_\sigma)$ is in $\Scal(F_\8^2)$ and
\[\varphi_\sigma(\xbf_\sigma)=\exp \left (-\pi x_\sigma^2 -\pi {x_\sigma'}^2 \right )(x_\sigma+x_\sigma') \in \Scal(F^2_\sigma).\]
\end{prop}
\begin{proof} By the first example in \cite{rbr} we have
\[\restr{\varphi_{KM}(\xbf)}{\D_0^+}=\varphi_{KM}^{\sigma_1}(\xbf_{\sigma_1})\wedge \cdots \wedge \varphi_{KM}^{\sigma_N}(\xbf_{\sigma_N}) \]
where $\varphi_{KM}^{\sigma}$ is the Kudla-Millson form of the symmetric space $\D_\sigma$ and $\xbf_\sigma$ is the component of $\xbf$ in $F_\sigma^2$. By the second example we have
\[\varphi_{KM}^{\sigma}(\xbf_{\sigma_1})=\exp \left (-\pi \left ( \frac{x_\sigma}{t_\sigma} \right )^2 -\pi \left (x_\sigma't_\sigma \right )^2  \right )\left ( \frac{x_\sigma}{t_\sigma}+x_\sigma't_\sigma \right ) \frac{dt_\sigma}{t_\sigma}.\]
\end{proof}

Let $\varphi_\fin$ be a $K_\fin$-invariant Schwartz function in $\Scal(F_{\A_\fin}^2)$ as before. We identify $\A^2_{F,\fin}$ with $F_{\A_\fin}^2$ and view $\varphi_\fin$ in $\Scal(\A^2_{F,\fin})$. For $g$ in $\SL_2(\A_F)$ and $t$ in $\A_F^\times$ let us define a second theta series
\begin{align}
    \widetilde{\Theta}_{os}(g,t,\varphi) \coloneqq \sum_{\xbf \in F^2} \omega_{os}(g,t)\varphi(\xbf) 
\end{align}
where $\varphi \coloneqq \varphi_\8 \otimes \varphi_\fin$ is in $\Scal(\A_F^2)$. The function $\widetilde{\Theta}_{os}$ is a kernel for the pair $\SL_2(\A_F) \times \A_F^\times$. We view $\A_F^\times$ as $\SO(\A_F^2)$, so the dual pair that is involved is the ortho-symplectic pair given on Page \pageref{dualpairpage}.  After fixing $g$ we have a smooth function 
\begin{align}
   \widetilde{\Theta}_{os}(g,\cdot,\varphi) \in \left [C^\8(F^\times \backslash \A_F^\times) \right ]^{K^0(\fffrak)}
\end{align} in the variable $t$ in $\A_F^\times$. Let
\begin{align}
\restr{\Theta_{KM}(\tau, \varphi_\fin)}{M_\fffrak} \in \Omega^N(M_\fffrak)
\end{align}
be the restriction of the Kudla-Millson theta series to $M_\fffrak$. Let $g_\tau$ be the standard matrix in $\SL_2(\R)$ sending $i$ to $\tau$.

\begin{prop} \label{equiv1} The function $\widetilde{\Theta}_{os}(\iota_\Delta(g_\tau),\cdot,\varphi_\fin)$ correspond to $v^\frac{N}{2}\restr{\Theta_{KM}(\tau, \varphi_\fin)}{M_\fffrak}$ in the isomorphism \eqref{isom1}.
\end{prop}
\begin{proof} Recall that we defined $\Theta_{KM}(\tau, \varphi_\fin)=v^{-\frac{N}{2}}\Theta_{os}(g_\tau,\varphi_{KM} \otimes \varphi_\fin)$. First for $g$ in $\SL_2(\A)$ we have
\begin{align}
\restr{\Theta_{os}(g,\varphi_{KM}\otimes \varphi_\fin)}{M_\fffrak} & = \Theta_{os}(g,\restr{\varphi_{KM}}{\D_0^+}\otimes \varphi_\fin) \nonumber \\
& = \sum_{\afrak \in \Cl_\fffrak(F)^+} \Theta_{os}(g,h(t_\afrak),\restr{\varphi_{KM}}{\D_0^+}\otimes \varphi_\fin)\in \Omega^N(M_\fffrak).
\end{align}
Then, using Proposition \ref{phikmrest} we compute
\begin{align}
 \Theta_{os}(g,h(t_\afrak),\restr{\varphi_{KM}}{\D_0^+}\otimes \varphi_\fin) & = \sum_{\xbf \in F_\Q^2} \omega_{os}(g,t_\afrak)\restr{\varphi_{KM}}{\D_0^+}(\xbf) \varphi_\fin(\xbf) \nonumber  \\
 & = \sum_{\xbf \in F^2_\Q} \omega_{os}(g,h(t_\8,t_\afrak))\varphi_\8(\xbf) \varphi_\fin(\xbf)dt^\times_\8 \\
 & \stackrel{\textrm{seesaw}}{=} \sum_{\xbf \in F^2} \omega_{os}(\iota_\Delta(g),(t_\8,t_\afrak))\varphi_\8(\xbf) \varphi_\fin(\xbf)dt^\times_\8 \nonumber.
\end{align}
At $g=(g_\tau,1)$ and $t=(t_\8,t_\afrak)$ it corresponds to $\widetilde{\Theta}_{os}(\iota_\Delta(g_\tau),t,\varphi)$ in the isomorphism \eqref{isom1}.
\end{proof}

\paragraph{A few integrals.} Before passing to the regularized integral of $\Theta_{KM}(\tau,\varphi_\fin)$ on $C\otimes \psi$, we will need the following integrals for $\xbf$ in $F_\R^2$ and a complex number $s$
\begin{gather}
    J_\8(\xbf,s) \coloneqq \prod_\sigma J_\sigma(\xbf_\sigma,s), \qquad J_\sigma(\xbf_\sigma,s) \coloneqq \int_{\R^\times} \omega_{os}(t_\sigma) \varphi_\sigma(\xbf_\sigma) \psi(t_\sigma) \vert t_\sigma \vert^s dt^\times_\sigma,  \nonumber \\
    J_\8(\xbf,s)^+ \coloneqq \prod_\sigma J_\sigma(\xbf_\sigma,s)^+, \qquad J_\sigma(\xbf_\sigma,s)^+ \coloneqq \int_{\R^\times} \left \vert \omega_{os}(t_\sigma) \varphi_\sigma(\xbf_\sigma)\right \vert \vert t_\sigma \vert^s dt^\times_\sigma.
\end{gather}
Let us also define the following subsets of $F^2$:
\begin{gather}
M^+ \coloneqq  \left \{ \left .  \xbf \in F^2 \; \right \vert \;   Q(\xbf,\xbf) > 0  \right \}, \qquad M^- \coloneqq \left \{ \left . \xbf \in F^2 \; \right \vert \;  Q(\xbf,\xbf) < 0  \right \}, \qquad
M^\times \coloneqq  M^- \sqcup M^+, \nonumber \\
 l_1 \coloneqq  \left \{\left .\begin{pmatrix}
 x \\ 0
 \end{pmatrix} \in F^2 \right \vert x \neq 0 \right \}, \qquad
 l_2 \coloneqq  \left \{\left . \begin{pmatrix}
 0 \\ x'
 \end{pmatrix} \in F^2 \right \vert x' \neq 0  \right \},
\end{gather}
where $l_1$ and $l_2$ are two isotropic lines, spanned by the isotropic vectors $\ebf_1\coloneqq \transp{(1,0)}$ and $\ebf_2 \coloneqq \transp{(0,1)}$ in $F^2$. For a positive real number $\alpha$ and a complex number $s$ we define the $K$-Bessel function
\begin{align}
K_s(\alpha) \coloneqq \int_0^\8e^{-\alpha(\beta+\beta^{-1})/2}\beta^s\frac{d\beta}{\beta}.    
\end{align}

\begin{lem} \label{vanishintegral}
\begin{enumerate}
\item For $\xbf =x\ebf_1$ in $l_1$ and $\re(s)<1$ we have
    \begin{align*}J_\8(\xbf,s)= \frac{\N_{F/\Q}(x)}{ \vert \N_{F/\Q}(x) \vert ^{1-s}}\Gamma \left ( \frac{1-s}{2} \right )^N\pi^{-\frac{N(1-s)}{2}}.\end{align*}
\item For $\xbf=x'\ebf_2$ in $l_2$ and $\re(s)>-1$ we have
    \begin{align*}J_\8(\xbf,s) = \frac{\N_{F/\Q}(x')}{ \vert \N_{F/\Q}(x') \vert ^{1+s}}\Gamma \left ( \frac{1+s}{2} \right )^N\pi^{-\frac{N(1+s)}{2}}.\end{align*}
\item For $\xbf=(x,x')$ in $M^\times$ and any $s$ we have
    \begin{align*} J_\8(\xbf,s) = \lvert \N_{F/\Q}(x)\rvert^\frac{1+s}{2} & \lvert \N_{F/\Q}(x')\rvert^\frac{1-s}{2} \prod_{\sigma} \left (\sgn(x_\sigma) K_{\frac{1-s}{2}}(2\pi \lvert x'_\sigma x_\sigma \rvert) \right . \left . + \sgn(x'_\sigma)K_{\frac{1+s}{2}}(2\pi \lvert x_\sigma x'_\sigma \rvert) \right ).\end{align*}
\item For $\xbf$ in $M^-$ we have $J_\8(\xbf,0)=0$.
\end{enumerate}
\end{lem}

\begin{proof} By the definition of $\varphi_\8$ we have
\begin{align} \label{integralvanishes} J_\sigma(\xbf_\sigma,s) & \coloneqq \int_{-\8}^\8 \exp \left (-\pi \left ( \frac{x_\sigma}{t_\sigma} \right )^2 -\pi \left (x_\sigma't_\sigma \right )^2  \right )\left ( \frac{x_\sigma}{t_\sigma}+x_\sigma't_\sigma \right )\psi_\sigma(t_\sigma) \vert t_\sigma \vert ^{s} \frac{dt_\sigma}{\vert t_\sigma \vert} \nonumber \\
& = (1-\psi_\sigma(-1))\int_{0}^\8 \exp \left (-\pi \left ( \frac{x_\sigma}{t_\sigma} \right )^2 -\pi \left (x_\sigma't_\sigma \right )^2  \right )\left ( \frac{x_\sigma}{t_\sigma}+x_\sigma't_\sigma \right ) \vert t_\sigma \vert ^{s} \frac{dt_\sigma}{t_\sigma}.
\end{align}
Since $\psi$ is totally odd we have $\psi_\sigma(-1)=-1$ at every archimedean place $\sigma$ and this integral is nonzero. Note that $x$ being nonzero is equivalent to $x_\sigma$ being nonzero for every embedding $\sigma$.
\begin{enumerate}
\item Suppose that $x'$ is zero. Then for $\re(s)<1$ we have
\begin{align}
    J_\sigma(\xbf_\sigma,s) = 2\int_0^\8 \exp \left (-\pi \frac{x_\sigma^2}{t_\sigma^2} \right ) {x_\sigma}t_\sigma^{s-1} \frac{dt_\sigma}{t_\sigma} =\frac{x_\sigma}{\pi^\frac{1-s}{2} \vert x_\sigma \vert^{1-s}}\Gamma\left ( \frac{1-s}{2} \right ).
\end{align}
\item Suppose that $x$ is zero. Then for $\re(s)>-1$ we have
\begin{align}
    J_\sigma(\xbf_\sigma,s) = 2 \int_0^\8 \exp \left (-\pi {x'_\sigma}^2t^2 \right ) x'_\sigma t_\sigma^{1+s} \frac{dt_\sigma}{t_\sigma}=\frac{x'_\sigma}{\pi^\frac{1+s}{2} \vert x'_\sigma \vert^{1+s}}\Gamma\left ( \frac{1+s}{2} \right ).
\end{align}

\item \label{point3} Finally, suppose that $x_\sigma x'_\sigma$ is nonzero. Then, using the substitution $t=\sqrt{\left \vert \frac{x'_\sigma}{x_\sigma} \right \vert }u$ we have for any $s$ in $\C$
\begin{align}
    J_\sigma(\xbf_\sigma,s) & = 2 \lvert x_\sigma \rvert^\frac{1+s}{2} \lvert x'_\sigma\rvert^\frac{1-s}{2}\int_{0}^\8 e^{-\pi \lvert x'_\sigma x_\sigma \rvert \left ( u^{-2} + u^{2} \right )}\left (\sgn(x_\sigma)u^{1+s} + \sgn(x'_\sigma)u^{s-1} \right ) \frac{du}{u} \nonumber \\
    & = \lvert x_\sigma\rvert^\frac{1+s}{2} \lvert x'_\sigma\rvert^\frac{1+s}{2} \left (\sgn(x_\sigma) K_{\frac{1+s}{2}}(2\pi \lvert x_\sigma x'_\sigma \rvert)+\sgn(x'_\sigma)K_{\frac{1-s}{2}}(2\pi \lvert x_\sigma x'_\sigma \rvert) \right ).
\end{align}
 We made use of the substitution $v=u^2$ and the fact that $K_{-s}=K_s$.
 
 \item Since $M^-$ is a subset of $M^\times$ we have by \ref{point3}. that
 \begin{align} \label{integral1} J_\8(\xbf,0) =\prod_{\sigma }\left (\sgn(x_\sigma)+\sgn(x'_\sigma) \right ) \sqrt{\lvert x_\sigma x'_\sigma\rvert } K_{\frac{1}{2}}(2\pi \lvert x_\sigma x'_\sigma \rvert).
 \end{align}
This can only be nonzero if $\sgn(x_\sigma x'_\sigma)=1$ for all $\sigma$, since $x_\sigma x'_\sigma$ is nonzero. This implies that
\begin{align}
    Q(\xbf,\xbf)=\Tr_{F/\Q}( xx') = \sum_{\sigma} x_\sigma x'_\sigma = \sum_{\sigma } \lvert x_\sigma x'_\sigma \rvert  \geq 0.
\end{align}
 \end{enumerate}
 \end{proof}

\paragraph{Regularization of the integral.} We define the integral
\begin{align}
    I(\tau_1, \dots, \tau_N,\varphi,\psi,s) \coloneqq \frac{(v_1 \cdots v_N)^{-\frac{1}{2}}}{\vol^\times(K^0_\8)}\int_{F^\times \backslash \A_F^\times} \widetilde{\Theta}_{os}(g_{\underline{\tau}},t,\varphi)\psi(t) \lvert t \rvert^s dt^\times
\end{align}
where $\underline{\tau}=\underline{u}+\underline{i}\underline{v}=(\tau_1, \dots, \tau_N)$ in $\HH^N$. In this section we want to show that the integral $I(\tau_1, \dots, \tau_N,\varphi,\psi,s)$ converges absolutely in a regularized way. Note that for the diagonal restriction to $(\tau, \dots, \tau)$ we have
\begin{align}
    I(\tau,\varphi,\psi,s) \coloneqq  I(\tau, \dots,\tau,\varphi,\psi,s) = \frac{v^{-\frac{N}{2}}}{\vol^\times(K^0_\8)}\int_{F^\times \backslash \A_F^\times} \widetilde{\Theta}_{os}(\iota_\Delta(g_{\tau}),t,\varphi)\psi(t) \lvert t \rvert^s dt^\times.
\end{align}
\begin{rmk} \label{rmkodd} It follows from \eqref{integralvanishes} that the integral vanishes if for one place we have $\psi_\sigma(-1)=1$ {\em i.e.} if $\psi$ is not totally odd. 
\end{rmk} 
\noindent Since the trace form is non degenerate we can write $F^2-(0,0)$ as a disjoint union $l_1 \sqcup l_2 \sqcup M^\times $. Hence we can also split the theta series $ \widetilde{\Theta}_{os}(g,t,\varphi)$ as a sum of three terms
\begin{align}\label{splitsum}
 \widetilde{\Theta}_{os}(g,t,\varphi)= \widetilde{\Theta}^{(l_1)}_{os}(g,t,\varphi)+\widetilde{\Theta}^{(l_2)}_{os}(g,t,\varphi)+\widetilde{\Theta}^{\times}_{os}(g,t,\varphi)
 \end{align}
 where we restrict the summation to the sets above. In view of the next proposition we call the first two terms $\widetilde{\Theta}^{(l_i)}_{os}(g,t,\varphi)$ the {\em singular terms} and the third one the {\em regular term}. The integral $I(\tau_1, \dots, \tau_N,\varphi,\psi,s)$ can also be written as a sum
 \begin{align}
     I(\tau_1, \dots, \tau_N,\varphi,\psi,s) = I^{(l_1)}(\tau_1, \dots, \tau_N,\varphi,\psi,s)+I^{(l_2)}(\tau_1, \dots, \tau_N,\varphi,\psi,s)+I^\times(\tau_1, \dots, \tau_N,\varphi,\psi,s)
 \end{align}
 The following proposition shows that the two first integrals converge on two disjoint domains. However, the condition of vanishing of $\varphi_1$ or $\varphi_2$ kills one of these two terms. We will then be able to have a meromorphic continuation to the whole plane which allows to define the integrals for every $s$ in $\C$.

\begin{prop} \label{convergence} The regular term converges for every $s$. The singular terms converge for $\re(s)<-1$ on $l_1$ and $\re(s)>1$ on $l_2$.
\end{prop}

\begin{proof}
We want to show that the integrals $I^{(l_i)}$ and $I^{\times}$ converge termwise absolutely on the corresponding domains. First, note that after taking the absolute value, the action of $\omega(g_{\tau_k})$ is simply rescaling by $\sqrt{v_k}$. Hence we can assume $\tau_k=i$ for showing the termwise absolut convergence. Let $H^0(\8) \coloneqq (\R^\times)^N \times \widehat{U}(\fffrak)$. By the decomposition \eqref{decomp2} we have a surjection 
\begin{align} \label{asurjection} \bigsqcup_{[\afrak] \in \Cl_\fffrak(F)^+} t_\afrak H^0(\8) \longrightarrow F^\times \backslash \A_F^\times.
\end{align}
Let $\nu \coloneqq \re(s)$ and let $\MMcal$ be $M^\times$ or one of the two isotropic lines $l_1$ and $l_2$. Recall that $\psi$ is unitary. Using the surjection \eqref{asurjection} we get the inequality
\begin{align} \label{ineq1}
 \int_{F^\times \backslash \A_F^\times} \sum_{ \xbf \in \MMcal} \left \vert \omega_{os}(t)\varphi(\xbf) \psi(t)\right \vert \lvert t \rvert^{\nu} dt^\times & \leq \sum_{[\afrak] \in \Cl_\fffrak(F)^+} \int_{t_\afrak H^0(\8)} \sum_{ \xbf \in \MMcal} \left \vert \omega_{os}(t)\varphi(\xbf) \right \vert \lvert t \rvert^{\nu} dt^\times \nonumber \\
& =  \sum_{[\afrak] \in \Cl_\fffrak(F)^+}  \lvert t_\afrak \rvert^{\nu} \int_{H^0(\8)} \sum_{ \xbf \in \MMcal} \left \vert \omega_{os}(tt_\afrak)\varphi(\xbf) \right \vert \lvert t \rvert^{\nu} dt^\times.
\end{align}
The result does not depend on the finite Schwartz function, so we can replace $\varphi_\fin$ by $\omega_{os}(t_\afrak)\varphi_\fin$. So we have to show the convergence of
\begin{align}
    \int_{H^0(\8)} \sum_{ \xbf \in \MMcal} \left \vert \omega_{os}(t)\varphi(\xbf) \right \vert \lvert t \rvert^{\nu} dt^\times,
\end{align}
which is equivalent to the convergence of
\begin{align}
     \sum_{ \xbf \in \MMcal} \int_{H^0(\8)} \left \vert \omega_{os}(t)\varphi(\xbf) \right \vert \lvert t \rvert^{\nu} dt^\times.
\end{align}
Since $\widehat{U}(\fffrak)$ is contained in $\widehat{\Ocal}^\times$ we can bound the previous integral
\begin{align} \label{inter1}
     \sum_{ \xbf \in \MMcal} \int_{H^0(\8)} \left \vert \omega_{os}(t)\varphi(\xbf) \right \vert \lvert t \rvert^{\nu} dt^\times \leq \sum_{\xbf \in \MMcal} J_\8(\xbf,\nu)^+ J_\fin(\xbf,\nu)^+,
\end{align}
where
\begin{align}
J_\fin(\xbf,\nu)^+ & =  \int_{\widehat{\Ocal}^\times} \left \vert \omega_{os}(t_\fin)\varphi(\xbf) \right \vert \lvert t_\fin \rvert^{\nu} dt_\fin^\times. \end{align}
 Suppose\footnote{We always have $\supp(\varphi_\fin)$ is contained in $m\widehat{\Ocal}^2$ for some $m$ in $F^\times$.} that $\supp(\varphi_\fin)$ is contained $\widehat{\Ocal}^2$, let $C_w \coloneqq \sup_{\xbf \in F_w^2} \vert \varphi_w(\xbf)\vert$ and $C \coloneqq \prod_{w<\8} C_w$. Note that $C_w$ is one for almost all places $w$ of $F$. Then
 \begin{align}
     J_\fin(\xbf,\nu)^+ & \leq C \int_{\widehat{\Ocal}^\times} \left \vert \id_{\widehat{\Ocal}}(t_\fin^{-1}x)\id_{\widehat{\Ocal}}(t_\fin x') \right \vert \lvert t_\fin \rvert^{\nu} dt^\times_\fin  = C \vol^\times(\widehat{\Ocal}^\times) \id_{\widehat{\Ocal}^2}(\xbf).
 \end{align} Hence we can bound \eqref{inter1} by
 \begin{align} \label{inter2}
    \sum_{\xbf \in \MMcal} J_\8(\xbf,\nu)^+ J_\fin(\xbf,\nu)^+ \leq C \vol^\times(\widehat{\Ocal}^\times) \sum_{\xbf \in \MMcal \cap \widehat{\Ocal}^2_{F}} J_\8(\xbf,\nu)^+.
\end{align}

\begin{enumerate}
    \item Suppose that $\MMcal=M^\times$. Following the proof of Lemma \ref{vanishintegral} we see that
    \begin{align} 
    J_\8(\xbf,\nu)^+  \leq \lvert \N(x)\rvert^\frac{1+\nu}{2} \lvert \N(x')\rvert^\frac{1-\nu}{2} \prod_{\sigma} \left ( K_{\frac{1-\nu}{2}}(2\pi \lvert x'_\sigma x_\sigma \rvert)+K_{\frac{1+\nu}{2}}(2\pi \lvert x_\sigma x'_\sigma \rvert) \right ).
    \end{align}
For $\nu$ real we have the bound
    \begin{align}K_\nu(\alpha)<2^{2(\vert \nu \vert +1)} \left (1 + \frac{\Gamma(\vert \nu \vert +1)}{\alpha^{\vert \nu \vert+1}} \right )e^{-\alpha},\end{align}
see \cite[Lemma.~1.3]{sullivan} for example. Thus we can bound
\begin{align} \label{bound1}
    J_\8(\xbf,\nu)^+ \leq G(\xbf) e^{-\sum x_\sigma x'_\sigma}
\end{align}
where $G(\xbf)$ is a rational function in $\xbf$. Since we sum over $M^\times$, none of the $x_\sigma$ and $x'_\sigma$ are zero. Hence the right hand side of \eqref{bound1} is now rapidly decreasing on $M^\times$ and the sum \eqref{inter2} converges.
 
 \item Now suppose that $\MMcal=l_2$. Following the proof of Lemma \ref{vanishintegral} we can bound \eqref{inter2} by
 \begin{align}
\sum_{\xbf \in l_2} J_\8(\xbf,\nu)^+ \leq  \Gamma\left ( \frac{1+\nu}{2} \right )^N\pi^{-N\frac{1+\nu}{2}} \sum_{x \in \Ocal^\ast}\frac{1}{\lvert \N(x)\vert^{\nu}}, \nonumber
 \end{align}
which converges for $\nu>1$.
 \item The case $\MMcal=l_1$ is similar.
\end{enumerate}
\end{proof}
 
\noindent We recall that $\Lambda$ was defined by $\Lambda(s) = \Gamma\left (\frac{1+s}{2} \right)^N\pi^{-\frac{N(1+s)}{2}}$.

\begin{prop} \label{singularterms} For the singular terms we have
\begin{align}
    I^{(l_1)}(\tau_1, \dots, \tau_N,\varphi,\psi,s) & = 2^{1-N}\Lambda(-s)(v_1 \cdots v_N)^{\frac{s}{2}}\zeta_\fin(\varphi_1,\psi^{-1},-s) \nonumber \\
    I^{(l_2)}(\tau_1, \dots, \tau_N,\varphi,\psi,s) & = 2^{1-N}\Lambda(s)(v_1 \cdots v_N)^{-\frac{s}{2}}\zeta_\fin(\varphi_{2},\psi,s), \nonumber
\end{align}
where $\varphi_{1}$ and $\varphi_{2}$ are the Schwartz functions in $\Scal(\A_F)$ obtained by restricting $\varphi_\fin$ to the isotropic lines $l_1$ and $l_2$. In particular they both have a continuation to the whole plane.
\end{prop}
\begin{proof}
For the singular term $I^{(l_1)}(\tau_1, \dots, \tau_N ,\varphi,\psi,s)$ we know that we have absolute convergence for every $\re(s)<-1$. Hence we can exchange summation and integration and unfold the integral to get
\begin{align} 
 \int_{F^\times \backslash \A_F^\times} \sum_{x \in F^\times} \omega_{os}(g_{\underline{\tau}},t)\varphi(x\ebf_1) \psi(t) \lvert t \rvert^s dt^\times  \nonumber & = (v_1 \cdots v_N)^\frac{1}{2} \int_{F^\times \backslash \A_F^\times} \sum_{x \in F^\times} \varphi_1(\sqrt{\underline{v}}t^{-1}x) \psi(t) \lvert t \rvert^s dt^\times  \nonumber \\
  & = (v_1 \cdots v_N)^\frac{1}{2} \int_{\A_F^\times} \varphi_1(\sqrt{\underline{v}}t^{-1}x) \psi(t) \lvert t \rvert^s dt^\times.
\end{align}
After the substitutions $ z_\8=\sqrt{\underline{v}}t_\8^{-1}$ and $z_\fin=t_\fin^{-1}$ we get
\begin{align} 
 & = (v_1 \cdots v_N)^{ \frac{1}{2}+\frac{s}{2}} \int_{\A_F^\times} \varphi_1(z) \psi(z)^{-1} \lvert z \rvert^{-s}  dz^\times \nonumber \\
 & = (v_1 \cdots v_N)^{ \frac{1}{2}+\frac{s}{2}} \zeta(\varphi_{1},\psi^{-1},-s),
\end{align}
where $z=(z_\8,z_\fin)$.
The result follows after multiplying by $\vol^\times(K^0_\8) ^{-1} (v_1 \cdots v_N)^{-\frac{1}{2}}$ and using the computation \eqref{zeta1} of the archimedean zeta function. Recall that $\vol^\times(K^0_\8)=2^{N-1}$. For the second term the computation is similar.

Since we supposed that the character is unitary and odd, its restriction to the adèles of norm $1$ is never trivial. Hence by \cite[Proposition.~3.16]{bump}, the continued function is entire.
\end{proof}

\begin{prop} \label{Ifourier} Suppose that $\varphi_1$ or $\varphi_2$ vanishes. The integral $I(\tau,\varphi,\psi,s)$ can be analytically continued to the whole plane. The value $I(\tau,\varphi,\psi) \coloneqq \restr{I(\tau, \varphi,\psi,s)}{s=0}$ at $s=0$ is
\[I(\tau, \varphi,\psi)=2^{1-N}\zeta_\fin(\varphi_1,\psi^{-1},0)+2^{1-N}\zeta_\fin(\varphi_2,\psi,0)+\sum_{n \in \Q_{>0}} \left ( \int_{C \otimes \psi} \Theta_n(v,\varphi_\fin) \right ) e^{2i\pi n \tau}.\]
\end{prop}

\begin{proof}By Proposition \ref{convergence} the singular terms converge on two disjoint half planes. Hence the assumption that $\varphi_1$ or $\varphi_2$ vanishes guarantees that one of the singular terms is zero.

As in \eqref{splitsum} let us write the theta series as the sum
\begin{align}
 I(\tau,\varphi,\psi,s) & = I^-(\tau,\varphi,\psi,s) + I^{(l_1)}(\tau,\varphi,\psi,s)+I^{(l_2)}(\tau,\varphi,\psi,s)+I^+(\tau,\varphi,\psi,s)\end{align}
 where the $+$ (respectively the $-$) means that we sum over positive (respectively negative) vectors. In particular $I^\times=I^-+I^+$.

For the term $I^-(\tau,\varphi,\psi,s)$ we know that we have absolute convergence for every $s$ in $\C$. Hence we can exchange summation and integration at $s=0$ and we get the bound (as in \eqref{ineq1})
  
 \begin{align} \label{splitsumneg}
 I^-(\tau,\varphi,\psi)
&  = v^{-\frac{N}{2}} 2^{1-N} \int_{F^\times \backslash \A_F^\times} \widetilde{\Theta}^-_{os}(\iota_\Delta(g_{\tau}),t,\varphi)\psi(t) dt^\times  \nonumber\\
& \leq v^{-\frac{N}{2}} 2^{1-N}\sum_{\afrak \in \Cl_\fffrak(F)^+} \sum_{ \xbf \in M^-}  \int_{ t_\afrak H^0(\8)}\omega_{os}(\iota_\Delta(g_\tau),t)\varphi(\xbf) \psi(t) dt^\times  \nonumber\\
& = 2^{1-N}\sum_{\afrak \in \Cl_\fffrak(F)^+} \sum_{ \xbf \in M^-}  e^{i\pi u Q(\xbf,\xbf)} J_\8(\sqrt{v}\xbf,0) \int_{t_\afrak\widehat{U}(\fffrak)}\omega_{os}(t)\varphi_\fin(\sqrt{v}\xbf) \psi_\fin(t) dt^\times.
 \end{align}
The vanishing of the integral \eqref{splitsumneg} follows from Lemma \ref{vanishintegral}.

For the term $I^+(\tau,\varphi,\psi,s)$ we also have absolute convergence for every $s$ in $\C$. Thus, using Proposition \ref{equiv1}  we can write
\begin{align}
I^+(\tau,\varphi,\psi) & =v^{-\frac{N}{2}}2^{1-N}\int_{F^\times \backslash \A_F^\times} \widetilde{\Theta}^+_{os}(\iota_\Delta(g_{\tau}),t,\varphi)\psi(t) dt^\times \nonumber \\
& = \sum_{[\afrak] \in \Cl_\fffrak(F)^+} \psi(\afrak) v^{-\frac{N}{2}}\int_{\Gamma_\afrak \backslash \D_0^+} \widetilde{\Theta}^+_{os}(\iota_\Delta(g_{\tau}),(t_\8,t_\afrak),\varphi) dt^\times_\8 \nonumber \\
& = \int_{C \otimes \psi} \Theta^+_{KM}(\tau,\varphi_\fin),
\end{align}
where $\Theta_{KM}^+(\tau,\varphi_\fin)$ is the restriction to $\Theta_{KM}(\tau,\varphi_\fin)$ to positive vectors. By Lemma \ref{fourierinter} we then have
\begin{align}\int_{C \otimes \psi} \Theta^+_{KM}(\tau,\varphi_\fin) = \sum_{n \in \Q_{>0}}\left ( \int_{C \otimes \psi} \Theta_n(v,\varphi_\fin) \right )e^{2i \pi n \tau}.
\end{align}
Finally, the constant term of $I(\tau,\varphi,\psi)$ is $    I^{(l_1)}(\tau,\varphi,\psi,0)+I^{(l_2)}(\tau,\varphi,\psi,0)
$ and was computed in Proposition \ref{singularterms}. In particular it was shown that $I^{(l_1)}(\tau,\varphi,\psi,s)$ and $I^{(l_2)}(\tau,\varphi,\psi,s)$ have an analytic continuation to the whole plane.
\end{proof}

\paragraph{Orientations.} Before showing that the Fourier coefficients are intersection numbers, we need to fix some orientations. We fixed the orientations $o(F^2_\R) \coloneqq \ebf_{\sigma_1} \wedge \cdots \wedge \ebf_{\sigma_N} \wedge \fbf_{\sigma_1} \wedge \cdots \wedge \fbf_{\sigma_N}$ and $o(z_0) \coloneqq \fbf_{\sigma_1} \wedge \cdots \wedge \fbf_{\sigma_N}$ of $F^2_\R$ and of $z_0$, and we can use it to orient $\D^+$ and $\D_\xbf^+$ as in Subsection \ref{sectionlss}. Let us also orient $\D_0^+$ at $z_0$. The differential at $(1, \dots, 1)$ of the immersion $h$ is an embedding 
\begin{align}
dh \colon \R^N \longrightarrow T_{z_0} \D_0^+ \subset z_0^\vee \otimes z_0^\perp,
\end{align}
and we use the standard orientation $o(\R^N)=\frac{\partial}{\partial t_1} \wedge \cdots \wedge \frac{\partial}{\partial t_N}$ to orient the tangent space $T_{z_0}\D_0^+$. Hence orienting $\D_0^+$ amounts to a choice of an ordering of the places $\sigma_k$.
\begin{lem} We have $dh \left ( \frac{\partial}{\partial t_i}\right )=\fbf_i^\vee \otimes \ebf_i$.
\end{lem}
\begin{proof} Without loss of generality we take $i=1$. Let $\rho \colon \R \longrightarrow \R_{>0}^N$ be the curve $\rho(u)=(e^u, 1 \cdots,1) $ representing $\frac{\partial}{\partial t_1}$. We have
\begin{align}
h(\rho(u))z_0&= h(\gamma(u)) \Span \bigl < \fbf_1, \dots , \fbf_N \bigr > \nonumber \\
& = \Span  \bigl < \frac{e^u-e^{-u}}{2} \ebf_1+\frac{e^u+e^{-u}}{2} \fbf_1, \fbf_2, \dots, \fbf_N \bigr> \nonumber \\
& = \Span  \bigl < \frac{e^u-e^{-u}}{e^u+e^{-u}} \ebf_1+ \fbf_1, \fbf_2, \dots, \fbf_N \bigr> \eqqcolon z(u).
\end{align}
As an element of $z_0^\vee \otimes z_0^\perp$ we have $    z(u)= \frac{e^u-e^{-u}}{e^u+e^{-u}} \fbf_1^\vee \otimes \ebf_1.
$ Then $dh \left ( \frac{\partial}{\partial t} \right)$ is equal to
\begin{align}
    \restr{\frac{d}{du} z(u)}{u=0} & = \fbf_1^\vee \otimes \ebf_1.
\end{align}
\end{proof}

\paragraph{The positive Fourier coefficients as intersection numbers.} In this subsection we show that the Fourier coefficients $\int_{C \otimes \psi} \Theta_n(v,\varphi_\fin)$ are intersections numbers, it is the content of Proposition \ref{intertop}.
\begin{prop}
Let $\xbf$ be a positive vector in $F^2$. Then
\[ \vert \D_0^+ \cap  \D_\xbf^+ \vert =
\begin{cases}
1 & \textrm{if} \; \sgn(x_\sigma x'_\sigma)=1 \; \; \textrm{for all} \; \; \sigma \\
0 & \textrm{otherwise},
\end{cases}
\]
and it only depends on the $\Ocal^{\times,+}$ orbit of $\xbf$, where we view $\Ocal^{\times,+}$ in $\Gamma'_{h_i}$ embedded by the map $h \colon F_\8^\times \hooklongrightarrow H(\R)^+$. Furthermore the intersection in $\D^+$is transversal and we have
\[ \bigl < \D_\xbf^+ ,  \D_0^+ \bigr >_{\D^+} =
\begin{cases}
(-1)^N \sgn \N(x) & \textrm{if} \; \sgn(x_\sigma x'_\sigma)=1 \; \; \textrm{for all} \; \; \sigma \\
0 & \textrm{otherwise}.
\end{cases}
\]
\end{prop}
\begin{proof}
The set $\D_0^+ \cap \D_\xbf^+$ is precisely the zero locus of $\oplus_\sigma s_{\xbf_\sigma} \colon \D_0^+ \longrightarrow \D_0^+ \times F_\8$, where
\begin{align}
s_{\xbf_\sigma}(t_\sigma)=\left (t_\sigma, \frac{t_\sigma^{-1}x_\sigma -t_\sigma x'_\sigma}{\sqrt{2}} \right ) \in \R_{>0} \times F_\sigma,
\end{align}
see the diagram \eqref{diagram}. This section vanishes exactly when $t_\sigma^2=\frac{x_\sigma}{x_\sigma'}$ for all $\sigma$. This has precisely one solution when $\sgn(x_\sigma x'_\sigma)=1$ for all $\sigma$.

Suppose that the intersection point is $z_0$. We have $o(N_{z_0}\D_\xbf^+)=(\fbf_1^\vee \otimes \xbf) \wedge \cdots \wedge (\fbf_N^\vee \otimes \xbf)$ and $o(T_{z_0}\D_0^+)=(\fbf_1^\vee \otimes \ebf_1) \wedge \cdots \wedge (\fbf_N^\vee \otimes \ebf_N)$. Since $z_0$ is in $\D_\xbf^+$ we  have $\xbf$ is in $z_0^\perp$. Hence we can write $\xbf= \transp{(x,x)} = \sum_j x_j \ebf_j$ and
\begin{align}
    q_{z_0} (o(T_{z_0}\D_0^+),o(N_{z_0}\D_\xbf^+) \bigr ) & = \det \left ( -Q(\fbf^\vee_i,\fbf_j^\vee)Q(\ebf_i,\xbf) \right )_{ij} \nonumber \\
    & = (-1)^N 2^N  \prod_{j} x_j = (-1)^N 2^N \N(x),
\end{align}
where the factor $2^N$ comes from $Q(\fbf^\vee_i,\fbf_j^\vee)=2\delta_{ij}$. Hence the intersection is transversal and
\begin{align}
  \bigl < \D_{\xbf}^+,\D_0^+ \bigr >_{\D^+} = (-1)^N \sgn \N(x).
\end{align}

\noindent Now suppose that the intersection point is $z \neq z_0$. Then for some $t_\8=(t_1, \dots, t_N)$ with $t_j$ positive we have $z=h(t_\8)z_0$ and thus
\begin{align}
   z_0 = h(t_\8)^{-1} \left (  \D_0^+ \cap \D_\xbf^+ \right ) = \D_0^+ \cap \D_{h(t_\8)^{-1}\xbf}^+.
\end{align}
The intersection at $z_0$ is transversal and since $h(t_\8)$ is orientation preserving (because $t_j$ is positive for every $j$) the sign of $\bigl < \D_{\xbf}^+,\D_0^+ \bigr >_{\D^+}$ equals the sign of
\begin{align}
  \bigl < \D_{h(t_\8)^{-1}\xbf}^+,\D_0^+ \bigr >_{\D^+} = (-1)^N \sgn \prod t_j^{-1} x_j = (-1)^N \sgn \N(x).
\end{align}
\end{proof}

\noindent Consider the set $\Lcal_{h_i} = \left \{ \left . \xbf \in F_\Q^2 \right \vert \varphi_\fin(h_i^{-1}\xbf) \neq 0 \right \}$. There exists $m$ in $F^\times$ such that $\supp(\varphi_\fin)$ is contained in $m \widehat{\Ocal}^2$, hence there is an $m_i$ in $F^\times$ so that $\Lcal_{h_i}$ is contained $m_i \Ocal^2$. Let $\Gamma'_{h_i}$ be the subgroup that appeared in \eqref{adeless}.

\begin{prop} \label{intersectionD}  Let $\xbf$ be in $\Lcal_{h_i}$. The intersection $\D_\afrak^+ \cap  \D_\ybf^+$ in $\D^+$ is nonzero for only finitely many orbits $\Ocal^{\times,+} \ybf$ in $ \Ocal^{\times,+}\backslash \Gamma'_{h_i} \xbf$.
\end{prop}

\begin{proof} Let $Q(\xbf,\xbf) = 2n$. For every $\ybf$ in $\Gamma'_{h_i} \xbf$ we have $Q(\ybf,\ybf)=2n$. Since $\D_\afrak^+ \cap \D_\ybf^+= h_\afrak \left ( \D_0^+ \cap \D_{h_\afrak^{-1}\ybf}^+ \right )$ it is enough to prove the statement for the intersection $\D_0^+ \cap \D_{\ybf}^+$. We begin by showing that if $\Kcal$ is a compact subset of $\D^+$ then $\Kcal \cap \D_\ybf^+$ is non-empty for only finitely many $\ybf$ in $\Gamma'_{h_i} \xbf$. For a negative plane $z$ in $\D^+$ and a vector $\ybf$ in $F^2$ we have
\begin{align}
    Q(\ybf,\ybf)=Q(\ybf_{z^\perp},\ybf_{z^\perp})+Q(\ybf_z,\ybf_z)
\end{align}
where $\ybf_z + \ybf_{z^\perp}$ is the splitting of $\ybf$ with respect to the orthogonal decomposition $F_\R^2=z \oplus z^\perp$. The Siegel majorant
\begin{align}
    Q_z^+(\ybf,\ybf) \coloneqq Q(\ybf_{z^\perp},\ybf_{z^\perp})-Q(\ybf_z,\ybf_z)
\end{align}
is a positive definite quadratic form. If $z$ is in $\D^+_\ybf$ then $z$ is contained in $\ybf^\perp$ and $Q_z^+(\ybf,\ybf)=Q(\ybf,\ybf)=2n$. Since the Siegel majorant is positive definite we can find real numbers $M_z>m_z>0$ such that $m_z \lVert \ybf \rVert^2 \leq Q_z^+(\ybf,\ybf) \leq M_z \lVert \ybf \rVert^2$
where $\lVert - \rVert$ is the Euclidean norm on $\R^{2N}$. Since $\Kcal$ is compact and $Q^+_z$ is continuous in $z$ we can also find constants $M_\Kcal>m_\Kcal>0$ such that $m_\Kcal \lVert \ybf \rVert^2 \leq Q_z^+(\ybf,\ybf) \leq M_\Kcal \lVert \ybf \rVert^2$ for every $z$ in $\Kcal$. Hence for $z$ in $\D^+_\ybf \cap \Kcal $ we have
\begin{align}
    \frac{2n}{M_\Kcal} = \frac{Q_z^+(\ybf,\ybf)}{M_\Kcal} \leq \lVert \ybf \rVert^2 \leq \frac{Q_z^+(\ybf,\ybf)}{m_\Kcal}= \frac{2n}{m_\Kcal}.
\end{align}
There are only finitely many vectors $\ybf$ of bounded norm in the lattice $\Lcal_{h_i}$, let alone in an orbit $\Gamma_{h_i}' \xbf$.

\noindent Now let $\D_0^{+,1} \subset \D_0^+$ be the subset of elements of norm $1$, so that we have a diffeomorphism
\begin{align}
 \D_0^+ & \longrightarrow  \D_0^{+,1} \times \R_{>0} \nonumber \\
 (t_{\sigma_1}, \dots, t_{\sigma_N}) & \longmapsto \left [ \left (t_{\sigma_1}, \dots, t_{\sigma_{N-1}}, \frac{1}{t_{\sigma_1} \cdots t_{\sigma_{N-1}}} \right ),t_{\sigma_1} \cdots t_{\sigma_N} \right ]
\end{align}
For $T>1$ we define
\begin{align}
 \D_0^+(T) \coloneqq \D_0^{+,1} \times \left [\frac{1}{T},T\right ] \simeq \left \{ (t_1, \dots,t_N) \in \R_{>0}^N \left \vert \frac{1}{T} \leq t_1 \cdots t_N \leq T\right . \right \} .
\end{align}
The group $\Ocal^{\times,+}$ preserves $\D_0^{+,1}$, and by Dirichlet's unit Theorem the quotient is compact. Hence we can find a fundamental domain $\Fcal \subset \D_0^{+,1}$ such that $\overline{\Fcal}$ is compact and
\begin{align}
    \D_0^{+,1} = \bigsqcup_{\lambda \in \Ocal^{\times,+}} \lambda \Fcal.
\end{align}
We also set $\Fcal_T \coloneqq \Fcal \times \left [\frac{1}{T},T\right ] $ so that
\begin{align}
    \D_0^{+}(T) = \bigsqcup_{\lambda \in \Ocal^{\times,+}} \lambda \Fcal_T.
\end{align}
Suppose that there is $T>1$ that only depends on $n$ such that
\begin{align} \label{stepinproof}
    \D_0^+ \cap \D_{\ybf}^+= \D_0^+(T) \cap \D_{\ybf}^+.
\end{align}
Since $\overline{\Fcal_T}$ is compact there are finitely many distinct vectors $\ybf_1, \dots, \ybf_k$ such that  $\vert \Fcal_T \cap \D_{\ybf_i}^+ \vert $ is non-zero. For $\ybf$ in $F^2$ we then have
\begin{align}
    \D_0^+ \cap \D_\ybf^+  & =  \D_0^+(T) \cap \D_{\ybf}^+  \nonumber \\
     & = \bigsqcup_{\lambda \in \Ocal^{\times,+}} \vert \lambda \Fcal_T \cap \D_\ybf^+ \vert \nonumber \\
     & = \bigsqcup_{\lambda \in \Ocal^{\times,+}} \lambda \left ( \Fcal_T \cap \D_{\lambda^{-1}\ybf}^+ \right ).
     \end{align}
This intersection is empty if $\ybf$ lies in none of the orbits $\Ocal^{\times,+} \ybf_i$.

Finally  let us prove \eqref{stepinproof}. Let $\ybf = \begin{pmatrix} y \\ y' \end{pmatrix}$ be in $\Lcal_{h_i}$ with $Q(\ybf,\ybf)=2\Tr_{F/\Q}(yy')=2n$. If $t_\8$ is in $\D_0^+ \cap \D_\ybf^+$, then $y_\sigma y_\sigma'$ is positive and
\begin{align}
    \prod_\sigma t_\sigma= \prod_\sigma \sqrt{\frac{y_\sigma}{y_\sigma'}}=\sqrt{\left \vert \frac{\N(y)}{\N(y')} \right \vert}.
\end{align}
Since $\Lcal_{h_i}$ is contained in $m_i \Ocal^2$ for some $m_i$ in $F^\times$ we have $\vert \N(y) \vert \geq \vert \N(m_i) \vert>0$. On the other hand, since $y_\sigma y_\sigma'$ is positive we can use the inequality between arithmetic and geometric mean to show that
\begin{align}
    \frac{2n}{N}=\frac{Q(\ybf,\ybf)}{N} \geq 2\vert \N(y)\N(y') \vert^\frac{1}{N} \geq 2 \vert \N(y) \N(m_i) \vert^\frac{1}{N}.
\end{align}
Hence 
\begin{align}
    \frac{1}{\vert \N(m_i) \vert} \left ( \frac{n}{N} \right )^N \geq \vert \N(y) \vert \geq \vert \N(m_i) \vert.
\end{align}
By replacing $y$ by $y'$ we obtain the same bound for $\N(y')$, which shows that we can take 
\begin{align}
    T = \frac{1}{\vert \N(m_i) \vert} \left ( \frac{n}{N} \right )^\frac{N}{2}.
\end{align}
\end{proof}
\noindent Every connected cycle $C_\afrak$ land in exactly one connected component $M_{h_i}$ of $M_K$. We can then define the intersection numbers in $M_{h_i}$
\begin{align} \label{interdef}
    \bigl < C_\xbf(h_i),C_\afrak \bigr >_{M_{h_i}} \coloneqq \frac{1}{\kappa}\sum_{[\ybf] \in \Ocal^{\times,+}\backslash \Gamma'_{h_i} \xbf} \bigl < \D_\ybf^+,\D_\afrak^+ \bigr >_{\D^+},
\end{align}
which are well-defined by the previous proposition. We also define

 \begin{align} \label{noncompactinter}
       \bigl < C_n(\varphi_\fin), C \otimes \psi \bigr >_{M_K} \coloneqq \sum_{[\afrak] \in \Cl_\fffrak(F)^+}\sum_{i=1}^r \sum_{\substack{\xbf \in \Gamma'_{h_i} \backslash F^2_\Q\\ Q(\xbf,\xbf)=2n}} \psi(\afrak)\varphi_\fin(h_i^{-1}\xbf) \bigl < C_\xbf(h_i),C_\afrak \bigr >_{M_{h_i}},
 \end{align}
where for every class $[\afrak]$ the intersection number $\bigl < C_\xbf(h_i),C_\afrak \bigr >$ is non-zero for at most one $h_i$, corresponding to the connected component of $M_K$ that contains $C_\afrak$.

\begin{rmk} \label{rmkkappa} Suppose that $\Gamma_{h_i}'$ contains $-1$, thus $\kappa=2$. Since $\Ocal^{\times,+}$ does not contain, every orbit appears twice, we sum over $\ybf=\Ocal^{\times,+}\xbf$ and $\ybf=-\Ocal^{\times,+}\xbf$ in \eqref{interdef}. Since $\D^+_{-\ybf}=\D^+_\ybf$ we are counting the intersection number $\bigl < \D_\ybf^+,\D_\afrak^+ \bigr >_{\D^+}$ twice and this is why we have to multiply by $\kappa=2$ in the following proposition.
\end{rmk}
\begin{prop} \label{intertop} We have
\begin{align}
    \int_{C \otimes \psi} \Theta_n(v,\varphi_\fin)= (-1)^N \kappa \bigl < C_n(\varphi_\fin),C \otimes \psi \bigr >_{M_K}.
\end{align}
\end{prop}
\begin{proof} It is enough to show that
\begin{align}
    \int_{C_\afrak} \Theta_n(v,\varphi_\fin)= (-1)^N\kappa \bigl  < C_n(\varphi_\fin),C_\afrak \bigr >_{M_K}.
\end{align}
For $\xbf$ a positive vector in $F^2$ we first show that
\begin{align} \label{stepinproof11}
    \int_{\D_0^+} \varphi^0(\xbf)= (-1)^N \bigl < \D_\xbf^+,\D_0^+ \bigr >_{\D^+}.
\end{align}
We have 
\begin{align} \label{expr1}
    \int_{\D_0^+} \varphi^0(\xbf) & = 2^{-N} e^{\pi Q(\xbf,\xbf)} J_\8(\xbf,0) \nonumber \\
    & = 2^{-N}\sqrt{\N(xx')}e^{\pi Q(\xbf,\xbf)}\prod_{\sigma }\left (\sgn(x_\sigma)+\sgn(x'_\sigma) \right )   K_{\frac{1}{2}}(2\pi \lvert x_\sigma x'_\sigma \rvert). \nonumber \\
    & \qquad
\end{align}
It follows from Lemma \ref{intersectionD} that the intersection $\D_0^+ \cap \D_\xbf^+$ is empty if and only if $\sgn(x_\sigma)\neq \sgn(x_\sigma)$ for some $\sigma$. Since $x_\sigma$ and $x_\sigma'$ are non-zero (otherwise $x$ would be zero) it follows that this is equivalent to the vanishing of \eqref{expr1}. Thus we can suppose that $\sgn(x_\sigma) = \sgn(x_\sigma)$ for all $\sigma$. Using the equality $K_{\frac{1}{2}}(\alpha)=\sqrt{\frac{2 \pi}{\alpha}}e^{-\alpha}$ we find that 
\begin{align}
    \int_{\D_0^+} \varphi^0(\xbf)=\sgn \N(x)= (-1)^N \bigl < \D_\xbf,\D_0^+ \bigr >_{\D^+}.
\end{align}
By the invariance property of the Kudla-Millson form we have
\begin{align} \int_{\D_\afrak^+}\varphi^0(\xbf) & = \int_{h_\afrak \D_0^+}\varphi^0(\xbf) \nonumber \\
    & = \int_{\D_0^+}\varphi^0(h_\afrak^{-1}\xbf) \nonumber \\
    & = (-1)^N\bigl < \D_{h_\afrak^{-1}\xbf}^+,\D_0^+ \bigr >_{\D^+} \nonumber \\
    & = (-1)^N\bigl < \D_{\xbf}^+,\D^+_\afrak \bigr >_{\D^+} \label{stepinproof2}.
\end{align}
Hence \eqref{stepinproof11} also holds for  $\D_\afrak^+$ instead of $\D_0^+$. Then we have
\begin{align}
\int_{C_\afrak} \Theta_n(v,\varphi_\fin) & = \sum_{i=1}^r \sum_{\substack{\xbf \in \Gamma_{h_i}' \backslash F^2_\Q \nonumber \\ Q(\xbf,\xbf)=2n}} \varphi_\fin(h_i^{-1}\xbf) \int_{C_\afrak} \sum_{\ybf \in \Gamma_{h_i}'\xbf}  \varphi^0(\sqrt{v}\ybf) \nonumber \\
& =  \sum_{i=1}^r \sum_{\substack{\xbf \in \Gamma_{h_i}' \backslash F^2_\Q \\ Q(\xbf,\xbf)=2n}} \varphi_\fin(h_i^{-1}\xbf) \sum_{\ybf \in \Ocal^{\times,+} \backslash \Gamma_{h_i}'\xbf } \int_{\D_\afrak^+}  \varphi^0(\sqrt{v}\ybf).
\end{align}
Since $\D_{\sqrt{v}\ybf}^+=\D_\ybf^+$ we then get
\begin{align}
\int_{C_\afrak} \Theta_n(v,\varphi_\fin) & = (-1)^N \sum_{i=1}^r \sum_{\substack{\xbf \in \Gamma_{h_i}' \backslash F^2_\Q \nonumber \\ Q(\xbf,\xbf)=2n}} \varphi_\fin(h_i^{-1}\xbf) \sum_{\ybf \in \Ocal^{\times,+} \backslash \Gamma_{h_i}'\xbf } \bigl < \D^+_\ybf,\D^+_\afrak \bigr >_{\D^+} \nonumber \\
& = (-1)^N\kappa \sum_{i=1}^r \sum_{\substack{\xbf \in \Gamma_{h_i}' \backslash F^2_\Q \\ Q(\xbf,\xbf)=2n}} \varphi_\fin(h_i^{-1}\xbf) \bigl < C_{\xbf}(h_i),C_\afrak \bigr >_{M_{h_i}}\nonumber  \\
& = (-1)^N\kappa \bigl  < C_n(\varphi_\fin),C_\afrak \bigr >_{M_K}.
\end{align}
\end{proof}

\noindent Hence with Proposition \ref{Ifourier} we get
\begin{align} \label{fourierofI} I(\tau, \varphi,\psi)=2^{1-N}\zeta_\fin(\varphi_1,\psi^{-1},0)+2^{1-N}\zeta_\fin(\varphi_2,\psi,0)+(-1)^N \kappa \sum_{n \in \Q_{>0}} \bigl < C_n(\varphi_\fin),C \otimes \psi \bigr >_{M_K} e^{2i\pi n \tau}.\end{align}

\paragraph{Change of model.} Recall that $W^0_F=X^0_F \oplus X^0_F$ is a $4$-dimensional symplectic space over $F$, whose restriction of scalars was $W_\Q=X_\Q \oplus X_\Q$. The sympectic form on $W_F^0$ is given by the skew symmetric-matrix
$
  \begin{pmatrix} 0 & A(Q^0) \\
  -A(Q^0) & 0
    \end{pmatrix} \in \Mat_4(F)
$
where $A(Q^0)=\begin{pmatrix}
 0 & 1 \\ 1 & 0
\end{pmatrix}$. Hence the symplectic group of $W^0_F$ is 
\begin{align}
\Sp(W^0_F) = \left \{ g \in \GL_{4}(F) \left \vert \transp{g} \begin{pmatrix}  & A(Q^0)\\ -A(Q^0) & \end{pmatrix}g=\begin{pmatrix}  & A(Q^0) \\ -A(Q^0) & \end{pmatrix} \right . \right \}.
\end{align}

Since the orthogonal group $\SO(F^2)$ is $F^\times$ we have two different models for the pair $ \SL_2(F) \times F^\times $ in $\Sp(W^0_F)$: on the one hand we can use the linear one for $\GL_2(F) \times F^\times$, on the other hand we can use the orthosymplectic model for $\SL_2(F) \times \SO(F^2)$; see Page \pageref{dualpairpage}. These models correspond to two different embeddings $\iota_{os}$ and $\iota_l$ of $\SL_2(F) \times F^\times$ in $\Sp(W^0_F)$ given by
\begin{align}
    \iota_{os}(g,t)& = \begin{pmatrix}at & & bt & \\ & at^{-1} & & bt^{-1} \\ ct & & dt & \\ & ct^{-1} & & dt^{-1} \end{pmatrix}, \\
    \iota_l(g,t) & = \begin{pmatrix}at & -bt & & \\ -bt & dt & & \\ & & at^{-1} & bt^{-1} \\ & & ct^{-1} & dt^{-1} \end{pmatrix},
\end{align}
where $g=\begin{psmallmatrix} a & b \\ c & d \end{psmallmatrix}$. With the latter embedding the linear pair $\SL_2(F) \times \GL_1(F)$ acts by
\begin{align}
    \omega_l(g,t)\varphi(\xbf)=\vert t \vert \varphi\left ( g^{-1}t\xbf \right ),
\end{align}
since 
\begin{align}
\adjast{\begin{pmatrix}
 a & -b \\ -c & d 
\end{pmatrix}}=\begin{pmatrix} 0 & 1 \\ 1 & 0 \end{pmatrix}\transp{\begin{pmatrix}
 a & -b \\ -c & d 
\end{pmatrix}}\begin{pmatrix} 0 & 1 \\ 1 & 0 \end{pmatrix}={\begin{pmatrix}
 a & b \\ c & d 
\end{pmatrix}}^{-1}.    
\end{align}
The two embeddings are conjugate to each other, we have $T\iota_{os}(g,t)T^{-1} = \iota_l(g,t)$ where \[T \coloneqq \begin{pmatrix}
 1 & 0 & 0 & 0 \\
 0 & 0 & -1 & 0 \\
 0 & 1 & 0 & 0 \\
 0 & 0 & 0 & 1
\end{pmatrix} \in \Sp(W^0_F).\]
We denote by $\Fcal$ the operator
\begin{align}
\Fcal \colon \Scal(\A_F^2) & \longrightarrow \Scal(\A_F^2) \nonumber \\
    \varphi & \longmapsto \omega(T)\varphi
\end{align}
Using the formula for the Weil representation in \eqref{weil rep} we get 
\begin{align}
\Fcal\varphi \begin{pmatrix}
    x \\ x'
    \end{pmatrix}=\int_{\A_F}\varphi \begin{pmatrix}
    z \\ x'
    \end{pmatrix}\chi(-xz)dz.
\end{align}
It is a partial Fourier transform and satisfies $   \Fcal \circ \omega_{os}(g,t)= \omega_{l}(g,t) \circ \Fcal
$ for every pair  $(g,t)$ in $\SL_2(\A_F) \times \A_F^\times$.

\begin{rmk} If $\varphi \begin{pmatrix} x \\ x' \end{pmatrix} = \varphi_1(x)\varphi_2(x')$ with $\varphi_1$ and $\varphi_2$ two Schwartz functions in $\Scal(\A_F)$, then $\Fcal \varphi\begin{pmatrix} x \\ x' \end{pmatrix} = {\varphi}^\vee_1(x) \varphi_2(x')$ where $\varphi^\vee$ is the Fourier transform on $\Scal(\A_F)$, see \eqref{Fourier1}.
\end{rmk}
For a Schwartz function $\phi$ in $\Scal(\A_F^2)$ and fixed $g$ in $\SL_2(\A_F)$, we define
\begin{align}
\widetilde{\Theta}_l(g,t,\phi) \coloneqq \sum_{\xbf \in F^2} \omega_l(g,t)\phi(\xbf) \in C^\8(F^\times \backslash \A_F^\times)^{K^0(\fffrak)}.
\end{align}
It is a function on $\A_F^\times$ in the variable $t$.
If $\Fcal\varphi=\phi$, then by Poisson summation we have
\begin{align} \label{poisson}
    \widetilde{\Theta}_l(g,t,\phi) =\widetilde{\Theta}_{os}(g,t,\varphi).
\end{align}

\begin{lem} \label{fouriercomput} Let $\phi_\8$ be the Schwartz function in $\Scal(F_\8^2)$ defined by $\phi_\8\coloneqq \Fcal\varphi_\8$. Then
\begin{align}
    \phi_\8 (\xbf)= (-i)^N\prod_{\sigma}e^{-\pi \lvert z_\sigma \rvert^2} z_\sigma
\end{align}
where $z_\sigma \coloneqq x_\sigma+ix'_\sigma$.
\end{lem}
\begin{proof}We compute
\begin{align} 
\Fcal \varphi_\8 (\xbf) & = \prod_{\sigma} \omega(T) \varphi_{\sigma} \begin{pmatrix} x_\sigma \\ x'_\sigma \end{pmatrix} \nonumber \\
& =\prod_{\sigma} \int_\R \varphi_{\sigma} \begin{pmatrix} \alpha_\sigma \\ x'_\sigma \end{pmatrix} e^{-2i\pi\alpha_\sigma x_\sigma}d\alpha_\sigma \nonumber \\
& =\prod_{\sigma} \int_\R e^{-\pi({x'}_\sigma^2+\alpha_\sigma^2)}(x'_\sigma+\alpha_\sigma)e^{-2i\pi\alpha_\sigma x_\sigma}d\alpha_\sigma \nonumber \\
& \stackrel{u_\sigma=\alpha_\sigma+ix_\sigma}{=}\prod_{\sigma} e^{-\pi({x'}_\sigma^2+x_\sigma^2)}\int_\R e^{-\pi u_\sigma^2}(u_\sigma+x'_\sigma-ix_\sigma)du_\sigma \nonumber \\
& = (-i)^N\prod_{\sigma} e^{-\pi \lvert z_\sigma \rvert^2} z_\sigma.\end{align}
\end{proof}

\begin{thm} \label{fouriercoeffs} Suppose that $\varphi_1$ or $\varphi_2$ vanishes. The diagonal restriction of the Eisenstein series $E(\tau_1 \dots, \tau_N,\phi_\fin,\psi)$ has the Fourier expansion
\[E(\tau, \dots, \tau,\phi_\fin,\psi)=\zeta_\fin(\varphi_1,\psi^{-1},0)+\zeta_\fin(\varphi_2,\psi,0)+(-1)^N 2^{N-1}\kappa\sum_{n \in \Q_{>0}} \bigl < C_n(\varphi_\fin),C \otimes \psi \bigr >_{M_K} e^{2i\pi n \tau},\]
where $\varphi_\fin$ is such that $\phi_\fin=\Fcal\varphi_\fin$.
\end{thm}

\begin{proof}

First by Poisson summation (P.S.) we have
\begin{align} \label{unfold0}
2^{N-1} I(\tau_1 \dots, \tau_N,\varphi_\fin,\psi,s) & = (v_1 \cdots v_N)^{-\frac{1}{2}}\int_{F^\times \backslash \A_F^\times}\widetilde{\Theta}_{os}(g_{\tauud},t,\varphi)\psi(t)\lvert t \rvert^{s} dt^\times \nonumber \\
& \stackrel{P.S.}{=} (v_1 \cdots v_N)^{-\frac{1}{2}}\int_{F^\times \backslash \A_F^\times}\widetilde{\Theta}_{l}(g_{\tauud},t,\phi)\psi(t)\lvert t \rvert^{s} dt^\times \nonumber \\
& =  (v_1 \cdots v_N)^{-\frac{1}{2}} \int_{F^\times \backslash \A_F^\times} \sum_{\xbf \in F^2-(0,0)} \omega_l(g_{\tauud},t)\phi(\xbf) \psi(t) \lvert t \rvert^{s} dt^\times \nonumber.
\end{align}
Note that by our choice of Schwartz function $\phi_\8$ we have $\phi_\8(0,0)=0$, hence the term at $(0,0)$ does not contribute to the summation. We have a bijection
\begin{align}
F^\times \times P(F) \backslash \GL_2(F) & \longrightarrow F^2-{(0,0)} \nonumber  \\
(u,\gamma) & \longmapsto u \gamma_0^{-1} \xbf_0.
\end{align}
where $P(F)$ is the stabilizer of $\xbf_0 = \transp{(1,0)}$ and $\gamma_0$ is one of the following representatives in $P(F)$
\begin{align}
    \begin{pmatrix} 1 & 0 \\ 0 & 1 \end{pmatrix} \; \textrm{or} \; \begin{pmatrix}
     \lambda & 1 \\ 1 & 0
    \end{pmatrix} \; \textrm{with} \; \lambda \in F^\times.
\end{align}
Hence the sum 
\begin{align}
\int_{F^\times \backslash \A_F^\times} \sum_{\xbf \in F^2-(0,0)} \omega_l(g_{\tauud},t)\phi(\xbf) \psi(t) \lvert t \rvert^{s} dt^\times 
\end{align}
can be unfolded as
\begin{align} \label{unfold2}
 \int_{F^\times \backslash \A_F^\times} \sum_{u \in F^\times}  \sum_{\gamma \in P(F) \backslash \GL_2(F)}  \omega_l(g_{\tauud},t)\phi(u \gamma^{-1}_0\xbf_0) \psi(t) \lvert t \rvert^{s} dt^\times
 = \int_{\A_F^\times} \sum_{\gamma \in P(F) \backslash \GL_2(F)}  \omega_l(g_{\tauud},t)\phi(u \gamma^{-1}_0\xbf_0) \psi(t) \lvert t \rvert^{s} dt^\times.
\end{align}
Since we have termwise absolute convergence for $\re(s)>N-1$ we can exchange the sum and the integral:
\begin{align} \label{unfold}
\int_{\A_F^\times} \sum_{\gamma \in P(F) \backslash \GL_2(F)}  \omega_l(g_{\tauud},t)\phi(u \gamma^{-1}_0\xbf_0) \psi(t) \lvert t \rvert^{s} dt^\times  = \sum_{\gamma \in P(F) \backslash \GL_2(F) }  Z\left (g_{\tauud},\gamma_0^{-1}\xbf_0, \phi,\psi,s \right ).
\end{align}
Thus we get $2^{N-1} I(\tau_1,\dots,\tau_N,\varphi_\fin,\psi,s)=E(\tau_1,\dots,\tau_N,\phi_\fin,\psi,s)$, and the Fourier expansion follows from \ref{fourierofI}.
\end{proof}

\section{Classical formulation for quadratic fields} \label{sectionrealfield}

We want to specialize Theorem \ref{fouriercoeffs} to the case where $N=2$ and $F=\Q(\sqrt{D})$ a quadratic field with $D>0$ and squarefree. We have $ \Ocal=\Z \left [ \lambda \right ]$, where $\lambda \coloneqq  \frac{d_F+\sqrt{d_F}}{2}$. and $d_F$ is the fundamental discriminant. We explicit the choices that allow us to recover \cite[Theorem.~A]{DPV}.

\subsection{The symmetric space associated to $\SO(2,2)$}

We identify $(F_\Q^2,Q)$ with the quadratic space $\left ( \Mat_2(\Q),2\det \right )$ via
\begin{align} \label{isommat2}
    F_\Q^2 & \longrightarrow \Mat_2(\Q) \nonumber \\
    \xbf=\begin{pmatrix} x \\ x' \end{pmatrix} & \longmapsto [ x', SAx ],
\end{align}
where $S=\begin{psmallmatrix}   0 & -1 \\ 1 & 0 \end{psmallmatrix}$. The fact that this is an isometry follows from $\det[a,b]=\transp{a}S^{-1}b$. Let
\begin{align}
    \Htil(\Q) \coloneqq \GL_2(\Q) \times_{\Q^\times} \GL_2(\Q)  \nonumber =\left \{ (g_1,g_2) \in \GL_2(\Q) \times \GL_2(\Q) \, \mid \det(g_1)=\det(g_2) \right \}.
\end{align}
It acts on $\Mat_2(\Q)$ by $\htil X  \coloneqq g_1Xg_2^{-1}$, where $\htil=(g_1,g_2)$ is in $\Htil(\Q)$. With this identification we have an isomorphism between $\Htil$ and the spin group $\GSpin_X$ of the quadratic space $X_\Q=F^2_\Q$. There is an exact sequence
\begin{align}
    1 \xrightarrow{\quad \quad} \Q^\times  \xrightarrow{\quad \quad} \Htil(\Q) \xrightarrow{\quad \nu \quad} H(\Q) \xrightarrow{\quad  \quad} 1
\end{align}
where the kernel of $\nu$ is given by $\Q^\times = \left \{ (t \id_2,t\id_2) \in \GL_2(\Q)^2 \, \mid t \in \Q^\times \right \}.$

\begin{lem}The map $\nu \colon \Htil(\Q) \longrightarrow H(\Q)$ is given by
\begin{align}
     \nonumber \\
    \nu \left ( g_1, g_2 \right ) & =   \begin{pmatrix}
     a\adjhash{g_1} & \frac{bA^{-1}Sg_1}{\det(g_2)} \\[4.5ex] cS^{-1}A\adjhash{g}_1  &  \frac{dg_1}{\det(g_2)}
    \end{pmatrix},
\end{align}
 where $g_2=\begin{psmallmatrix} a & b \\ c & d  \end{psmallmatrix}$ and $\adjhash{g_1}=A^{-1}\transp{g}_1^{-1}A$.
 \end{lem}
\begin{proof} We will use several times the equality $-Sg_1S=\det(g_1)\transp{g}_1^{-1}$.
Note that the inverse of the map \eqref{isommat2} is given by
\[ [ u, u' ]  \longmapsto \begin{pmatrix}
-A^{-1}Su' \\ u
\end{pmatrix}. \]
We compute the action separately on the two coordinates $\xbf =(x,x')$ of $F^2_\Q$. First if $x=\transp{(0,0)}$:
\begin{align}
    (g_1,g_2) \cdot [x',0]=\frac{1}{\det(g_2)}g_1[x',0]\begin{pmatrix}
     d & -b \\ -c & a
    \end{pmatrix} = \frac{1}{\det(g_2)}[ dg_1x' , -bg_1x' ]
\end{align} which is mapped in $F^2_\Q$ to
\[\frac{1}{\det(g_2)} \begin{pmatrix} bA^{-1}Sg_1x' \\ dg_1x' \end{pmatrix}.\]
On the other hand if $x'=\transp{(0,0)}$ we have
\begin{align}
    (g_1,g_2)\cdot [0,SAx ]= \frac{1}{\det(g_2)}[0,g_1SAx ]\begin{pmatrix}
     d & -b \\ -c & a
    \end{pmatrix} = \frac{1}{\det(g_2)}[ -cg_1SAx , ag_1SAx ]
\end{align}
This is mapped to
\[\begin{pmatrix} a \adjhash{g}_1x \\ -cSA \adjhash{g}_1x \end{pmatrix} \in F^2_\Q.\]
\end{proof}

\noindent The group of real points $\Htil(\R)$ has two connected components and
\begin{align}
    \Htil(\R)^+ =\left \{ (g_1,g_2) \in \GL_2(\R) \times \GL_2(\R) \, \mid \det(g_1)=\det(g_2)>0 \right \},
\end{align}
is the connected component of the identity. Its image is $\nu\left (\Htil(\R)^+ \right )=H(\R)^+$. 
We have a transitive action of $\Htil(\R)^+$ on the space $\D^+$ of negative lines in $\Mat_2(\R)$ that sends $z$ to $\htil z=g_1zg_2^{-1},$
where $\htil=(g_1,g_2)$. Consider the basis of $\Mat_2(\Q)$
\begin{align}
\Ebf_1=\begin{pmatrix} 1 & 0 \\ 0 & 1 \end{pmatrix} \quad   \Ebf_2=\begin{pmatrix} 0 & 1 \\ -1 & 0 \end{pmatrix} \quad 
\Fbf_1=\begin{pmatrix} 1 & 0 \\ 0 & -1 \end{pmatrix} \quad   \Fbf_2=\begin{pmatrix} 0 & 1 \\ 1 & 0 \end{pmatrix},
\end{align}
and the negative plane
\begin{align}
X_0 \coloneqq  \Span\bigl < \Fbf_1,\Fbf_2 \bigr > = \left \{  \begin{pmatrix} a & b \\ b & -a\end{pmatrix}, a,b \in \R \right \},
\end{align}
oriented by $\Fbf_1 \wedge \Fbf_2$. Its stabilizer in $\Htil(\R)^+$ is $\R_{>0}(\SO(2)\times \SO(2))$. Hence the stabilizer of 
$\alpha X_0 \beta^{-1}$ for $(\alpha,\beta)$ in $\Htil(\R)^+$ is $\R_{>0}\Ktil_\8(\alpha X_0\beta^{-1})$ where
\begin{align} \label{maxcompact}
\Ktil_\8(\alpha X_0 \beta^{-1})=(\alpha,\beta) \SO(2) \times \SO(2)(\alpha,\beta)^{-1}.
\end{align}
Note that under the isomorphism \eqref{isommat2} the negative plane $z_0=\{ (v,-v),v \in \R^2\}$ is mapped to $g_\8^{-1}X_0$. On the other hand the group $\Htil(\R)^+$ acts on $\HH \times \HH$ by mapping $(\tau_1,\tau_2)$ to $(g_1\tau,g_2\tau_2)$. The stabilizer of $(\alpha i, \beta i)$ is $\R_{>0}\Ktil_\8(\alpha X_0 \beta^{-1})$. Hence we have isomorphisms
\begin{align} \label{identDH}
    \D^+   \longrightarrow \Htil(\R)^+/\R_{>0}\Ktil_\8(\alpha X_0 \beta^{-1}) & \longrightarrow \HH \times \HH \nonumber \\
    g_1(\alpha X_0 \beta^{-1})g_2^{-1}  \longmapsto (g_1,g_2)\R_{>0}\Ktil_\8(\alpha X_0 \beta^{-1}) & \longmapsto (g_1\alpha i,g_2\beta i).
\end{align}
In the other direction we can express the map more concretely by
\begin{align} \label{identHD}
    \Psi \colon \HH \times \HH & \longrightarrow \D^+ \nonumber \\
    (\tau_1, \tau_2) & \longmapsto X(\tau_1,\tau_2) = \Span \bigl < \Fbf_1(\tau_1,\tau_2), \Fbf_2(\tau_1,\tau_2) \bigr >
\end{align}
where 
\begin{align}
\Fbf_1(\tau_1,\tau_2)& = \sqrt{y_1y_2} g_{\tau_1} \Fbf_1 g_{\tau_2}^{-1} = \begin{pmatrix}
 y_1 & -x_2y_1-x_1y_2 \\ 0 & -y_2
\end{pmatrix} \nonumber \\
\Fbf_2(\tau_1,\tau_2)& = \sqrt{y_1y_2} g_{\tau_1} \Fbf_2 g_{\tau_2}^{-1} = \begin{pmatrix}
 x_1 & -x_1x_2+y_1y_2 \\ -1 & x_2
\end{pmatrix}
\end{align}and $g_\tau=\begin{pmatrix}
 \sqrt{y} & \nicefrac{x}{\sqrt{y}} \\ 0 & \nicefrac{1}{\sqrt{y}}
\end{pmatrix}$ maps $i$ to $\tau=x+iy$. 

\subsection{The adelic isomorphism}
 
 Let $p$ be some fixed prime (later we will also assume that $p$ is an odd split prime). Let $\Ktil \coloneqq \Ktil_\8(z_0) \Ktil_0(p)$, where $\Ktil_0(p) \coloneqq K_0(p) \times_{\det} K_0(p)$ is an open compact in $\Htil(\widehat{\Z})$ and
\begin{align}
K_0(p)=\left \{ \begin{pmatrix}
 a & b \\ c & d
\end{pmatrix} \in \GL_2(\widehat{\Z}), p \mid c \right \}. 
\end{align}
By strong approximation for $\SL_2$ we know that 
\begin{align}
    \SL_2(\A_\fin)=\SL_2(\Q) K_0(p).
\end{align}Using the fact that the determinant $\det \colon K_0(p)_p \longrightarrow \Z_p^\times
$ is surjective for every $p$, one can also show that we have $\GL_2(\A_{\fin})=\GL_2(\Q)^+K_0(p)$. Hence the space 
\begin{align}
    M_{\Ktil}=\Htil(\Q) \backslash \Htil(\A)/\A^{\times} \Ktil
\end{align}
is connected. The map $\nu$ induces an isomorphism between $M_{\Ktil}$ and $M_K$. Since $\Htil(\Q)^+ \cap \Ktil_\fin=\Gamma_0(p) \times \Gamma_0(p)$, where $\Gamma_0(p) = \left \{ \begin{pmatrix}
 a & b \\ c & d
\end{pmatrix} \in \SL_2(\Z), \; p \mid c \right \}$, we have
\begin{align}
M_{\Ktil}=Y_0(p)\times Y_0(p).
\end{align}
where $Y_0(p) = \Gamma_0(p) \backslash \HH$.
\begin{rmk}
Note that since $ (-\id_2,\id_2)$ is in $\Htil(\Q)^+ \cap \Ktil_0(p)$ we have $\kappa=2$ in this setting.
\end{rmk}

\subsection{Hecke correspondences}

\begin{prop} \label{cyclesN2} Let $\xbf$ be in $\GL_2(\Q)^+$. After identifying $\D^+$ with $\HH \times \HH$, the submanifold $\D_\xbf^+$ is the image of the embedding
\begin{align}
    \HH & \hooklongrightarrow \HH \times \HH \nonumber \\
    \tau & \longmapsto (\xbf \tau,\tau).
\end{align}
\end{prop}

\begin{proof} We have $g_1 \xbf g_2^{-1}=\xbf$ if and only if $g_1=\xbf g_2 \xbf^{-1}$, which means that the stabilizer of $\xbf$ is
    \[\Htil_\xbf(\R)^+= \left \{  \left. (\xbf g \xbf^{-1},g) \in \Htil(\R)^+ \right \vert g \in \GL_2(\R)^+ \right \} \simeq \GL_2(\R)^+.\]
We fix the basepoint $\xbf X_0$ in $\D^+$ so that the stabilizer is
   \begin{align}
     \Ktil_\8(\xbf X_0 ) \coloneqq \R_{>0}(  \xbf \SO(2) \xbf^{-1}  \times \SO(2) ).  
   \end{align}
Moreover, the intersection $\Ktil_\xbf(\xbf X_0 ) \coloneqq \Htil_\xbf(\R)^+ \cap \Ktil_\8(\xbf X_0 ) \simeq \R_{>0} \SO(2)$ and we have an isomorphism
    \begin{align}
    \GL_2(\R)^+/\R_{>0}\SO(2)&  \longrightarrow \Htil_\xbf(\R)^+/\Ktil_\xbf( \xbf X_0) \nonumber \\
    g\R_{>0}\SO(2) & \longmapsto (\xbf g \xbf^{-1},g)\Ktil_\xbf(\xbf X_0).
    \end{align}
    Finally, composing with the identification \eqref{identDH} we obtain that $\D_\xbf^+$ is the image of
    \begin{align}
    \HH  \longrightarrow \Htil_\xbf(\R)^+/\Ktil_\xbf(\xbf X_0) & \hooklongrightarrow \HH \times \HH \nonumber \\
    \tau  \longmapsto (\xbf g_\tau \xbf^{-1},g_\tau) \Ktil_\xbf(\xbf X_0)& \longmapsto ( \xbf \tau, \tau).
    \end{align}
  \end{proof}
\noindent Let $\Gamma = \Gamma_0(p)$ for some fixed prime $p$ and define the set
\begin{align}
    \Delta_0(p) \coloneqq \left \{ \left . \begin{pmatrix}
    a & b \\ c & d
    \end{pmatrix} \in \Mat_2(\Z)  \right \vert p \mid c, \, (a,p)=1, \, ad-bc>0\right \},
\end{align}
and $\Delta_0(p)^{(n)}$ the subset of matrices in $\Delta_0(p)$ of determinant $n$. Let 
\begin{align}
(\Gamma \times \Gamma)_{\xbf}= \Htil_\xbf(\R)^+ \cap (\Gamma \times \Gamma) = \left \{  \left. (\xbf \gamma \xbf^{-1},\gamma) \in \Gamma \times \Gamma \right \vert \gamma \in \Gamma_\xbf \right \}
\end{align}
where $\Gamma_\xbf \coloneqq \Gamma \cap \xbf^{-1}\Gamma \xbf$. For $\xbf$ in $\GL_2(\Q)^+$, the special cycle $C_\xbf$ is the correspondence
\begin{align} \label{heckemap}
    \Gamma_\xbf \backslash \HH \hooklongrightarrow (\Gamma \times \Gamma)_\xbf \backslash \HH^2 \longrightarrow Y_0(p) \times Y_0(p)
\end{align}
where the first map is sending $\tau$ to $(\xbf \tau,  \tau)$ and the second is the projection
\begin{align}
  (\Gamma \times \Gamma)_\xbf \backslash \HH^2\longrightarrow (\Gamma \times \Gamma ) \backslash \HH^2 = Y_0(p) \times Y_0(p).  
\end{align}
Define the open compact in $\Mat_2(\widehat{\Z})$
 \begin{align}
     \widehat{\Delta}_0(p) \coloneqq \left \{ \left . \begin{pmatrix}
    a & b \\ c & d
    \end{pmatrix} \in \Mat_2(\widehat{\Z})  \right \vert a_p \in \Z_p^\times, \, c_p \in p\Z_p \right \},
     \end{align}
that satisfies $\Mat_2(\Q)^+ \cap \widehat{\Delta}_0(p)=\Delta_0(p)$. If we take as Schwartz function $\varphi_\fin = \id_{\widehat{\Delta}_0(p)}$ in $\Scal(\Mat_2(\A_F))$, then for positive $n$ we have
\begin{align}
    C_n(\id_{\widehat{\Delta}_0(p)}) = \sum_{\xbf \in \Gamma \backslash \Delta_0(p)^{(n)}/\Gamma} C_\xbf.
\end{align}

\subsection{Orientations}
In this section we transfer the orientations defined in Subsection \ref{sectionlss} on $\D^+$ to $\HH \times \HH$. In order to do so we define a basis of $\Mat_2(\Q)$:
\begin{align}
\Fbf_1(\tau_1,\tau_2)& = \sqrt{y_1y_2} g_{\tau_1} \Fbf_1 g_{\tau_2}^{-1} =  \begin{pmatrix}
 y_1 & -x_2y_1-x_1y_2 \\ 0 & -y_2
\end{pmatrix} \nonumber \\
\Fbf_2(\tau_1,\tau_2)& = \sqrt{y_1y_2} g_{\tau_1} \Fbf_2 g_{\tau_2}^{-1} =  \begin{pmatrix}
 x_1 & -x_1x_2+y_1y_2 \\ 1 & -x_2
\end{pmatrix} \nonumber \\
\Ebf_1(\tau_1,\tau_2)& = \sqrt{y_1y_2} g_{\tau_1} \Ebf_1 g_{\tau_2}^{-1} =  \begin{pmatrix}
 y_1 & -x_2y_1+x_1y_2 \\ 0 & y_2
\end{pmatrix} \nonumber \\
\Ebf_2(\tau_1,\tau_2)& =\sqrt{y_1y_2} g_{\tau_1} \Ebf_2 g_{\tau_2}^{-1} = \begin{pmatrix}
 x_1 & -x_1x_2-y_1y_2 \\ 1 & -x_2.
\end{pmatrix}
\end{align}
If the basepoint $(\tau_1,\tau_2)$ is clear we will write $\Ebf_1,\Ebf_2,\Fbf_1$ and $\Fbf_2$. As in Subsection \ref{sectionlss} we identify 
\begin{align}T_{X(\tau_1,\tau_2)}\D^+ = \Hom(X(\tau_1,\tau_2),X(\tau_1,\tau_2)^\perp) \simeq X(\tau_1,\tau_2)^\vee \otimes X(\tau_1,\tau_2)^\perp
\end{align}
and orient $\D^+$ by
\begin{align}
    (\Fbf_1^\vee \otimes \Ebf_1) \wedge (\Fbf_2^\vee \otimes \Ebf_1) \wedge (\Fbf_1^\vee \otimes \Ebf_2) \wedge (\Fbf_2^\vee \otimes \Ebf_2).
\end{align}
Let $\Psi \colon \HH \times \HH \longrightarrow \D^+$ be the isomorphism \eqref{identHD} and consider the differential
\begin{align}
    d\Psi \colon T_{(\tau_1,\tau_2)} \HH \times \HH \longrightarrow T_{X(\tau_1,\tau_2)}\D^+.
\end{align}
Let $\tau_1=x_1+iy_1$ and $\tau_2=x_2+iy_2$ be the coordinates on $\HH \times \HH$.

\begin{prop} \label{orientHH}
Under the identification of $\D^+$ with $\HH \times \HH$ the orientation of $\HH \times \HH$ is
\[4y_1^2y_2^2\frac{\partial}{\partial x_1} \wedge \frac{\partial}{\partial y_1} \wedge \frac{\partial}{\partial x_2} \wedge \frac{\partial}{\partial y_2} \]
and the orientation of $\D^+_\id$ in $\D^+$ at the point $X(\tau,\tau)$ is given by 
\[-y^4\left ( \frac{\partial }{\partial x_1}+\frac{\partial }{\partial x_2} \right ) \wedge \left ( \frac{\partial }{\partial y_2}+\frac{\partial }{\partial y_1} \right ). \]
\end{prop}
\noindent For a proof see \cite[Proposition 4.4.3]{branchereau_2022}

\subsection{A choice of basis for $F$} \label{sectionbaseF}
Now suppose that $p$ is an odd split prime in $F$. Then $d_F$ is a quadratic residue modulo $p$ and we can find an integer $r$ such that $r^2 \equiv d_F \pmod{p}$ and $r^2-d_F$ is positive. Since the discriminant $d_F$ is always congruent to $0$ or $1$ modulo $4$, it is always a square mod $4$ and we can suppose furthermore that $r^2 \equiv d_F \pmod{4p}$. We fix such a root $r$ and set
\begin{align}
 \epsilon_{r} \coloneqq \frac{\sqrt{d_F}-r}{2} \in \Ocal.
\end{align} We take the positive $\Z$-basis $\left \{ \epsilon_{r},1 \right \}$ of $\Ocal$. Let $N_0 \coloneqq 2N(\epsilon_r)$, so that $2N_0=r^2-d_F$. By changing the root we obtain another basis $\left \{ \epsilon_{-r},1 \right \}$. Note that $N_0$ is positive.

\begin{rmk} In the previous sections of this chapter, various objects we used implicitely depended on the choice of the $\Z$-basis of $\Ocal$, for example the matrix $g_\8$ in $\GL_N(\R)$, the embedding $h$ of $F_v^\times$ in $H(\Q_v)$ or the cycle $C \otimes \psi$. From now we will use the two $\Z$-bases $\left \{ \epsilon_{\pm r},1 \right \}$ of $\Ocal$ and decorate the symbols by $ r$, for example we will write $g_{\8, r},h_r$ or $C_{r} \otimes \psi$. 
\end{rmk}

The regular representation with respect to the basis $\{ \epsilon_r,1\}$ is given by
\begin{align} \label{embM0p}
    \gamma_{r} \colon (F\otimes \Q_v)^\times & \hooklongrightarrow \GL_2(\Q_v) \nonumber \\
    a\epsilon_{r}+b & \longmapsto\begin{pmatrix}
     b-ar & a \\ -\frac{aN_0}{2} & b
    \end{pmatrix}.
\end{align}
At all the other split places $q$ we can find $r_{q,r}$ in $\Z$ such that $r_{q,r}^2 \equiv d_F \pmod{q}$ and $r_{q,r}^2-d_F$ is positive. Moreover by changing the sign of $r$ if necessary we can also suppose that $\frac{r_{q,r}}{r}$ is positive. We set $r_{q,-r}=-r_{q,r}$. As above, for $q$ odd we can suppose furthermore that $r_{q,r}^2 \equiv d_F \pmod{4q}$. We define
\begin{align}
\epsilon_{q,r} \coloneqq \frac{\sqrt{d_F}-r_{q,r}}{2},
\end{align}
and since $r_{q,r} \equiv d_F \pmod{2}$ we have $\Ocal = \Z[\epsilon_{q,r}]$. Modulo $q$ the minimal polynomial of $\epsilon_{q,r}$ splits as $x(x+r_{q,r})$. We have the splitting $(q)=\q \q^\sigma$, where
\begin{align} \label{qideal}
    \q \coloneqq \left (\frac{\sqrt{d_F}-r_{q,r}}{2} \right )\Z+q\Z, \qquad    \q^\sigma \coloneqq \left (\frac{\sqrt{d_F}+r_{q,r}}{2} \right )\Z+q\Z;
\end{align}
see  \eqref{factorisation}.
At every split prime $q$ we have two square roots of $d_F$ in $\Z_q^\times$. We take $\beta_{q,r}$ in $\Z^\times_q$ to be the one that satisfies $\beta_{q,r} \equiv r_{q,r} \pmod{q\Z_q}$. Let $\varsigma_{v,r}$ be as in \eqref{varsigmaiso}. If $v$ is non-split then $\varsigma_{v,r}$ is the identity on $F_{\Q_v}$, if $v=q$ is split then 
\begin{align} \label{map1}
\varsigma_{q,r}  \colon F_{\Q_q} & \longrightarrow F_q=\Q_q \times \Q_q \nonumber \\
\sqrt{d_F} & \longmapsto (\beta_{q,r},-\beta_{q,r}) .
\end{align}
It maps $\q \otimes \Z_q$ to $(q\Z_q,\Z_q^\times)$ and $\q^\sigma \otimes \Z_q$ to $(\Z^\times_q,q\Z_q)$. 
Then, we precompose the map \eqref{embM0p} by the isomorphism $\varsigma_{v,r}^{-1} \colon F_v \longmapsto F_{\Q_v}$, which is the identity at the non split places. This yields an embedding
\begin{align}
\gamma_r \colon F_v^\times \hooklongrightarrow \GL_2(\Q_v) 
\end{align}
that we also denote by $\gamma_r$.
\begin{prop} \label{explicitemb} At split primes $q$ the embedding $\gamma_r$ of $F^\times_q$ in $\GL_2(\Q_q)$ is defined by
\[\gamma_r(t_q,t_{q}^\sigma)=\begin{pmatrix}
        \frac{t_q+t_{q}^\sigma}{2}-r\frac{t_q-t_{q}^\sigma}{2\beta_{q,r}} &  \frac{t_q-t_{q}^\sigma}{\beta_{q,r}} \\[2.5ex]
        -\frac{(t_q-t_{q}^\sigma)N_0}{2\beta_{q,r}} & \frac{t_q+t_{q}^\sigma}{2}+r\frac{t_q-t_{q}^\sigma}{2\beta_{q,r}}
    \end{pmatrix}.\]
\end{prop}
\begin{proof}
By definition $\varsigma_{q,r}$ maps $a\epsilon_r + b$ to $(a\epsilon_r+b,a\epsilon^\sigma_r+b)$, where $\epsilon_r=\frac{\beta_{q,r}-r}{2}$ and $\epsilon^\sigma_r=-\frac{\beta_{q,r}+r}{2}$ are in $\Q_q$. The preimage of $(t_\q,t_{\q^\sigma})$ is $a\epsilon_r+b$ in $F_{\Q_q}$ with
\begin{align}
a =\frac{t_q-t_{q}^\sigma}{\beta_{q,r}}, \qquad b= \frac{t_q+t_{q}^\sigma}{2}+r\frac{t_q-t_{q}^\sigma}{2\beta_{q,r}}
\end{align}
and the result follows by applying the map \eqref{embM0p}.
\end{proof}

\noindent For the basis $\{ \epsilon_{r},1\}$ we denote by $h_{r}$ the embedding of $F_v^\times$ in $H(\Q_v)$ defined in Subsection \ref{embH0}. Define the map
\begin{align}
\htil_r \colon F_v^\times & \hooklongrightarrow \Htil(\Q_v) \nonumber \\
t & \longmapsto \htil_r(t) = \begin{pmatrix} \gamma_{r}(t), \begin{pmatrix}  1  & 0 \\ 0 & \det(\gamma_{r}(t))\end{pmatrix} \end{pmatrix},
\end{align}
which is a lift of the map $h_{ r}$  {\em i.e.} $\nu \left (\htil_r(t) \right ) =h_{r}(t).$ At infinity this can be diagonalized as
\begin{align} \label{diagemb}
 \begin{pmatrix} \gamma_{r}(t_\8), \begin{pmatrix}  1 & 0 \\ 0 & \det(\gamma_{r}(t_\8)) \end{pmatrix} \end{pmatrix} = (g_{\8, r},1)^{-1} \begin{pmatrix} g(t_\8), \begin{pmatrix}  1  & 0 \\ 0 & \det(g(t_\8)) \end{pmatrix} \end{pmatrix} (g_{\8, r},1),
\end{align}
where \begin{align}
g_{\8,r} =\begin{pmatrix}
     \epsilon_r & 1 \\  \epsilon^{\sigma}_r & 1
    \end{pmatrix} \in \GL_2(\R).
\end{align}

\subsection{Geodesics on the modular curve and Hecke operators}

Suppose we have an embedding 
\begin{align} \label{orientation}
    \Phi \colon \R_{>0}  & \longrightarrow \HH \nonumber \\
    t & \longmapsto \Phi_x(t)+i\Phi_y(t).
\end{align}
whose image is the geodesic joining two points $\alpha$ and $\beta$ in $\R \cup \{\8 \}$. We orient the image $\Phi( \R_{>0})$  by 
\begin{align}
    d\Phi\left ( \frac{\partial}{\partial t}\right )=\frac{\partial \Phi_x}{\partial t} \frac{\partial}{\partial x}+\frac{\partial \Phi_y}{\partial t} \frac{\partial}{\partial y}.
\end{align}
We say that the geodesic is oriented from $\alpha$ to $\beta$ if
\begin{align} \label{orientcond}
\sgn \frac{\partial \Phi_x}{\partial t} = \sgn (\beta-\alpha)
\end{align}
and write $\Qcal(\alpha,\beta)$ for the (oriented) image in $\HH$. Note that the condition \eqref{orientcond} is equivalent to $\Phi(0)=\alpha$ and $\Phi(\8)=\beta$. If $\alpha$ is an RM-point we write $\Qcal(\alpha)=\Qcal(\alpha,\alpha^\sigma)$. Let $\Qcalb(\alpha,\beta)$ be the image of $\Qcal(\alpha,\beta)$ in $Y_0(p)$. If $\alpha$ is an RM-point, then $\Qcalb(\alpha)$ is a closed (compact) geodesic. Let $-\Qcal(\alpha)\coloneqq \Qcal(\alpha^\sigma,\alpha)$. 

We define Hecke operators on geodesics as follows. Let $\Qcal=\Qcal(\alpha)$ for some RM-point $\alpha$ and
\begin{align}
\Gamma[\Qcal] \coloneqq \left \{ \left . \gamma \in \Gamma \right \vert \gamma \Qcal = \Qcal \right \}
\end{align}
be its stabilizer. It is isomorphic to $\{ \pm 1 \} \times \gamma_{\Qcal}^{\Z}$, for some $\gamma_{\Qcal}$ in $\Gamma$. In particular, if $P\Gamma$ denotes the image of $\Gamma$ in $\PSL_2(\Z)$ then $P\Gamma[\Qcal]=\gamma_{\Qcal}^{\Z}$. Let $ R_n$ be a finite set of representatives in $\Delta_0(p)^{(n)}$ of the double quotient $\Gamma \backslash \Delta_0(p)^{(n)}/\Gamma$. Hence we can write
\begin{align} \label{doublecoset}
    \Delta_0(p)^{(n)} & = \bigsqcup_{\delta \in R_n} \Gamma \delta \Gamma.
\end{align}
Let $R_n(\Qcal)$ be a subset of $R_n$ that are representatives for the double quotient $\Gamma[Q] \backslash \Delta_0(p)^{(n)}/\Gamma$. Then we have
\begin{align} \label{doublecoset2}
    \Delta_0(p)^{(n)} & = \bigsqcup_{\delta \in R_n}  \Gamma \delta \Gamma = \bigsqcup_{\delta \in R_n(\Qcal)} \Gamma[\Qcal] \delta \Gamma.
\end{align}
We define the Hecke operator by
\begin{align}
    T_n\Qcalb \coloneqq \sum_{\delta \in R_n(\Qcal)} \overline{\delta^{-1} \Qcal}.
\end{align}

\subsection{The twisted class $C \otimes \psi$ in $\HH \times \HH$}

Define the points 
\begin{align}
    \alpha_r & \coloneqq g_{\8,r}^{-1}(0)= -\frac{1}{\epsilon_r}=\frac{r+\sqrt{d_F}}{N_0 }, \qquad \alpha_{-r} \coloneqq g_{\8,-r}^{-1}(0) = \frac{-r+\sqrt{d_F}}{N_0 }=-\alpha_r^\sigma.
\end{align}
By strong approximation, for every fractional ideal $\afrak$ we can find an element $g_{\afrak,r}$ in $\GL_2(\Q)^+$ such that $\gamma_r(t_{\afrak})$ lies in $g_{\afrak,r}^{-1}K_0(p)$,
where we suppose that $g_{\Ocal,r}=1$. We define the RM-points
\begin{align} \label{taudef}
    \alpha_{\afrak,r} \coloneqq g_{\afrak,r}\alpha_r, \qquad    \alpha_{\afrak,-r} \coloneqq g_{\afrak,-r}\alpha_{-r}.
\end{align} 
In particular $\alpha_{\Ocal,r}=\alpha_r$.

\begin{rmk}\label{rmkequiv} Suppose we replace $g_{\afrak,r}$ by another $\gtil_{\afrak,r}$ in $\GL_2(\Q)^+$ that satisfies that $\gamma_r(t_{\afrak})$ lies in $\gtil_{\afrak,r}^{-1}K_0(p)$, and let $\widetilde{\alpha}_{\afrak,r}=\gtil_{\afrak,r} \alpha_r$. Then $\gtil_{\afrak,r} g_{\afrak,r}^{-1}$ is in $\GL_2(\Q)^+ \cap K_0(p)=\Gamma
$ and hence $\widetilde{\alpha}_{\afrak,r}$ is in $\Gamma \alpha_{\afrak,r}$.
\end{rmk}

\begin{lem} \label{rules} The orbit $\Gamma\alpha_{\afrak,r}$ only depends on the narrow class $[\afrak]$ in $\Cl(F)^+$. In particular $\Qcalb(\alpha_{\afrak,r})$ only depends on the narrow class of $\afrak$.
\end{lem}
\begin{proof} First note that if $\lambda$ is in $F^\times$, then $\gamma_r(\lambda) \alpha_r=\alpha_r$ since $\gamma_r(\lambda)=g_{\8,r}^{-1} \begin{pmatrix} \lambda & 0 \\ 0 & \lambda^{\sigma}
    \end{pmatrix} g_{\8,r}$. Now suppose that we replace $\afrak$ by $\bfrak=(\lambda) \afrak$ and $\lambda$ is totally positive. Then we have $\gamma_r(t_{\bfrak})$ is in $\gamma_r(\lambda)g_{\afrak,r}^{-1}K_0(p)$. Hence, by the previous remark, for $g_{\bfrak,r}=g_{\afrak,r} \gamma_r(\lambda)^{-1}$ we get $\alpha_{\bfrak,r}$ is in $\Gamma g_{\afrak,r} \gamma_r(\lambda)^{-1} \alpha_r= \Gamma \alpha_{\afrak,r}$.
\end{proof}

Suppose that $\psi$ is unramified, so $\fffrak=\Ocal$. Let $K^0(\Ocal)= K_\8^0 \times \widehat{\Ocal}^\times$ where
 \begin{align}
 K_\8^0= \left \{ \left . (t_1,t_2) \in F^\times_\8 \; \right \vert \; t_i =\pm 1, \,  t_1t_2=1 \right \}.
 \end{align}
 We have
\begin{align}
    M_{\Ocal}= F^\times \backslash \A_F^\times/K^0(\Ocal) = \bigsqcup_{\afrak \in \Cl(F)^+} \widehat{\Ocal}^{\times,+} \backslash \R_{>0}^2,
\end{align}
where $\Cl(F)^+$ is the narrow class group.
 \begin{lem} \label{compactemb} We have $\htil_r(\widehat{\Ocal}^\times)$ is contained in $\Ktil_0(p)$ and $\htil_r(K^0_\8)=\htil_r((\R^\times)^N) \cap \Ktil_\8(z_0)$.
 \end{lem}
\begin{proof}  Since $\N(\epsilon_r)$ is in $p\Z_p$ and $2\beta_q$ in $\Z_q^\times$ for all split odd primes $q$, it follows that $ \gamma_r(\Ocal_q^\times)$ is included in $K_0(p)_q$, where $\gamma_r$ is the embedding in Proposition \ref{explicitemb}. Suppose that $2$ is split in $F$, and let $t_2$ and $t_2^\sigma$ be in $\Z_2^\times$. Then $t_2\equiv t_2^\sigma \equiv 1 \pmod{2\Z_2}$, hence
\begin{align}
\left \vert \frac{t_2 + t_2^\sigma}{2} \right \vert_2 \leq 1, \qquad \left \vert \frac{t_2 - t_2^\sigma}{2} \right \vert_2 \leq 1,
\end{align}
and $ \gamma_r(\Ocal_q^\times)$ is also contained in $K_0(p)_q$ for $q=2$. Then, since $\det(\gamma_r(t_q,t_{q}^\sigma))=t_q t_{q}^\sigma$ is in $\Z_q^\times$ for $t_q$ and $t_{q}^\sigma$ in $\Z_q^\times$, it follows that $ \htil_r(\Ocal_q^\times)$ is contained in $\Ktil_0(p)_q
$. On the other hand at the non split primes $q$ we have that $\Ocal_q$ is isomorphic to $\Ocal \otimes \Z_q$. The fact that $\gamma_r(\Ocal_q^\times) $ is contained in $\GL_2(\Z_q)$ follows from the fact that $\{\epsilon_r,1\}$ is a $\Z$-basis of $\Ocal$.
 Finally, since $z_0=g_{\8,r}^{-1} X_0$ we have
\begin{align}
\Ktil_\8(z_0)=(g_{\8,r}^{-1},1) \SO(2) \times \SO(2)(g_{\8,r},1),
\end{align}
and the second statement follows by \eqref{diagemb}.
  \end{proof}  
\noindent It follows from the lemma that the embedding $\htil_r$ induces a map
\begin{align}
    \htil_r \colon M_\Ocal \longrightarrow M_{\Ktil}
\end{align}
whose image in $M_{\Ktil}$ corresponds to the image of $h_{r}$ in $M_K\simeq M_{\Ktil}$. Let $C_{\afrak,r}$ be the image of a connected component of $\Ocal^{\times,+} \backslash \R_{>0}^2$ by $h_r$, that we considered in Section \ref{sectionclass}. Similarly, we define
\begin{align}
    C_{r} \otimes \psi = \sum_{[\afrak] \in \Cl(F)^+} \psi(\afrak) C_{\afrak,r} \in \Zcal_2(\overline{M_K},\partial \overline{M_K};\R).
\end{align}

Let $\alpha_1,\alpha_2,\beta_1$ and $\beta_2$ be real numbers. There is a natural product orientation on $\Qcal(\alpha_1,\beta_1) \times \Qcal(\alpha_2,\beta_2)$. If $\Phi_1$ and $\Phi_2$ are diffeomorphisms as in \eqref{orientation}, then we orient the product by
\begin{align}
    d(\Phi_1 \times \Phi_2) \left (t_1\frac{\partial}{\partial t_1} \wedge t_2\frac{\partial}{\partial t_2} \right )& = t_1t_2\left (\frac{\partial \Phi_{1,x}}{\partial t_1} \frac{\partial}{\partial x_1}+\frac{\partial \Phi_{1,y}}{\partial t_1} \frac{\partial}{\partial y_1} \right ) \wedge \left (\frac{\partial \Phi_{2,x}}{\partial t_2} \frac{\partial}{\partial x_2}+\frac{\partial \Phi_{2,y}}{\partial t_2} \frac{\partial}{\partial y_2} \right ).
\end{align}

\begin{prop} \label{geodesic} After identifying $M_K$ with $Y_0(p) \times Y_0(p)$, the cycle $C_{\afrak,r}$ corresponds to $\Qcalb(\alpha_{\afrak, r}) \times \Qcalb(\8,0)$.
\end{prop}
\begin{proof}  Let $\Ocal^{\times,+} t_\8$ be in the connected component of $M_\Ocal$ corresponding to $\afrak$. It corresponds to the point $F^\times(t_\8,t_{\afrak})K^0(\Ocal)$ in $M_\Ocal$, which is mapped to $\Htil(\Q)(\htil_r(t_\8),\htil_r(t_{\afrak}))\Ktil$. By definition 
\begin{align}
\htil_r(t_\8) = (g_{\8, r}^{-1},1) \begin{pmatrix} \begin{pmatrix} t_1 & 0 \\ 0 & t_2 \end{pmatrix}, \begin{pmatrix} 1 & 0 \\ 0 & t_1t_2 \end{pmatrix}\end{pmatrix} (g_{\8,r},1),
\end{align}
where $g_{\8,r} = \begin{pmatrix} \epsilon_{r} & 1 \\ \epsilon_{r}^\sigma & 1 \end{pmatrix}$ and the orientation is given by $t_1\frac{\partial}{\partial t_1} \wedge t_2 \frac{\partial}{\partial t_2}$. The image in $\HH \times \HH$ is the set
\begin{align} \label{orientation2}
   \D_0^+ & =\left \{ \left (g_{\8, r}^{-1} \frac{t_1}{t_2}i,\frac{1}{t_1t_2 }i \right ), t_1,t_2>0 \right \} \nonumber = \left \{ \left (g_{\8, r}^{-1} u_1i,\frac{1}{u_2} i \right ), u_1,u_2>0 \right \},
\end{align}
where we did the change of variable $(u_1,u_2)=\left ( \frac{t_1}{t_2},t_1t_2 \right )$. It preserves the orientation since
\begin{align}
    t_1\frac{\partial}{\partial t_1} \wedge t_2\frac{\partial}{\partial t_2} = 2 u_1 \frac{\partial}{\partial u_1} \wedge u_2 \frac{\partial}{\partial u_2}.
\end{align}We can parametrize $\D_0^+$ by
\begin{align}
    \R_{>0}^2 & \longrightarrow \D_0^+ \subset \HH \times \HH \nonumber \\
    (u_1,u_2) & \longmapsto (\Phi_1(u_1),\Phi_2(u_2))
\end{align}
where
\begin{align}
    \Phi_1(u_1)& = g_{\8, r}^{-1} u_1i =- \frac{\epsilon_r+u_1^2 \epsilon_r^\sigma}{\epsilon_r^2+u_1^2(\epsilon_r^\sigma)^2}+ i \frac{u_1(\epsilon_r-\epsilon_r^\sigma)}{\epsilon_r^2+u_1^2(\epsilon_r^\sigma)^2} \nonumber \\
    \Phi_2(u_2) & = \frac{1}{u_2}i.
\end{align}
Since $\Phi_1(0)=-\frac{1}{\epsilon_r}=\alpha_r$ and $\Phi_1(\8)=-\frac{1}{\epsilon^\sigma_r}=\alpha_r^\sigma<\Phi_1(0)$ we find that $\D_0^+=\Qcal(\alpha_r) \times \Qcal(\8,0)$, at least as nonoriented manifolds. The orientation at a point $(\Phi_1(u_1),\Phi_2(u_2))$ in $\Qcal(\alpha_r) \times \Qcal(\8,0)$ is given by
\begin{align}
   2 \left (u_1\frac{\partial \Phi_{1,x}}{\partial u_1} \frac{\partial}{\partial x_1}+u_1\frac{\partial \Phi_{1,y}}{\partial u_1} \frac{\partial}{\partial y_1} \right ) \wedge u_2\frac{\partial \Phi_{2,y}}{\partial u_2} \frac{\partial}{\partial y_2}
\end{align}
where
\begin{align}
 \frac{\partial \Phi_{1,x}}{\partial u_1} & = - \sqrt{d_F} \frac{2u_1\N(\epsilon_r)}{\left (\epsilon_r^2+u_1^2(\epsilon_r^\sigma)^2 \right )^2}<0 \nonumber \\
 \frac{\partial \Phi_{1,y}}{\partial u_1} & = \sqrt{d_F} \frac{\epsilon_r^2-u_1^2(\epsilon_r^\sigma)^2 }{\left (\epsilon_r^2+u_1^2(\epsilon_r^\sigma)^2 \right )^2} \\
 \frac{\partial \Phi_{2,y}}{\partial u_2} & = -\frac{1}{u_2^2}<0. \nonumber
\end{align}
Hence the orientations on both sides of $\D_0^+ = \Qcal(\alpha_r) \times \Qcal(\8,0)$ match.

Recall that for some $g_{\afrak,r}$ in $\GL_2(\Q)^+$ we have that $g_{\afrak,r} \gamma_r(t_\afrak)$ is in $K_0(p)$. Hence we write $\htil_r(t_{\afrak})=\htil_{\afrak,r}^{-1} \ktil_\fin$ in $\Htil(\Q)^+\Ktil_\fin$ for some $\ktil_\fin$ in $\Ktil_\fin$, where
\begin{align}
 \htil_{{\afrak}, r}=\left ( g_{{\afrak},r}, \begin{pmatrix} 1 & 0 \\ 0 & \det(g_{{\afrak},r}) \end{pmatrix} \right ).
\end{align} Since $g_{{\afrak},r}\Qcal(\alpha_r)=\Qcal(\alpha_{{\afrak},r})$ it follows that 
$\D_\afrak^+=\Qcal(\alpha_{\afrak,r})\times \Qcal(\8,0)$ and that $C_{\afrak, r}$ is the image in $Y_0(p) \times Y_0(p)$ of $\Qcal(\alpha_{\afrak,r})\times \Qcal(\8,0)$.
\end{proof}

\noindent We set 
\begin{align} 
\Qcalb(\psi) & \coloneqq \sum_{[\afrak] \in \Cl(F)^+} \psi(\afrak) \left ( \Qcalb(\alpha_{\afrak,r})+\Qcalb(\alpha_{\afrak,-r}) \right )
\end{align}
so that
\begin{align}
\Qcalb(\psi) \times \Qcalb(\8,0)= C_r\otimes \psi+C_{-r} \otimes \psi.    
\end{align} 

\begin{rmk}
By using the involution that maps $[\afrak]$ to $[\afrak^\sigma]=[\afrak]^{-1}$ and the fact that $\alpha_{\afrak^\sigma,r}$ and $\alpha_{\afrak,-r}$ are in the same $\Gamma$-orbit we find that $\Qcalb(\psi)=\Qcalb(\psi^{-1})$.
\end{rmk}

\subsection{Intersection numbers of geodesics}

Let $\Qcal(\alpha_1,\beta_1)$ and $\Qcal(\alpha_2,\beta_2)$ be two geodesics with pairwise distinct endpoints. We fix the orientation $y^2\frac{\partial}{\partial x} \wedge \frac{\partial}{\partial y}$ on $\HH$ and define the intersection of $\Qcal(\alpha_1,\beta_1)$ and $\Qcal(\alpha_2,\beta_2)$ as in Subsection \ref{poincaredual}.  In particular, the intersection in $\HH$ with the winding element $\Qcal(0,\8)$ is
\begin{align}
 \bigl < \Qcal(0,\8),\Qcal(\alpha,\beta) \bigr >_\HH & = \begin{cases}
 1 & \textrm{if} \; \beta <0 < \alpha \\
 -1 & \textrm{if} \; \alpha <0 < \beta  \\
 0 & \textrm{else}.
 \end{cases} 
\end{align}
 Since the geodesics are $1$ dimensional we have
\begin{align}
 \bigl < \Qcal(\alpha_1,\beta_1), \Qcal(\alpha_2,\beta_2) \bigr >_\HH = -\bigl < \Qcal(\alpha_2,\beta_2),\Qcal(\alpha_1,\beta_1) \bigr >_\HH= \bigl < \Qcal(\beta_2,\alpha_2),\Qcal(\alpha_1,\beta_1) \bigr >_\HH.
\end{align}

If $\Qcal=\Qcal(\alpha)$ for some RM-point $\alpha$ we can define the intersection on $Y_0(p)$ by
\begin{align} \label{intersec1}
    \bigl < \Qcalb(0,\8),\Qcalb \bigr >_{Y_0(p)} \coloneqq \sum_{\gamma \in \Gamma / \Gamma[\Qcal]} \bigl < \Qcal(0,\8),\gamma \Qcal \bigr >_\HH.
\end{align}
Let $\tau$ be any point on $\Qcal$ and $[\tau,\gamma_{\Qcal} \tau)$ the half-open geodesic segment, where $\gamma_Q$ is a generator of $\Gamma[Q]=\{\pm 1 \} \times \gamma_Q^\Z$. Then
\begin{align} \label{disjointunion}
    \Qcal = \bigsqcup_{i \in \Z} [\gamma_{\Qcal}^i\tau,\gamma_{\Qcal}^{i+1} \tau)
\end{align}
and we also have
\begin{align} \label{intersec3}
    \bigl < \Qcalb(0,\8),\Qcalb \bigr >_{Y_0(p)}=\sum_{\gamma \in P\Gamma} \bigl < \Qcal(0,\8),[\gamma \tau,\gamma \gamma_{\Qcal} \tau) \bigr >_\HH.
\end{align}
The intersection on the right hand side is only-non zero for finitely many $\gamma$'s.

\begin{lem} \label{interseclem} Let $\alpha$ and $\beta$ be two real numbers and $\ybf$ a matrix in $\Mat_2(\Z)$ of positive determinant. Then
\begin{align}
    \bigl < \D_{\ybf}^+, \Qcal(\alpha,\beta) \times \Qcal(\8,0) \bigr >_{\HH \times \HH}= \bigl < \Qcal(0,\8),\ybf^{-1}\Qcal(\alpha,\beta) \bigr >_\HH.
\end{align}
\end{lem}
\begin{proof} 
First since $\D_\ybf^+=\ybf \D_\id^+$ we have
\begin{align}
    \bigl < \D_{\ybf}^+, \Qcal(\alpha,\beta) \times \Qcal(\8,0) \bigr >_{\HH \times \HH} & = \bigl < \D_{\id}^+, \ybf^{-1}\Qcal(\alpha,\beta) \times \Qcal(\8,0) \bigr >_{\HH \times \HH} \nonumber \\
    & = \bigl < \D_{\id}^+, \Qcal(\ybf^{-1}\alpha,\ybf^{-1}\beta) \times \Qcal(\8,0) \bigr >_{\HH \times \HH}
\end{align}
where $\D_{\id}^+$ is the diagonal embedding of $\HH$ in $\HH \times \HH$. We set $\alpha'=\ybf^{-1}\alpha$ and $\beta'=\ybf^{-1}\beta$. We have a bijection
\begin{align} \label{bijecinter}
     \Qcal(\alpha',\beta') \cap  \Qcal(\8,0)  & \longrightarrow (\Qcal(\alpha',\beta') \times \Qcal(\8,0)) \cap \D_{\id}^+ \nonumber \\
    \tau & \longmapsto (\tau,\tau),
\end{align}
and these intersections contain at most one point.

We only have to check that the signs of the orientations matches when the intersection is non-empty. Let $\tau$ be a point in the intersection $\Qcal(0,\8) \cap  \Qcal(\alpha',\beta')$ and suppose that $\alpha'<0<\beta'$. Then the intersection number $\bigl < \Qcal(0,\8), \Qcal(\alpha',\beta') \bigr >_\HH$ is $-1$.
At the point $\tau$ the orientation on $\Qcal(\alpha',\beta')$ is given by $a\frac{\partial}{\partial x_1}+b\frac{\partial}{\partial y_1}$ with $a$ is a positive real number, and on $\Qcal(\8,0)$ by $-\frac{\partial}{\partial y_2}$. Hence the orientation on $\Qcal(\alpha',\beta') \times \Qcal(\8,0)$ is given by $-\left ( a\frac{\partial}{\partial x_1}+b\frac{\partial}{\partial y_1} \right )\wedge \frac{\partial}{\partial y_2}$. By Proposition \ref{orientHH} the orientation of $\D_\id^+$ is given (up to a positive scalar) by $-\left (\frac{\partial}{\partial x_1}+\frac{\partial}{\partial x_2} \right )\wedge\left (\frac{\partial}{\partial y_1}+\frac{\partial}{\partial y_2} \right )$. Since
\begin{align}
 o(T_{X(\tau,\tau)}\D_\id^+) \wedge o(T_{X(\tau,\tau)}\Qcal(\alpha,\beta) \times \Qcal(\8,0))  & = \left (\frac{\partial}{\partial x_1}+\frac{\partial}{\partial x_2} \right )\wedge\left (\frac{\partial}{\partial y_1}+\frac{\partial}{\partial y_2} \right ) \wedge \left ( a\frac{\partial}{\partial x_1}+b\frac{\partial}{\partial y_1} \right )\wedge \frac{\partial}{\partial y_2} \nonumber \\
& = -a \frac{\partial}{\partial x_1} \wedge \frac{\partial}{\partial y_1} \wedge \frac{\partial}{\partial x_2} \wedge \frac{\partial}{\partial y_2}
\end{align}
and $a$ is positive we see that
\begin{align} \label{orient2}
    \bigl < \D_\id^+,  \Qcal(\alpha',\beta') \times \Qcal(\8,0) \bigr >_{\HH \times \HH} = -1.
\end{align}
When $\alpha'<0<\beta'$ then $a$ is negative and the intersection numbers $\bigl < \Qcal(0,\8), \Qcal(\alpha',\beta') \bigr >_\HH$ and $\bigl < \D_\id^+,  \Qcal(\alpha',\beta') \times \Qcal(\8,0) \bigr >_{\HH \times \HH}$ are both equal to $1$.
\end{proof}

\begin{prop} \label{geodinter} For $\varphi_{\fin}=\id_{\widehat{\Delta}_0(p)}$ we have
\[ \int_{\Qcalb(\alpha_{\afrak,r}) \times \Qcalb(\8,0)} \Theta_n(v,\varphi_\fin) = 2 \bigl < \Qcalb(0,\8),T_n\Qcalb(\alpha_{\afrak,r}) \bigr >.\]
\end{prop}

\begin{proof} By definition, for $\varphi_\fin = \id_{\widehat{\Delta}_0(p)}$ we have
\begin{align}
    \Theta_n(v,\varphi_\fin) = \sum_{\substack{\xbf \in \Mat_2(\Q) \\ Q(\xbf,\xbf)=2n }} \varphi_\fin(\xbf) \varphi^0(\sqrt{v}\xbf) = \sum_{\xbf \in \Delta_0(p)^{(n)}} \varphi^0(\sqrt{v}\xbf).
\end{align}We write $\Qcal=\Qcal(\alpha_{\afrak,r})$. Let $\tau$ be any point on $\Qcal$ and $[\tau,\gamma_{\Qcal} \tau)$ the half-open geodesic segment. The group $  \Gamma[\Qcal] \times \left ( \Gamma/\Gamma[\delta^{-1} \Qcal]  \right ) $ acts transitively on $\Gamma[\Qcal] \delta \Gamma$ , where $\Gamma[\delta^{-1} \Qcal]=\delta^{-1} \Gamma[\Qcal] \delta$. Using the decomposition \eqref{doublecoset2}  we write 
\begin{align}
    \Theta_n(v,\varphi_\fin) & = \sum_{\xbf \in \Delta_0(p)^{(n)}} \varphi^0(\sqrt{v}\xbf) \nonumber \\
    & = \sum_{\delta \in R_n(\Qcal)} \sum_{\gamma_1 \in \Gamma[\Qcal]} \sum_{\gamma_2 \in \Gamma/\Gamma[\delta^{-1} \Qcal]} \varphi^0(\sqrt{v}\gamma_1\delta \gamma_2^{-1}).
\end{align}
 We get
\begin{align}
   \int_{\Qcalb \times \Qcalb(\8,0)} \Theta_n(v,\varphi_\fin) & = \sum_{\delta \in R_n(\Qcal)}   \sum_{\gamma_1 \in \Gamma[\Qcal]} \sum_{\gamma_2 \in \Gamma / \Gamma[\delta^{-1} \Qcal]} \int_{[\tau, \gamma_{\Qcal} \tau ) \times \Qcal(\8,0)} \varphi^0(\sqrt{v}\gamma_1\delta \gamma_2^{-1}).
\end{align}
By the invariance of $\varphi^0$ we have $\varphi^0(\sqrt{v}\gamma_1\delta \gamma_2^{-1})=(\gamma_1^{-1},1)^\ast \varphi^0(\sqrt{v}\delta\gamma_2^{-1})$. Thus using \eqref{disjointunion} we get
\begin{align}
 \sum_{\gamma_1 \in \Gamma[\Qcal]} \int_{[\tau, \gamma_{\Qcal} \tau ) \times \Qcal(\8,0)} \varphi^0(\sqrt{v}\gamma_1\delta \gamma_2^{-1}) & = \sum_{\gamma_1 \in \Gamma[\Qcal]} \int_{[\tau, \gamma_{\Qcal} \tau ) \times \Qcal(\8,0)} (\gamma_1^{-1},1)^\ast \varphi^0(\sqrt{v}\delta\gamma_2^{-1}) \nonumber \\
 & = \sum_{\gamma_1 \in \Gamma[\Qcal]} \int_{[\gamma_1^{-1}\tau, \gamma_1^{-1}\gamma_{\Qcal} \tau ) \times \Qcal(\8,0)} \varphi^0(\sqrt{v}\delta\gamma_2^{-1}) \nonumber \\
 & =  2 \int_{\Qcal \times \Qcal(\8,0)} \varphi^0(\sqrt{v}\delta\gamma_2^{-1} ),
\end{align}
and the factor $2$ comes from the subgroup $\{ \pm 1\}$ of $\Gamma[\Qcal]$. Since $\D_\afrak^+=\Qcal(\alpha_{\afrak,r}) \times \Qcal(\8,0)$ it follows from \eqref{stepinproof2} in the proof of Proposition \ref{intertop} that
\begin{align}
  \int_{\Qcal \times \Qcal(\8,0)} \varphi^0(\sqrt{v}\delta \gamma_2^{-1} )  & = \bigl < \D_{\delta \gamma_2^{-1}}^+, \Qcal \times \Qcal(\8,0) \bigr >_{\HH \times \HH},
\end{align}
and by Lemma \ref{interseclem} we have
\begin{align}
  \bigl < \D_{\delta \gamma_2^{-1}}^+, \Qcal \times \Qcal(\8,0) \bigr >_{\HH \times \HH}  & = \bigl < \Qcal(0,\8), \gamma_2 \delta^{-1} \Qcal \bigr >_\HH.
\end{align}
Thus
\begin{align}
   \int_{\Qcalb \times \Qcalb(\8,0)} \Theta_n(v,\varphi_\fin) & =  2 \sum_{\delta \in R_n(\Qcal)}  \sum_{\gamma_2 \in \Gamma / \Gamma[\delta^{-1} \Qcal]} \bigl < \Qcal(0,\8),
  \gamma_2 \delta^{-1}  \Qcal \bigr >_\HH \nonumber \\
  & = 2 \sum_{\delta \in R_n(\Qcal)} \bigl < \Qcalb(0,\8),
  \overline{\delta^{-1} \Qcal} \bigr >_{Y_0(p)} \nonumber \\
  & = 2 \bigl < \Qcalb(0,\8),
  T_n\overline{\Qcal} \bigr >_{Y_0(p)}
\end{align}
where we use \eqref{intersec1} in the second equality.
\end{proof}

\subsection{Two lattices} Let $\widehat{\dfrak}^{-1} \coloneqq \sideset{}{'} \prod_w \dfrak_w^{-1}$, where $\dfrak^{-1}_w$ is the inverse different ideal. Let also $\dfrak^{-1}_v \coloneqq \prod_{w \mid v} \dfrak^{-1}_w$. Define the following lattices in $\A_F^2$
\begin{align}
    L^{(r)} \coloneqq \left \{ \left . \begin{pmatrix}
    x \\x'
    \end{pmatrix} \in \widehat{\Ocal}^2 \; \right \vert \; x \in \widehat{\dfrak}^{-1}, \quad w_\p(x')\geq 1, \quad w_{\p^\sigma}(x')=0 \right \}, \\
    L^{(-r)} \coloneqq \left \{ \left . \begin{pmatrix}
    x \\x'
    \end{pmatrix} \in \widehat{\Ocal}^2 \; \right  \vert \; x \in \widehat{\dfrak}^{-1}, \quad w_\p(x')= 0, \quad w_{\p^\sigma}(x')\geq 1 \right \}.
\end{align}
Recall the isomorphism
\begin{align} \label{ismoFmat}
    F_\Q^2 & \longrightarrow \Mat_2(\Q) \nonumber \\
    \xbf=\begin{pmatrix} x \\ x' \end{pmatrix} & \longmapsto [ x', SAx ],
\end{align}
defined in \eqref{isommat2}, where we identify $F_\Q$ with $\Q^2 $ via the $\Z$-basis $(\epsilon_r,1)$ of $\Ocal$. After passing to the adèles and composing with the isomorphism between $\A_F$ and $F_\A$, we get an isomorphism between $\A_F^2$ and $\Mat_2(\A)$.
\begin{lem} \label{latticeisom} By the isomorphism between $\A_F^2$ and $\Mat_2(\A)$ the lattice
$L^{(r)}$ correspond to $\widehat{\Delta}_0(p)$. If we replace $(\epsilon_r,1)$ with $(\epsilon_{-r},1)$ then $L^{(-r)}$ correspond to $\widehat{\Delta}_0(p)$.
\end{lem}
\begin{proof} We have 
\begin{align}
    F_v\oplus F_v=
    \begin{cases}  \Q_v^2 \oplus \Q^2_v & \textrm{if $v$ split}\\
    F\otimes \Q_v \oplus F \otimes \Q_v & \textrm{otherwise},
    \end{cases}
\end{align}
where we write the vectors as line vectors instead of column vectors. We decompose $L^{(r)} = \prod_v L_v^{(r)}$ and $\widehat{\dfrak}^{-1} = \prod_v \dfrak^{-1}_v$ where
\begin{align}
    L_v^{(r)}=
    \begin{cases}  \Ocal_p \oplus (p\Z_p,\Z^\times_p) & \textrm{if $v=p$ }\\
    \Ocal_v \oplus \Ocal_v & \textrm{if $v \nmid pd_F$} \nonumber \\
    \dfrak_v^{-1} \oplus \Ocal_v   & \textrm{if $v \mid d_F$}.
    \end{cases}
\end{align}
If $v\mid d_F$, then $(p)=\p_v^2$ and $\dfrak_v^{-1}=\p^{-1}_v$, since $\dfrak^{-1}=\frac{1}{\sqrt{d_F}}\Ocal$. Similarly
\begin{align}
    L_v^{(-r)}=
    \begin{cases}  \Ocal_p \oplus (\Z^\times_p,p\Z_p)   & \textrm{if $v=p$ }\\
    \Ocal_v \oplus \Ocal_v & \textrm{if $v \nmid pd_F$} \\
    \dfrak_v^{-1} \oplus \Ocal_v  & \textrm{if $v \mid d_F$}.
    \end{cases}
\end{align} Let $A_r=\transp{g}_{\8,r}g_{\8,r}$ be the matrix $A=\transp{g}_\8g_\8$ relative to the basis $(\epsilon_r,1)$, which is given by
\begin{align}
    A_{r}=\begin{pmatrix}
     \frac{r^2+d_F}{2} & {r} \\ {r} & 2
    \end{pmatrix}.
\end{align}
We identity $F$ with $\Q^2$ by mapping $a\epsilon_{r}+b$ to $\begin{pmatrix}
a \\ b
\end{pmatrix}$. We look place by place.
\begin{itemize}
\item Suppose that $v = p$. We have $A_r$ is in $\GL_2(\Z_p)$, hence the isomorphism \eqref{ismoFmat} identifies
    \begin{align}
        F^2_{\Q_p} & \simeq \Mat_2(\Q_p) \nonumber \\
       \begin{pmatrix}
          \epsilon_r \Z_p+\Z_p \\
         \epsilon_r \Z^\times_p+p\Z_p
        \end{pmatrix}  & \simeq \begin{pmatrix}
        \Z_p^\times & \Z_p \\ p\Z_p & \Z_p
        \end{pmatrix}=\Delta_0(p)_p
    \end{align}
    On the other hand the isomorphism $\varsigma_{p,r} \colon F_{\Q_p} \longrightarrow F_p$ identifies $\epsilon_r \Z^\times_p+p\Z_p$ with $(p\Z_p,\Z_p^\times)$, since we chose $r$ such that $\epsilon_r$ is in $p\Z_p$ and $\epsilon_{-r}$ is in $\Z_p^\times$. Hence we have
    \begin{align}
        F^2_{\Q_p} & \simeq F_p^2 \nonumber \\
        \begin{pmatrix}
          \epsilon_r \Z_p+\Z_p \\
         \epsilon_r \Z^\times_p+p\Z_p
        \end{pmatrix} &   \simeq \Ocal_p \oplus (p\Z_p,\Z_p^\times) =L_p^{(r)}.
    \end{align}
    At this place if we replace $\epsilon_r$ by $\epsilon_{-r}$ we identify $\epsilon_{-r} \Z^\times_p+p\Z_p$ with $(\Z_p^\times,p\Z_p)$, hence $\Delta_0(p)_v$ is identified with $L_v^{(-r)}$.
    \item Suppose that $v \nmid pd_F$. We have $A_r$ is in $\GL_2(\Z_v)$, hence the isomorphism \eqref{ismoFmat} identifies
    \begin{align}
        F^2_{\Q_v} & \simeq \Mat_2(\Q_v) \nonumber \\
        \begin{pmatrix}
         \epsilon_r \Z_v+\Z_v \\
         \epsilon_r \Z_v+\Z_v
        \end{pmatrix}  & \simeq \Mat_2(\Z_v)=\Delta_0(p)_v
    \end{align}
    On the other hand we have
    \begin{align}
        F^2_{\Q_v} & \simeq F_v^2 \nonumber \\
        \begin{pmatrix}
         \epsilon_r \Z_v+\Z_v \\
         \epsilon_r \Z_v+\Z_v
        \end{pmatrix} = \Ocal^2 \otimes \Z_v & \simeq \Ocal_v \oplus \Ocal_v=L_v^{(r)}.
    \end{align}
    
    \item Finally suppose that $v \mid d_F$ is ramified. At this place the isomorphism \eqref{ismoFmat} is
\begin{align} \label{isomvmat}
    F_v^2 & \longrightarrow \Mat_2(\Q_v) \nonumber\\
    \begin{pmatrix}  c\epsilon_{r}+d \\ a\epsilon_{r}+b
     \end{pmatrix} & \longmapsto \begin{pmatrix}
     a & -2d-cr \\
     b & ds+c\Tr(\epsilon_{ r}^2)
    \end{pmatrix},
\end{align}
and similarly for $\epsilon_{-r}$. By definition $c\epsilon_{r}+d$ is in $\dfrak_v^{-1}$ if and only if $2d+cr $ and $dr+c\Tr(\epsilon_{r}^2)$ are both in $\Z_v$. Hence the isomorphism \eqref{isomvmat} identifies
\begin{align}
        F_v^2 & \simeq \Mat_2(\Q_v) \nonumber \\
        L_v^{(r)} =\dfrak_v^{-1} \oplus  \Ocal_v   & \simeq \Mat_2(\Z_v)=\Delta_0(p)_v.
    \end{align}
\end{itemize}

\end{proof}

\subsection{$p$-smoothing of Eisenstein series}

For a squarefree ideal $I$ we define the Schwartz functions $\varphi^{I}$ in $\Scal(\A_{F}^2)$ by $\varphi^I_\8=\varphi_\8$ as before and
\begin{align}
\varphi^{I}_w\begin{pmatrix}
x \\ x'
\end{pmatrix} \coloneqq 
\begin{dcases}
\id_{\dfrak^{-1}_{w}}(x) \left ( \id_{\Ocal_{w}}(x')-\id_{\m_w}(x') \right ) & \quad \textrm{if} \quad w(I)>0 \\
\id_{\dfrak^{-1}_{w}}(x)\id_{\Ocal_{w}}(x') & \quad  \textrm{if} \quad w(I)=0.
\end{dcases}
\end{align}
Let $\phi^{I} \coloneqq \Fcal\varphi^{I}$ be its partial Fourier transform.

\begin{lem} We have \begin{align}
\phi^{I}_w \begin{pmatrix}
x \\ x'
\end{pmatrix} = 
\vol(\dfrak^{-1}_w) \begin{dcases}
\id_{\Ocal_{w}}(x) \left ( \id_{\Ocal_{w}}(x')-\id_{\m_w}(x') \right ) & \quad \textrm{if} \quad w(I)>0 \\
\id_{\Ocal_{w}}(x)\id_{\Ocal_{w}}(x') & \quad  \textrm{if} \quad w(I)=0.
\end{dcases}
\end{align}
\end{lem}
\begin{proof}
First note that if $\chi \colon K \longrightarrow U(1)$ is a character on a compact group $K$ then
\begin{align} \int_K \chi(y)d\mu(y)= \begin{cases}
\mu(K) & \textrm{if $\chi=1$} \\
0 & \textrm{if $\chi \neq 1$}.
\end{cases}
\end{align}
We have
\begin{align}
{\id}^\vee_{\m_w^m \dfrak^{-1}_w}(x) = \int_{\m_w^m \dfrak^{-1}_w}\chi_x(y)d\mu(y),
\end{align}
where $\chi_x(y) = \chi(xy)$. The character $\chi_x$ is trivial on $\m_w^m \dfrak^{-1}_w$ if and only if $x$ is in $\m_w^{-m} \Ocal_w$, hence ${\id}^\vee_{\m_w^m \dfrak^{-1}_w}=q_w^{-m}\vol(\dfrak_w^{-1})\id_{\m_w^{-m} \Ocal_w}$. 
Since $\varphi^I\begin{pmatrix} x \\ x' \end{pmatrix}=\varphi_1(x)\varphi_2(x')$ is a product of two Schwartz functions $\varphi_1$ and $\varphi_2$ in $\Scal(\A_F)$, we have $\phi^I\begin{pmatrix} x \\ x' \end{pmatrix}={\varphi}^\vee_1(x)\varphi_2(x')$, where ${\varphi}^\vee_1(x)$ is the Fourier transform of $\varphi_1$. The lemma then follows from the previous computation.
\end{proof}

\noindent The Eisenstein series
\begin{align}
    E(\tau_1,\tau_2,\psi) \coloneqq E(\tau_1,\tau_2,\id_{\widehat{\Ocal}^2},\psi)
\end{align}
is a Hilbert modular form of parallel weight $1$ for $\SL_2(\Ocal)$. Its diagonal restriction $E(\tau,\tau,\psi,0)$ vanishes, since it is a weight $2$ modular form for $\SL_2(\Z)$. For an arbitrary ideal $I$ we define the Eisenstein series
\begin{align}
    E^{I}(\tau_1,\tau_2,\psi) \coloneqq E(\tau_1,\tau_2, \phi_\fin^{I},\psi),
\end{align}
which is of level 
\begin{align}
\Gamma_0(I) = \left \{ \left . \begin{pmatrix} a & b \\ c & d \end{pmatrix} \in \SL_2(\Ocal) \right \vert c \in I \right \}.
\end{align}
When $I=(p)$ we call it the $p$-stabilization of $E(\tau_1,\tau_2,\psi)$.

Let $\varphi_1^{I}$ and $\varphi_2^{I}$ be the restriction of $\varphi_\fin^{I}$ to the isotropic lines $l_1$ and $l_2$. When $w(I)$ is positive we have
\begin{align}
    \varphi_w^I \begin{pmatrix}
x \\ 0
\end{pmatrix} = \id_{\dfrak^{-1}_{w}}(x) \left ( \id_{\Ocal_{w}}(0)-\id_{\m_w}(0) \right ) =0.
\end{align}
Hence the function $\varphi_1^{I}$ vanishes.

\begin{cor} \label{cordpv} We have
\begin{align}
    E^{(p)}\left (\tau,\tau,\psi \right ) =
    \begin{dcases}
   L^{(p)}(\psi,0)- 4\sum_{n=1}^\8 \bigl < \Qcalb(0,\8), T_n\Qcalb(\psi) \bigr >_{Y_0(p)} e^{2i\pi n \tau} & \textrm{if $p$ is split}, \\
    0 & \textrm{if $p$ is inert}.
    \end{dcases}
\end{align}
\end{cor}

\begin{proof}
We can rewrite the Schwartz functions as
\begin{align} \label{schwartzsum2}
    \varphi_\fin^{(p)}=
    \begin{cases}
    \id_{ \widehat{\dfrak}^{-1} \oplus \widehat{\Ocal} }-\id_{  \widehat{\dfrak}^{-1}  \oplus \widehat{\p} }-\id_{  \widehat{\dfrak}^{-1} \oplus \widehat{\p}^\sigma}+\id_{\widehat{\dfrak}^{-1} \oplus p\widehat{\Ocal} } & \textrm{if $p= \p{\p^\sigma }$} \\
    \id_{ \widehat{\dfrak}^{-1} \oplus \widehat{\Ocal} }-\id_{ \widehat{\dfrak}^{-1} \oplus  p\widehat{\Ocal} } & \textrm{if $p$ is inert},
    \end{cases}
\end{align}
\begin{align} \label{schwartzsum}
    \phi_\fin^{(p)}= d_F^\frac{1}{2}
    \begin{cases} \id_{\widehat{\Ocal}^2}-\id_{\widehat{\Ocal} \oplus \widehat{\p}}-\id_{\widehat{\Ocal} \oplus \widehat{\p}^\sigma }+\id_{\widehat{\Ocal} \oplus p\widehat{\Ocal}} & \textrm{if $p= {\p \p^\sigma}$} \\ \id_{\widehat{\Ocal}^2}-\id_{\widehat{\Ocal} \oplus p\widehat{\Ocal}} & \textrm{if $p$ is inert}.
    \end{cases}
\end{align}
First note that for $\gamma=\begin{pmatrix} 1 & 0 \\ 0 & p \end{pmatrix}$ we have $\omega_l(\gamma)\id_{\widehat{\Ocal}^2}=p^{\frac{1}{2}}\id_{\widehat{\Ocal} \oplus p\widehat{\Ocal}}.$ By applying the transformation in Proposition \ref{proptransfo} we find that
\begin{align}
    E(\tau_1,\tau_2, \id_{\widehat{\Ocal} \oplus p\widehat{\Ocal}},\psi)=p^2E\left (p\tau_1,p\tau_2, \id_{\widehat{\Ocal}^2},\psi \right).
\end{align}
Hence if $p$ is inert, we have
\begin{align}
E^{(p)}(\tau,\tau, \psi)=d_F^\frac{1}{2} \left ( E(\tau,\tau,\psi)-p^2E\left (p\tau,p\tau,\psi \right) \right )=0,
\end{align}
since $E(\tau,\tau)$ is zero. 

Suppose that $p$ is split. We have $\id_{L^{(r)}}=\id_{\widehat{\dfrak}^{-1} \oplus \widehat{\p} }-\id_{ \widehat{\dfrak}^{-1}\oplus p\widehat{\Ocal} }$ and $\id_{L^{(-r)}}=\id_{\widehat{\dfrak}^{-1} \oplus \widehat{\p^\sigma} }-\id_{ \widehat{\dfrak}^{-1}\oplus p\widehat{\Ocal} }$, hence
\begin{align}
    \varphi^{(p)}_\fin+\id_{L^{(r)}}+\id_{L^{(-r)}}& =\id_{\widehat{\dfrak}^{-1} \oplus \widehat{\Ocal} }-\id_{\widehat{\dfrak}^{-1} \oplus p\widehat{\Ocal}}, \nonumber \\
    \phi^{(p)}_\fin+\Fcal\id_{L^{(r)}}+\Fcal\id_{L^{(-r)}}& =d_F^\frac{1}{2} \left ( \id_{\widehat{\Ocal}^2 }-\id_{\widehat{\Ocal} \oplus p\widehat{\Ocal}} \right ).
\end{align}
This implies
\begin{align}
& E^{(p)}(\tau_1,\tau_2,\psi) = - E(\tau_1,\tau_2, \Fcal{\id}_{L^{(r)}},\psi)-E(\tau_1,\tau_2, \Fcal{\id}_{L^{(-r)}},\psi),
\end{align}
which also means that for the respective constant terms we have
\begin{align}
    c_0(\varphi^{(p)})=-c_0(\id_{L^{(r)}})-c_0(\id_{L^{(-r)}}).
\end{align}
Note that ${\id}_{L^{(r)}}$ vanishes on $l_1$. By Theorem \ref{fouriercoeffs} and the isomorphism between $L^{(r)}$ and $\widehat{\Delta}_0(p)$ in Lemma \ref{latticeisom}, we have
\begin{align}
    E(\tau,\tau, \Fcal{\id}_{L^{(r)}},\psi) = c_0(\id_{L^{( r)}}) + 2 \sum_{n=1}^\8 \left ( \int_{C_{r} \otimes \psi} \Theta_n(v,\id_{\widehat{\Delta}_0(p)}) \right )  e^{2i\pi n \tau}
\end{align}
and similarly for $-r$. Putting the two together we get
\begin{align}
E^{(p)}(\tau,\tau,\psi)& = c_0(\varphi^{(p)})- 2\sum_{n=1}^\8 \left ( \int_{C_{r} \otimes \psi+C_{-r} \otimes \psi} \Theta_n(v,\id_{\widehat{\Delta}_0(p)}) \right ) e^{2i\pi n \tau}.
\end{align}
Since $C_{r} \otimes \psi+C_{-r} \otimes \psi=\Qcalb(\psi) \times \Qcalb(\8,0)$ it follows from Proposition \ref{geodinter} that
\begin{align}
E^{(p)}(\tau,\tau,\psi)& = c_0(\varphi^{(p)})- 4\sum_{n=1}^\8 \bigl < \Qcalb(0,\8), T_n\Qcalb(\psi) \bigr >_{Y_0(p)} e^{2i\pi n \tau}.
\end{align}

\noindent It remains to compute $c_0(\varphi^{(p)})=\zeta_\fin(\varphi^{(p)}_1,\psi^{-1},0)+\zeta_\fin(\varphi^{(p)}_2,\psi,0)$. The Schwartz function $\varphi_\fin^{(p)}$ vanishes on $l_1$ hence the first singular term of $c_0(\varphi^{(p)})$ is zero and we have $c_0(\varphi^{(p)}) = \zeta_\fin(\varphi_2^{(p)},\psi,0)
$, where 
\begin{align}
\varphi^{(p)}_{2,w}(x')=
\begin{cases} \id_{\Ocal_w^\times}(x') & \textrm{if} \, w(p)>0 \\
\id_{\Ocal_w}(x') & \textrm{if} \, w(p)=0.
\end{cases}
\end{align}
By the computation in \eqref{zeta2} this shows that
\begin{align}
\zeta_w(\varphi_2^{(p)},\psi,s)= 
\begin{cases}
1 & \textrm{if} \, w(p)>0, \\
L_{w}(\psi,s) & \textrm{if} \, w(p)=0.
\end{cases}
\end{align}
hence $\zeta_\fin(\varphi_2^{(p)},\psi,0)=L^{(p)}(\psi,0)$.
\end{proof}

\begin{rmk} \label{factor2}
Our formula in Corollary \ref{cordpv} differs by a factor of $2$ from \cite[theorem.~A]{DPV}: the factor $4$ that we obtain in front of the positive Fourier coefficients is a factor $2$ in {\em loc. cit}. We already mentionned that this is due to the absence of the factor $\kappa$ in {\em loc.cit.} but let us make this more precise. 
The difference comes from the different definitions of intersection numbers of geodesic. Let $\Qcalb$ be the (compact) image of the geodesic $\Qcal$ in $Y_0(p)$. The subgroup of $\Gamma$ stabilizing $\Qcal$ is $    \Gamma[\Qcal]=\{ \pm 1 \} {\times \gamma_\Qcal}^\Z
$ for some $\gamma_\Qcal \in \Gamma$. For some $\tau \in \Qcal$ let $[\tau,\gamma_\Qcal \tau)$ be the (half-open) geodesic segment in $\HH$ between $\tau$ and $\gamma_{\Qcal} \tau$. In our case  - see \eqref{intersec3} - the definition of intersection number between $\Qcalb(0,\8)$ and $\Qcalb$ that we use is 
\begin{align} \label{intersection1}
    \bigl < \Qcalb(0,\8), \Qcalb \bigr >_{Y_0(p)} = \sum_{\gamma \in P\Gamma} \bigl < [\gamma \tau,\gamma \gamma_{\Qcal} \tau),\Qcal \bigr >_\HH,
\end{align}
where $P\Gamma$ is the image of $\Gamma$ in $\PSL_2(\Z)$. On the other hand, in \cite{DPV} the intersection numbers are defined by
\begin{align}\label{intersection2}
    \bigl < \Qcalb(0,\8), \Qcalb \bigr >_{Y_0(p)} = \sum_{\gamma \in \Gamma} \bigl < \Qcal(0,\8), [\gamma \tau,\gamma \gamma_{\Qcal} \tau) \bigr >_\HH
\end{align}
which is twice the number in \eqref{intersection1}.
\end{rmk}

\printbibliography
\end{document}

%% file: packages.tex
\usepackage{booktabs} 
\usepackage{array} 
\usepackage{paralist} 
\usepackage{verbatim} 
\usepackage{subfig} 
\usepackage{mathtools}
\usepackage{amssymb}
\usepackage{amsthm}
\usepackage{tikz-cd}
\usepackage{mathrsfs}
\usepackage{enumitem}
\usepackage[colorlinks, allcolors=blue]{hyperref}
\usepackage{appendix}
\usepackage{stmaryrd}
\usepackage{titlesec}
\usepackage{bm}
\usepackage{nicefrac}
\usepackage[backend=bibtex,style=alphabetic,sorting=nyvt]{biblatex}
\usepackage[nottoc,notlof,notlot]{tocbibind} 
\usepackage[titles,subfigure]{tocloft} 
\usepackage{sectsty}
\usepackage{graphicx} 
\usepackage{setspace}
\usepackage{minitoc}

%% file: commands.tex
\newcommand{\8}{\infty}

\newcommand{\Z}{\mathbb{Z}}

\newcommand{\id}{\mathbf{1}}

\newcommand{\tauud}{\underline{\tau}}

\newcommand{\xbf}{\mathbf{x}}
\newcommand{\ebf}{\mathbf{e}}
\newcommand{\fbf}{\mathbf{f}}
\newcommand{\Ebf}{\mathbf{E}}
\newcommand{\Fbf}{\mathbf{F}}

\newcommand{\ybf}{\mathbf{y}}

\newcommand{\Span}{\operatorname{span}}
\newcommand{\SO}{\operatorname{SO}}

\newcommand{\PSL}{\operatorname{PSL}}
\newcommand{\SL}{\operatorname{SL}}
\newcommand{\Sp}{\operatorname{Sp}}

\newcommand{\GL}{\operatorname{GL}}

\newcommand{\gtil}{\tilde{g}}
\newcommand{\htil}{\tilde{h}}
\newcommand{\ktil}{\tilde{k}}
\newcommand{\Htil}{\widetilde{H}}
\newcommand{\Ktil}{\widetilde{K}}

\newcommand{\GSpin}{\operatorname{GSpin}}

\newcommand{\D}{\mathbb{D}}
\newcommand{\HH}{\mathbb{H}}
\newcommand{\R}{\mathbb{R}}
\newcommand{\A}{\mathbb{A}}
\newcommand{\Q}{\mathbb{Q}}
\newcommand{\C}{\mathbb{C}}

\newcommand{\afrak}{\mathfrak{a}}
\newcommand{\bfrak}{\mathfrak{b}}
\newcommand{\dfrak}{\mathfrak{d}}
\newcommand{\p}{\mathfrak{p}}
\newcommand{\m}{\mathfrak{m}}
\newcommand{\q}{\mathfrak{q}}
\newcommand{\g}{\mathfrak{g}}
\newcommand{\cfrak}{\mathfrak{c}}
\newcommand{\fffrak}{\mathfrak{f}}

\newcommand{\kfrak}{\mathfrak{k}}

\newcommand{\pr}{\operatorname{pr}}
\newcommand{\cohom}{H}
\newcommand{\End}{\operatorname{End}}

\newcommand{\vol}{\operatorname{vol}}
\newcommand{\PD}{\operatorname{PD}}

\newcommand{\Hom}{\operatorname{Hom}}

\newcommand{\Tr}{\operatorname{tr}}
\newcommand{\ad}{\operatorname{ad}}

\newcommand{\Res}{\operatorname{Res}}

\newcommand{\N}{\operatorname{N}}
\newcommand{\sgn}{\operatorname{sgn}}

\newcommand{\Ocal}{\mathscr{O}}
\newcommand{\Scal}{\mathscr{S}}
\newcommand{\Fcal}{\mathscr{F}}

\newcommand{\Ucal}{\mathcal{U}}

\newcommand{\Kcal}{\mathcal{K}}
\newcommand{\Lcal}{\mathcal{L}}
\newcommand{\MMcal}{\mathcal{M}}

\newcommand{\Cl}{\textrm{Cl}}
\newcommand{\Pcal}{\mathcal{P}}
\newcommand{\Ical}{\mathcal{I}}
\newcommand{\Zcal}{\mathcal{Z}}
\newcommand{\Qcal}{\mathcal{Q}}
\newcommand{\Qcalb}{\overline{\mathcal{Q}}}

\newcommand{\re}{\operatorname{Re}}
\newcommand{\supp}{\operatorname{supp}}
\newcommand{\Mat}{\operatorname{Mat}}

\newcommand{\fin}{f}

\newcommand*{\transp}[2][-3mu]{\ensuremath{\mskip1mu\prescript{\smash{\mathrm t\mkern#1}}{}{\mathstrut#2}}}%
\newcommand*{\adjast}[2][-3mu]{\ensuremath{\mskip1mu\prescript{\smash{\ast \mkern#1}}{}{\mathstrut#2}}}%
\newcommand*{\adjhash}[2][-3mu]{\ensuremath{\mskip1mu\prescript{\smash{\# \mkern#1}}{}{\mathstrut#2}}}%
\newcommand{\hooklongrightarrow}{\lhook\joinrel\longrightarrow}

\newcommand\restr[2]{{
  \left.\kern-\nulldelimiterspace 
  #1 
  \vphantom{\big|} 
  \right|_{#2} 
  }}
 


%% file: theoremlayout.tex
\newtheoremstyle{mytheoremstyle} 
    {2em}                    
    {2em}                    
    {\itshape}                   
    {}                           
    {\normalsize \bfseries \scshape}                   
    {.}                          
    {0,5em}                       
    {}  
    
\theoremstyle{mytheoremstyle}

\newtheorem{thm}{Theorem}[section]
\newtheorem*{thm*}{Theorem}
\newtheorem{cor}{Corollary}[thm]
\newtheorem*{cor*}{Corollary}
\newtheorem{lem}[thm]{Lemma}
\newtheorem{prop}[thm]{Proposition}

\newtheorem*{que*}{Question}

\theoremstyle{remark}
\newtheorem{rmk}{Remark}[section]
\newtheorem*{ex}{Example}
\numberwithin{equation}{section}

%% file: sample.bib
@phdthesis{branchereau_2022, title={Diagonal restriction and denominators of some Eisenstein cohomology classes}, author={Branchereau, Romain}, year={2022}}

@book{pr,
	Author = {Platonov, Vladimir and Rapinchuk, Andrei},
	Editor = {Academic Press, Inc},
	Publisher = {Elsevier Science},
	Title = {Algebraic Groups and Number Theory},
	Year = {1994}}

@book {iwasawaL,
    AUTHOR = {Iwasawa, Kenkichi},
     TITLE = {Hecke's {$L$}-functions},
    SERIES = {SpringerBriefs in Mathematics},
      NOTE = {Spring, 1964,
              Lectures at Princeton University,
              With a foreword by John Coates and Masato Kurihara},
 PUBLISHER = {Springer, Singapore},
      YEAR = {2019},
       DOI = {10.1007/978-981-13-9495-9},
       URL = {https://doi.org/10.1007/978-981-13-9495-9},
}

@article{funkmil,
author = {Jens Funke and John Millson},
title = {{The geometric theta correspondence for Hilbert modular surfaces}},
volume = {163},
journal = {Duke Mathematical Journal},
number = {1},
publisher = {Duke University Press},
pages = {65 -- 116},
year = {2014},
doi = {10.1215/00127094-2405279},
URL = {https://doi.org/10.1215/00127094-2405279}
}

@misc{rbr,
  doi = {10.48550/ARXIV.2211.10341},
  
  url = {https://arxiv.org/abs/2211.10341},
  
  author = {Branchereau, Romain},
  
  title = {The Kudla-Millson form via the Mathai-Quillen formalism},
  
  publisher = {arXiv},
  
  year = {2022},
  
  copyright = {Creative Commons Attribution 4.0 International}
}

@book{moeglin,
	Author = {Moeglin, C. and Vign{\'e}ras, M.F. and Waldspurger, J.L.},
	Isbn = {9780387186993},
	Lccn = {87032415},
	Publisher = {Springer},
	Series = {Lecture Notes in Mathematics},
	Title = {Correspondances de Howe sur un corps p-adique},
	Year = {1987}}

@article{weil,
	Author = {Andr{\'e} Weil},
	Journal = {Acta Mathematica},
	Number = {none},
	Pages = {143 -- 211},
	Title = {{Sur certains groupes d'op{\'e}rateurs unitaires}},
	Volume = {111},
	Year = {1964}}

@article{kudlasingular,
  title={On the integrals of certain singular theta-functions},
  author={Kudla, S},
  journal={J. Fac. Sci. Univ. Tokyo},
  volume={28},
  pages={439--465},
  year={1982}
}

@article{funke-compositio,
	Author = {Funke, Jens},
	Doi = {10.1023/A:1020002121978},
	Journal = {Compositio Mathematica},
	Number = {3},
	Pages = {289--321},
	Publisher = {London Mathematical Society},
	Title = {Heegner Divisors and Nonholomorphic Modular Forms},
	Year = {2002}}

@book{bump,
	Author = {Bump, Daniel},
	Collection = {Cambridge Studies in Advanced Mathematics},
	Place = {Cambridge},
	Publisher = {Cambridge University Press},
	Series = {Cambridge Studies in Advanced Mathematics},
	Title = {Automorphic Forms and Representations},
	Year = {1997}}

@article{rao,
	Author = {R. Ranga Rao},
	Journal = {Pacific Journal of Mathematics},
	Number = {2},
	Pages = {335 -- 371},
	Title = {{On some explicit formulas in the theory of Weil representation.}},
	Volume = {157},
	Year = {1993}}

@article{sullivan,
	Author = {O'Sullivan, Cormac},
	Journal = {Research in Number Theory},
	Number = {3},
	Pages = {36},
	Title = {Formulas for non-holomorphic Eisenstein series and for the Riemann zeta function at odd integers},
	Volume = {4},
	Year = {2018}}

@article{DPV,
	Author = {Darmon, Henri and Pozzi, Alice and Vonk, Jan},
	Da = {2021/02/01},
	Doi = {10.1007/s00208-020-02086-2},
	Id = {Darmon2021},
	Isbn = {1432-1807},
	Journal = {Mathematische Annalen},
	Number = {1},
	Pages = {503--548},
	Title = {Diagonal restrictions of p-adic Eisenstein families},
	Ty = {JOUR},
	Url = {https://doi.org/10.1007/s00208-020-02086-2},
	Volume = {379},
	Year = {2021}}

@article{KMIHES,
	Author = {Kudla, Stephen S. and Millson, John J.},
	Journal = {Publications Math\'ematiques de l'IH\'ES},
	Language = {en},
	Mrnumber = {92e:11035},
	Pages = {121--172},
	Publisher = {Institut des Hautes \'Etudes Scientifiques},
	Title = {Intersection numbers of cycles on locally symmetric spaces and Fourier coefficients of holomorphic modular forms in several complex variables},
	Url = {http://www.numdam.org/item/PMIHES_1990__71__121_0/},
	Volume = {71},
	Year = {1990},
	Zbl = {0722.11026}}

@conference{sspair,
	Author = {Kudla, Stephen S.},
	Booktitle = {Automorphic forms of several variables},
	Organization = {Taniguchi Symposium},
	Pages = {244-268},
	Publisher = {Birkh{\"a}user Boston},
	Series = {Progress in Mathematics},
	Title = {Seesaw dual reductive pairs},
	Volume = {46},
	Year = {1984}}

@book{lionvergne,
	Author = {Lion, G. and Vergne, M.},
	Edition = {1},
	Isbn = {978-0-8176-3007-2},
	Publisher = {Birkh{\"a}user Basel},
	Series = {Progress in Mathematics},
	Title = {The Weil representation, Maslov index and Theta series},
	Year = {1980}}

@article{km3,
	Author = {Kudla, Stephen S. and Millson, John J.},
	Da = {1987/06/01},
	Doi = {10.1007/BF01457364},
	Id = {Kudla1987},
	Isbn = {1432-1807},
	Journal = {Mathematische Annalen},
	Number = {2},
	Pages = {267--314},
	Title = {The theta correspondence and harmonic forms. II},
	Ty = {JOUR},
	Url = {https://doi.org/10.1007/BF01457364},
	Volume = {277},
	Year = {1987}}

@article{km2,
	Author = {Kudla, Stephen S. and Millson, John J.},
	Da = {1986/09/01},
	Doi = {10.1007/BF01457221},
	Id = {Kudla1986},
	Isbn = {1432-1807},
	Journal = {Mathematische Annalen},
	Number = {3},
	Pages = {353--378},
	Title = {The theta correspondence and harmonic forms. I},
	Ty = {JOUR},
	Url = {https://doi.org/10.1007/BF01457221},
	Volume = {274},
	Year = {1986}}

@article{km,
	Author = {Kudla, Stephen S. and Millson, John J.},
	Da = {1990/12/01},
	Doi = {10.1007/BF02699880},
	Id = {Kudla1990},
	Isbn = {1618-1913},
	Journal = {Publications Math{\'e}matiques de l'Institut des Hautes {\'E}tudes Scientifiques},
	Number = {1},
	Pages = {121--172},
	Title = {Intersection numbers of cycles on locally symmetric spaces and fourier coefficients of holomorphic modular forms in several complex variables},
	Ty = {JOUR},
	Url = {https://doi.org/10.1007/BF02699880},
	Volume = {71},
	Year = {1990}}

@article{mq,
	Author = {Varghese Mathai and Daniel Quillen},
	Doi = {https://doi.org/10.1016/0040-9383(86)90007-8},
	Issn = {0040-9383},
	Journal = {Topology},
	Number = {1},
	Pages = {85 - 110},
	Title = {Superconnections, thom classes, and equivariant differential forms},
	Url = {http://www.sciencedirect.com/science/article/pii/0040938386900078},
	Volume = {25},
	Year = {1986}}

@book{botu,
	Author = {Bott, Raoul and Tu, Loring W},
	Edition = {1},
	Publisher = {Springer-Verlag New York},
	Title = {Differential forms in algebraic topology},
	Volume = {82},
	Year = {1982}}

@article{wielonsky,
	Author = {F. Wielonsky},
	Journal = {l'Enseignement Math{\'e}matique},
	Pages = {93-135},
	Title = {S{\'e}ries d'Eisenstein, int{\'e}grales toro{\"\i}dales et une formule de Hecke},
	Volume = {31},
	Year = {1985}}

@book{neukirch,
	Author = {J{\"u}rgen Neukirch},
	Edition = {1},
	Publisher = {Springer},
	Title = {Algebraic number theory},
	Translator = {Schappacher, N.},
	Year = {1999}}
